\documentclass[11pt,leqno]{amsart}
\usepackage{url}
\usepackage{xspace}
\usepackage{amsfonts,amssymb}
\usepackage{amsthm}
\usepackage[all]{xy}
\usepackage{textcomp}

\usepackage{mathrsfs}
\usepackage{stmaryrd}
\usepackage{bbm}
\usepackage{oldgerm}
\usepackage[T1]{fontenc}
\usepackage[latin1]{inputenc}
\theoremstyle{definition}
\newtheorem{defn}[subsection]{Definition}
\newtheorem{rem}[subsection]{Remark}
\theoremstyle{plain}
\newtheorem{lemma}[subsection]{Lemma}

\newtheorem{prop}[subsection]{Proposition}
\newtheorem{thm}[subsection]{Theorem}
\newtheorem{cor}[subsection]{Corollary}

\theoremstyle{remark}

\numberwithin{equation}{subsection}

\newcommand{\beq}{\begin{equation}}
\newcommand{\eeq}{\end{equation}}

\newcommand{\bem}{\begin{bmatrix}}
\newcommand{\enm}{\end{bmatrix}}

\newcommand{\ra}{\rightarrow}
\newcommand{\lra}{\longrightarrow}
\newcommand{\hra}{\hookrightarrow}

\newcommand{\xra}{\xrightarrow}
\newcommand{\xla}{\xleftarrow}

\newcommand{\Z}{\mathbf{Z}}
\newcommand{\C}{\mathbf{C}}
\newcommand{\F}{\mathbf{F}}
\newcommand{\G}{\mathbf{G}}
\newcommand{\Q}{\mathbf{Q}}
\newcommand{\R}{\mathbf{R}}

\newcommand{\fdelta}{\mathfrak{d}}

\newcommand{\bB}{\mathbb{B}}
\newcommand{\bD}{\mathbb{D}}
\newcommand{\bT}{\mathbb{T}}

\newcommand{\DD}{\mathbf{D}}

\newcommand{\bI}{\mathbf{I}}
\newcommand{\bJ}{\mathbf{J}}
\newcommand{\bJc}{\mathbf{J}_{\Box}}
\newcommand{\bL}{\mathbb{L}}
\newcommand{\X}{\mathbb{X}}
\newcommand{\bP}{\mathbb{P}}
\newcommand{\bA}{\mathbf{A}}
\newcommand{\ee}{\mathbf{e}}
\newcommand{\ff}{\mathbf{f}}

\newcommand{\cA}{\mathcal{A}}
\newcommand{\cB}{\mathcal{B}}
\newcommand{\cC}{\mathcal{C}}

\newcommand{\cF}{\mathscr{F}}
\newcommand{\cG}{\mathcal{G}}
\newcommand{\cH}{\mathcal{H}}
\newcommand{\cL}{\mathcal{L}}
\newcommand{\cM}{\mathcal{M}}
\newcommand{\cO}{\mathcal{O}}
\newcommand{\cP}{\mathcal{P}}
\newcommand{\cU}{\mathcal{U}}
\newcommand{\cV}{\mathcal{V}}

\newcommand{\cW}{\mathcal{W}}
\newcommand{\cY}{\mathcal{Y}}
\newcommand{\cZ}{\mathcal{Z}}

\newcommand{\Ab}{\overline{A}}
\newcommand{\cAb}{\overline{\cA}}
\newcommand{\cBb}{\overline{\cB}}
\newcommand{\cHb}{\overline{\cH}}
\newcommand{\Kb}{\overline{K}}
\newcommand{\Lb}{\overline{L}}

\newcommand{\kappab}{\overline{\kappa}}
\newcommand{\rhob}{\overline{\rho}}

\newcommand{\Pb}{\overline{P}}
\newcommand{\Qb}{\overline{Q}}

\newcommand{\fA}{\mathfrak{A}}
\newcommand{\fa}{\mathfrak{a}}
\newcommand{\fb}{\mathfrak{b}}
\newcommand{\fc}{\mathfrak{c}}
\newcommand{\fp}{\mathfrak{p}}
\newcommand{\fq}{\mathfrak{q}}
\newcommand{\fH}{\mathfrak{H}}
\newcommand{\fJ}{\mathfrak{J}}
\newcommand{\fK}{\mathfrak{K}}
\newcommand{\fM}{\mathfrak{M}}
\newcommand{\fL}{\mathfrak{L}}
\newcommand{\fN}{\mathfrak{N}}
\newcommand{\fS}{\mathfrak{S}}
\newcommand{\fX}{\mathfrak{X}}
\newcommand{\fY}{\mathfrak{Y}}

\newcommand{\Mod}{\mathbf{Mod}_{\fS\otimes \Z_{p^g}}}
\newcommand{\Modd}{\mathbf{Mod}_{\fS}}
\newcommand{\tor}{\mathrm{tors}}
\newcommand{\HW}{\mathrm{HW}}
\newcommand{\univ}{\mathrm{univ}}
\newcommand{\leng}{\mathrm{leng}}
\newcommand{\can}{\mathrm{can}}

\newcommand{\dep}{\mathrm{depth}}
\newcommand{\GL}{\mathrm{GL}}
\newcommand{\MF}{\mathbf{MF}_{S}}
\newcommand{\fr}{\mathrm{fr}}

\newcommand{\Tr}{\mathrm{trace}}
\newcommand{\Frob}{\mathrm{Frob}}

\newcommand{\vk}{\vec{k}}

\newcommand{\res}{\mathrm{Res}}
\newcommand{\tr}{\mathrm{tr}}
\newcommand{\ALG}{\mathrm{ALG}}
\newcommand{\SETS}{\mathrm{SETS}}
\newcommand{\ie}{i.e.}
\newcommand{\et}{\mathrm{\acute{e}t}}
\newcommand{\id}{\mathrm{Id}}
\newcommand{\ord}{\mathrm{ord}}
\newcommand{\rig}{\mathrm{rig}}
\newcommand{\an}{\mathrm{an}}

\newcommand{\rM}{\mathrm{M}}
\newcommand{\red}{\mathrm{red}}
\newcommand{\Fil}{\mathrm{Fil}}
\newcommand{\one}{\vec{\mathbf{1}}}

\newcommand{\omegab}{\underline{\omega}}

\newcommand{\m}{\mathfrak{m}}
\newcommand{\spe}{\mathrm{sp}}

\DeclareMathOperator{\End}{\mathrm{End}}
\DeclareMathOperator{\Lie}{\mathrm{Lie}}
\DeclareMathOperator{\Spec}{\mathrm{Spec}}
\DeclareMathOperator{\Hom}{\mathrm{Hom}}
\DeclareMathOperator{\Spf}{\mathrm{Spf}}
\DeclareMathOperator{\Ker}{\mathrm{Ker}}
\DeclareMathOperator{\im}{\mathrm{Im}}

\begin{document}

\title[Classicality of overconvergent Hilbert eigenforms]{Classicality  of Overconvergent Hilbert Eigenforms: case of quadratic residue degrees}
\author{Yichao TIAN}
\address{Mathematics Department, Fine Hall, Washington Road, Princeton, NJ, 08544, USA}
\email{yichaot@princeton.edu}

\begin{abstract}
Let $F$ be a quadratic real field,  $p$ be a rational prime inert in $F$. In this paper, we prove that an overconvergent $p$-adic Hilbert eigenform for $F$ of small slope is actually a classical Hilbert modular form.
\end{abstract}
\maketitle

\section{Introduction}
\subsection{} We fix a prime number $p>0$. A famous theorem of Coleman says that an overconvergent $p$-adic (elliptic) modular eigenform of small slope is  actually classical. More precisely, let $N\geq 5$ be an integer coprime to $p$, and $X_1(N)^{\an}$ be the rigid analytification of the  usual modular curve of level $\Gamma_1(N)$ over $\Q_p$. We denote by $X_1(N)^{\an}_{\ord}$ the ordinary locus of $X_1(N)^{\an}$.  For $p\geq 5$, $X_1(N)^{\an}_{\ord}$ is simply the locus where $E_{p-1}$ (Eisenstein series of weight $p-1$), the standard lift of the Hasse invariant, has non-zero reduction modulo $p$. For any integer $k\in \Z$, Katz \cite{Ka72} defined the space $\rM^{\dagger}_k(\Gamma_1(N))$ of overconvergent $p$-adic modular forms of weight  $k$.  An element  in  $\rM^{\dagger}_k(\Gamma_1(N))$ is a section of the modular line bundle  $\omegab^{k}$ defined over a strict neighborhood of $X_1(N)^{\an}_{\ord}$ in $X_1(N)^{\an}$. Moreover, Katz also defined a completely continuous operator $U_p$ acting on $\rM^{\dagger}(\Gamma_1(N))$. There is a natural injection from  $\rM_k(\Gamma_1(N)\cap \Gamma_0(p))$ to $\rM^{\dagger}_k(\Gamma_1(N))$, where $\rM_k(\Gamma_1(N)\cap\Gamma_0(p))$ is the space of classical modular forms of weight $k$ and level $\Gamma_1(N)\cap\Gamma_0(p)$, that is, sections of $\omegab^{k}$ over the modular curve $X(\Gamma_1(N)\cap\Gamma_0(p))$.   In \cite{col}, Coleman proved that if $f\in \rM^{\dagger}_k(\Gamma_1(N))$  is a  $U_p$-eigenvector with eigenvalue $a_p$ and $v_p(a_p)<k-1$, then $f$ actually lies in $\rM_k(\Gamma_1(N)\cap \Gamma_0(p))$. Coleman's original proof for this deep result was achieved by an ingenious dimension counting argument. Later on, Buzzard \cite{Bu} and Kassaei \cite{Ks} reproved Coleman's theorem by  an elegant  analytic continuation process. The basic idea of Buzzard-Kassaei was to extend successively the section $f$ by the functional equation $f=\frac{1}{a_p}U_p(f)$ to the entire rigid analytic space $X(\Gamma_1(N)\cap\Gamma_0(p))^{\an}$. Actually, Buzzard proved that $f$ can be extended to the union of ordinary locus and the area  with  supersingular reduction of $X(\Gamma_1(N)\cap\Gamma_0(p))^{\an}$.  Then Kassaei constructed another form $g$ on the complementary to Buzzard's area, and showed that $f$ and $g$ glue together to an analytic section of $\omegab^{k}$ over $X(\Gamma_1(N)\cap\Gamma_0(p))^{\an}$. The  rigid GAGA theorem \cite[7.6.11]{Ab} then implies that  this is indeed a genuine section of $\omegab^k$ over the algebraic modular curve $X(\Gamma_1(N)\cap \Gamma_0(p))$. In this process, the theory of canonical subgroups for elliptic curves developed in \cite{Ka72} due to Lubin and Katz plays a fundamental role.

There have been many efforts in generalizing the classical theory  on overconvergent $p$-adic modular forms
 to other situations.
 First of all, to generalize overconvergent $p$-modular forms and  the $U_p$-operator,
 we need to construct canonical subgroups in more general context. This has been done by many authors.
 For instance,  \cite{KL}, \cite{GK} consider the Hilbert case, and \cite{AM}, \cite{AG} treat the general case for abelian varieties,
 and finally in \cite{Ti}, \cite{Fa09}, \cite{Rab} and \cite{Ha} the canonical subgroups are constructed for general $p$-divisible groups. Using the canonical subgroups, overconvergent $p$-adic modular forms and the $U_p$-operators can
  be constructed similarly in various settings. However, the generalization of Coleman's classicality criterion need more hard work. As far as I know, this criterion has been generalized in the following cases. In \cite{col2}, Coleman himself generalized his results to modular forms of higher level at $p$. Kassaei considered in \cite{Ks09} the case of modular forms defined over various Shimura curves. In \cite{Sa}, Sasaki generalized it to the case of Hilbert eigenforms when $p$ totally splits in the totally real field defining the Hilbert-Blumenthal modular variety. Finally, Pilloni proved in \cite{Pi} the classicality criterion for  overconvergent Siegel modular forms of genus $2$.
 In this paper, we will follow the idea of Buzzard-Kassaei to study overconvergent Hilbert modular forms in the quadratic inert case.

\subsection{} To simplify the notation, let's describe our result in a special but essential case. Let $F$ be a real quadratic number field in which $p$ is inert, and $\cO_F$ be its ring of integers. We put $\kappa\simeq \F_{p^2}$,  $W=\cO_{F_p}$ and $\Q_{\kappa}=W[1/p]$. We denote by $\bB=\{\beta_1,\beta_2\}$ the two embeddings of $F$ into $\Q_{\kappa}$.  Let $N\geq 4$ be an  integer coprime to $p$.  We consider the Hilbert-Blumenthal modular variety $X$ over $\Spec(W)$ that classifies  prime-to-$p$ polarized abelian schemes $A$ with real multiplication by $\cO_F$ of level $\Gamma_{00}(N)$. Let $Y$ be the moduli space that classifies the same data and together with an $(\cO_F/p)$-cyclic subgroup of $A[p]$.  For each pair of integers $\vk=(k_1,k_2)\in \Z^{\bB}$, we have the modular line bundle $\omegab^{\vk}$ over $X$ and $Y$ (See \ref{mod-line} for its precise definition). For each finite extension $L$ of $\Q_{\kappa}$, we put $\rM_{\vk}(\Gamma_{00}(N)\cap\Gamma_0(p),L)=H^0(Y_{L},\omegab^{\vk})$, and call it the space of (geometric) Hilbert modular forms of level $\Gamma_{00}(N)\cap \Gamma_0(p)$ and weight $\vk$ with coefficients in $L$. This is a finite dimensional vector space over $L$ by classical Koecher principle, and the theory of arithmetic compactifications of Hilbert-Blumenthal modular varieties \cite{Rap,Ch,DP} implies that it actually descends to a finite flat $\Z[1/N]$-module.

 Let $\fX$ and $\fY$ be respectively the completion of $X$ and $Y$ along their special fibers, and $\fX_{\rig}$ and $\fY_{\rig}$ be their rigid analytic generic fibers \`a la Raynaud \cite[Ch. 4]{Ab}. We still have a natural forgetful map $\pi:\fY_{\rig}\ra \fX_{\rig}$. For each $\vk\in \Z^{\bB}$, we denote still by $\omegab^{\vk}$ the rigidification of the line bundle $\omegab^{\vk}$. Let $\fX^{\ord}_{\rig}$ be the ordinary locus of $\fX_{\rig}$, i.e. the  locus where the universal rigid Hilbert-Blumenthal abelian variety $\fA_{\rig}$ over $\fX_{\rig}$ has good ordinary reduction. Then the multiplicative part of the universal finite flat group  scheme $\fA_{\rig}[p]$ defines a section $s^{\circ}:\fX^{\ord}_{\rig}\ra \fY_{\rig}$ of the projection $\pi$ over $\fX^{\ord}_{\rig}$. We denote by $\fY_{\rig}^{\ord}$ the image of $s^{\circ}$, so that $\pi|_{\fY^{\ord}_{\rig}}:\fY_{\rig}^{\ord}\xra{\sim} \fX_{\rig}^{\ord}$ is an isomorphism of quasi-compact rigid analytic spaces. For a finite extension $L$ of $\Q_{\kappa}$,  Kisin and Lai  \cite{KL} defined an overconvergent Hilbert modular form of level $\Gamma_{00}(N)$ and weight $\vk$ with coefficients in $L$ to be a section of $\omegab^{\vk}$ over $\fX^{\ord}_{\rig}$ that extends to a strict neighborhood of $\fX^{\ord}_{\rig,L}$. We denote by $\rM^{\dagger}_{\vk}(\Gamma_{00}(N),L)$ the space of such forms. This is a  direct limit of infinite dimensional Banach spaces over $L$. Moreover, the theory of canonical subgroups for Hilbert modular varieties  allows them to define a completely continuous $U_p$-operator on $\rM^{\dagger}_{\vk}(\Gamma_{00}(N),L)$. Note that a relatively weak formulation of the existence of canonical subgroups says that the section $s^{\circ}:\fX^{\ord}_{\rig}\ra \fY_{\rig}^{\ord}$ extends to a strict neighborhood of $\fX_{\rig}^{\ord}$, or equivalently the isomorphism $\pi|_{\fY^{\ord}_{\rig}}$ extends to a strict neighborhood. Therefore, there exists a natural injection
 $$\rM_{\vk}(\Gamma_{00}(N)\cap\Gamma_0(p),L)\ra \rM^{\dagger}_{\vk}(\Gamma_{00}(N),L).$$
 We say an element $f$ in $\rM^{\dagger}_{\vk}(\Gamma_{00}(N),L)$ is classical if it lies in the image of this injection. The main result of this paper is the following
 \begin{thm}\label{thm-quad}
 Let $f\in \rM^{\dagger}_{\vk}(\Gamma_{00}(N),L)$ be a $U_p$-eigenvector with eigenvalue $a_p$. If $v_p(a_p)< \min\{k_1,k_2\}-2$, then $f$ is classical.
 \end{thm}


 Actually, we prove our main Theorem  in a slightly general setting \ref{thm-main}. Note that our results imply that, in the quadratic inert case, the classical points are Zariski dense in the eigencurve  for overconvergent Hilbert modular forms of level $\Gamma_{00}(N)$ constructed in \cite{KL} (See Theorem \ref{thm-density}).

   Let's indicate  the ideas of the proof. First, by rigid GAGA and a rigid version of Koecher principle (Prop. \ref{prop-koecher}), we just need to extend $f$  analytically to the entire analytic space $\fY_{\rig}$.  To achieve this, the key point is to understand the dynamics of Hecke correspondence $U_p$ on $\fY_{\rig}$ \eqref{sect-rig-U_p}. We will use extensively the  work of Goren and Kassaei \cite{GK}.
Three ingredients from their work will be important for us. The first one is the stratification on the special fiber $Y_{\kappa}$ defined by them; the second is their valuation on $\fY_{\rig}$  via local parameters; and the third one is the so-called ``Key Lemma'' \cite[2.8.1]{GK}, which relates the partial Hasse invariants with the certain local parameters of $Y_{\kappa}$. In this paper, we will interpret their valuation on $\fY_{\rig}$ in terms of partial degrees (cf. \ref{sect-mult-deg}, \ref{prop-val}). They are natural refinements in the real multiplication case of the usual degree function, which has been introduced by Fargues \cite{Fa} and applied by Pilloni \cite{Pi} to  analytic continuation of $p$-adic Siegel modular forms. Actually, our work originates from an effort to understand the geometric meaning of Goren-Kassaei's valuation. Compared with the totally split case, our difficulty comes from the fact that the $p$-divisible group associated with a Hilbert-Blumenthal abelian variety (HBAV) with RM by $\cO_F$ is a genuine $p$-divisible group of dimension $2$, so its group law can not be explicitly described by one-variable power series. We overcome this  by using  Breuil-Kisin modules to compute the partial degrees of the $p$-torsion of a HBAV. This approach is motivated by the recent work of Hattori \cite{Ha}. These local computations via Breuil-Kisin modules combined with Goren-Kassaei's ``Key Lemma'' will give us enough information to understand the dynamics of the Hecke correspondence $U_p$ except the case mentioned in Prop. \ref{prop-pilloni} or \ref{lem-V_n}. In this exceptional case, we have to study  in detail the local moduli of deformations of  a superspecial HBAV. This is achieved in Appendix B by using Zink's theory on Dieudonn\'e windows \cite{zink}. Finally, we can prove that the form $f$ extends to an admissible open subset of $\fY_{\rig}$ that contains the tube over the complementary to the codimension $2$ stratum in Goren-Kassaei's stratification on $Y_{\kappa}$ ( Prop.  \ref{prop-cont-1}). By a useful trick invented by Pilloni in \cite[\S7]{Pi}, this allows us to conclude that $f$ extends indeed to the entire $\fY_{\rig}$ (Prop. \ref{prop-exten}).

 \subsection{} This paper is organized as follows. In Section 2, we review the facts that we need on the Hilbert-Blumenthal modular varieties and state the main theorem \ref{thm-main} and its consequence on the density of classical points in the eigencurves for overconvergent Hilbert modular forms. In Section 3, we perform the computations mentioned above on the $(\cO_F/p)$-cyclic subgroups of a HBAV over a complete discrete valuation ring via Breuil-Kisin modules. In particular, we give an alternative proof (Thm. \ref{thm-can}) for the existence of canonical subgroups in the Hilbert case proven in \cite{GK}. Section 4 is mainly dedicated to the review on Goren-Kassaei's work, and we provide also another proof of their ``Key Lemma'' using Dieudonn\'e theory (Prop. \ref{prop-key-lemma}). Section 5 is the heart of this work, and it contains a complete proof of Theorem \ref{thm-quad}. Finally, we prove our general main theorem \ref{thm-main} in Section 6. The proof of the general case is a combination of the split case treated by Sasaki \cite{Sa} and the case in Section 5. We have two appendices. In the first one, we gather some general results on the extension and gluing of sections in rigid geometry. In Appendix B, we study the local deformation space of a superspecial $p$-divisible group with formal real multiplication by $\Z_{p^g}$, where $g\geq 1$ is an integer and $\Z_{p^g}$ is the ring of integers of the unramified extension of $\Q_{p}$ of degree $g$. As a by-product, we see that the local moduli admits some canonical choices of local parameters $T_1,\cdots,T_g$ such that the $p$-divisible groups corresponding to $T_i=0$ admits formal complex multiplication by $\Z_{p^{2g}}$ or $\Z_{p^g}\oplus\Z_{p^g}$ according to the parity of  $g$ (cf. Remark \ref{rem-CM}). These $p$-divisible groups (or those isogenous to them) seem to deserve more study, and  should be considered as  the canonical lifting (or quasi-canonical lifting) of the superspecial $p$-divisible group in the formal real multiplication case. We hope that we can  return to  the problem in the future.

 \subsection{} After I finished a preliminary version of this paper and distributed it among a small circle, V. Pilloni showed me a draft of his joint work \cite{PS} with B. Stroh, where similar results were obtained independently.  The  influence of the works \cite{Ks}, \cite{GK} and \cite{Pi} on this work will be obvious for the reader. I express my hearty gratitude to their authors. I am especially grateful to Christophe Breuil for his careful reading of a preliminary version of this paper, and for his valuable suggestions. Finally, I also would like to think  Ahmed Abbes and Liang Xiao for their encouragements and interest in this work,   Kaiwen Lan and Tong Liu for helpful discussions during the preparation of this paper.

\subsection{Notation}\label{notation} Let $F$ be a totally real number field with $g=[F:\Q]> 1$, $\cO_F$ be its ring of integers,  $\fdelta_F$  the different of $F$. Let $p$ be a fixed prime number unramified in $F$. For a prime ideal $\fp$ of $\cO_L$ above $p$, we put  $\kappa(\fp)=\cO_L/\fp$ and denote by $|\kappa(\fp)|=p^{f_{\fp}}$ the cardinality of $\kappa(\fp)$.  Let $f$ be the l.c.m. of all $f_{\fp}$ with $\fp|p$, and
$\kappa$ be the finite field with $p^f$ elements, $W=W(\kappa)$ be the ring of Witt vectors with coefficients in $\kappa$ and $\Q_{\kappa}=W[1/p]$. Let $\bB$ be the set of embeddings of $F$ into $\Q_{\kappa}$. For each prime ideal $\fp$ of $\cO_F$ dividing $p$, let $\bB_{\fp}\subset \bB$ be the subset consisting of embeddings $\beta$ such that $\beta^{-1}(pW)=\fp$.
So we have $\bB=\coprod_{\fp|p}\bB_{\fp}$. If $\sigma$ denotes the Frobenius on $\Q_{\kappa}$, then $\beta\mapsto \sigma\circ \beta$ defines a natural cyclic action of Frobenius on each $\bB_{\fp}$.

In general, for a finite set $I$, we denote by $|I|$ its cardinality.

Let $\C_p$ be the completion of an algebraic closure of $\Q_{\kappa}$. All the finite extensions of $\Q_{\kappa}$ are understood to be subfields of $\C_p$. We denote by $v_p$ the $p$-adic valuation on $\C_p$, and by $|\cdot|_p:\C_p^{\times}\ra \R_{>0}$ the non-archimedean absolute value  $|x|_p=p^{-v_p(x)}$.

\section{Hilbert modular varieties, Hilbert modular forms and the Statement of the main theorem}
\subsection{} Let $S$ be a scheme. A Hilbert-Blumenthal abelian variety by $\cO_F$ (or a  HBAV for short) over $S$ is an abelian scheme $A$ over $S$ equipped with an embedding of rings $\iota: \cO_F\hra \End_{S}(A)$ such that $\Lie(A)$ is an $\cO_S\otimes \cO_F$-module locally free of rank 1. If $A$ is a HBAV over $S$, the dual of $A$, denoted by $A^\vee$, has a canonical structure of HBAV over $S$. We denote by $\cP(A)$ the fppf-sheaf over $S$ of symmetric $\cO_L$-linear homomorphisms of abelian schemes $A\ra A^\vee$, and by $\cP(A)^+
\subset \cP(A)$ the cone consisting of symmetric polarizations.

We fix a positive integer $N\geq 4$ coprime to $p$. Let $\fc$ be a fractional ideal of $F$ prime to $p$, and $\fc^+\subset \fc$ be the cone of totally positive elements. Consider the functor
\[\cF_{\fc}:\;\; \ALG_{W}\lra \SETS\]
which associates to each $W$-algebra $R$ the set of isomorphism classes of triples $(A/R, \lambda, \psi_N)$ where:

\begin{itemize}
\item $A$ is a HBAV over $\Spec(R)$;

\item $\lambda$ is a $\fc$-polarization of $A$, \ie,   an $\cO_F$-linear homomorphism $\lambda: \fc\ra \cP(A)$ sending $\fc^{+}$ to $\cP(A)^{+}$ such that the induced map of fppf-sheaves on $\Spec(R)$
\[A\otimes_{\cO_F} \fc \xra{1_A\otimes \lambda} A\otimes_{\cO_F} \cP(A)\lra A^\vee\colon \; a\otimes x \mapsto a\otimes\lambda(x)\mapsto \lambda(x)(a) \]
 is an isomorphism.

\item $\psi_N$ is an embedding of abelian fppf-sheaves of $\cO_F$-modules $\mu_N\otimes\fdelta^{-1}_{F}\hra A[N]$.
\end{itemize}
It's well known that this functor is representable by a smooth and quasi-projective scheme $X_{\fc}$ over $\Spec(W)$ of relative dimension $g$, which we usually call the $\fc$-Hilbert modular variety over $W$ of level $\Gamma_{00}(N)$ \cite[Ch. 4, \S3.1]{Go}. By a result of Ribet, the fibers of $X_{\fc}$ are geometrically irreducible \cite[Ch. 3 \S6.3]{Go}.

\subsection{} Let $R$ be a $W$-algebra, and  $(A/R,\lambda, \psi_N)$ be an object in $X(R)$.  An \emph{isotropic $(\cO_F/p)$-cyclic subgroup} $H$ of $A$ is a closed subgroup scheme  $H\subset A[p]$ which is stable under $\cO_F$, free of rank 1 over $\cO_F/p$ as abelian fppf-sheaf over $\Spec(R)$, and isotropic under the $\gamma$-Weil pairing
\[A[p]\times A[p]\xra{1\times \gamma}A[p]\times A^\vee[p]\ra \mu_p\]
induced by a $\gamma\in \cP(A)^{+}$ of degree prime to $p$. So when $A$ is defined over a perfect field $k$ of characteristic $p$, a subgroup $H\subset A[p]$ is $(\cO_F/p)$-cyclic if and only if its Dieudonn\'e module is a free $(k\otimes\cO_F)$-module of rank $1$.
  Let $\cG_{\fc}$ be the functor which associates to each $W$-algebra $R$ the  set of isomorphism classes of $4$-tuples $(A/R, \lambda, \psi_N,H)$ where $(A/R, \lambda, \psi_N)$ is an object in $\cF_{\fc}(R)$, and $H\subset A[p]$ is an isotropic  $(\cO_L/p)$-cyclic subgroup of $A$.
The functor $\cG_{\fc}$ is representable by a scheme $Y_{\fc}$ over $\Spec(W)$.
We call $Y$ the $\fc$-Hilbert modular variety of level $\Gamma_{00}(N)\cap\Gamma_0(p)$. The natural forgetful map $(A/R, \lambda, \psi,H)\mapsto (A/R,\lambda, \psi)$ defines a morphism of $W$-schemes $\pi:Y_{\fc}\ra X_{\fc}$, which is finite \'etale of degree $\prod_{\fp|p}(p^{f_{\fp}}+1)$ on the generic fibers  over $\Q_{\kappa}$.

Note that, for an object $(A/R, \lambda_A, \psi_{A,N},H)$ in $Y_{\fc}(R)$, the quotient $B=A/H$ is naturally equipped with a structure of HBAV.  Let $f: A\ra B$ be the natural isogeny and $f^t:B\ra A$ be the unique isogeny such that $f\circ f^t= p \cdot 1_B$ and $f^t\circ f=p \cdot 1_A$. If $\lambda: A\otimes_{\cO_F} \fc\xra{\sim}A^\vee$ is the isomorphism given by $\lambda_A: \fc\ra \cP(A)$, we define a $\fc$-polarization on $B$ by $\lambda_B=\frac{1}{p}(f^t)^*\circ \lambda_A: \fc\ra \cP(A)\ra \cP(B)$, where $(f^t)^*: \cP(A)\ra \cP(B)$ is given by $\phi\mapsto (f^t)^\vee\circ \phi\circ f^t$. Finally, since $H$ has order prime to $N$, the isogeny $f:A\ra B$ induces an isomorphism  $f:A[N]\xra{\sim}B[N]$. We define  $\psi_{B,N}$ as $\mu_N\otimes\fdelta_F^{-1} \xra{\psi_{A,N}}A[N]\xra{\sim} B[N]$. We get an object $(B/R, \lambda_B,\psi_{B,N})$ in $X_{\fc}(R)$. 

Fix a finite set $\{\fc_1,\cdots, \fc_{h^+}\}$ of fractional ideals of $F$ prime to $p$, which form a set of representatives for the narrow class group $\mathrm{Cl}^{+}_F$ of $F$. We put
\[X=\coprod_{i=1}^{h^+}X_{\fc_i}\quad\text{and} \quad Y=\coprod_{i=1}^{h^+}Y_{\fc_i}.\]
We call $X$ (resp. $Y$) the Hilbert modular varities of level $\Gamma_{00}(N)$ (resp. of level $\Gamma_{00}(N)\cap \Gamma_0(p)$).  In the sequel, an object $(A/R, \lambda, \psi_N, H)$ in $Y(R)$ will be usually omitted as $(A/R, H)$, or even as $(A,H)$, if there is no confusions on the polarization $\lambda$ and the level structure $\psi_N$.

\subsection{}\label{mod-line}  Let $T$ be the algebraic group $(\res_{\cO_F/\Z}\G_m)\otimes_{\Z} W$ over $W$, and $\X(T)$ be the group of characters of $T$. For any $\beta\in \bB=\mathrm{Emd}_{\Q}(F,\Q_{\kappa})$, let  $\chi_\beta\in \X(T)$ be the character
\[T(R)=(R\otimes \cO_F)^\times \ra R^\times=\G_m(R)\quad \text{given by} \quad r\otimes a\mapsto r\beta(a).\]
Then $(\chi_\beta)_{\beta\in \bB}$ form a basis of $\X(T)\xra{\sim} \Z^{\bB}$ over $\Z$. For an element $(k_{\beta})_{\beta\in \bB}=\sum_{\beta\in\bB}k_{\beta}\beta\in \Z^{\bB}$, we denote by  $\chi_{\vk}= \prod_{\beta\in\bB}\chi_{\beta}^{k_\beta}$ the corresponding character of $T$.

Let $\cA\ra X$ be the universal HBAV over $X$, and $\omegab=e^*\Omega^1_{\cA/X}$ where $e: X\ra \cA$ is the unit section of $\cA$. This is a locally free $\cO_X\otimes \cO_F$-module of rank 1, and we have
\[\omegab=\bigoplus_{\beta\in\bB}\omegab_{\beta},\]
where $\omegab_{\beta}$ is the submodule of $\omegab_{\cA/X}$ where $\cO_F$ acts via $\chi_\beta$. For any character $\vk=(k_\beta)_{\beta\in\bB}\in\Z^{\bB}$, we define a line bundle $$\omegab^{\vk}=\bigotimes_{\beta\in \bB}\omegab_{\beta}^{\otimes k_\beta}.$$ 
 By abuse of notation, we denote still by $\omegab^{\vk}$ its pull-backs over  $Y$ via $\pi^*$.

\begin{defn}\label{defn-Hilb}For a $W$-algebra $R_0$, we call the elements of $H^0(X\otimes R_0, \omegab^{\vk} )$ (resp. $H^0(Y\otimes R_0, \omegab^{\vk} )$) (geometric) Hilbert modular forms with coefficients in $R_0$ of weight $\vk=\sum_{\beta}k_\beta\cdot \beta$ and level $\Gamma_{00}(N)$ (resp. of level $\Gamma_{00}(N)\cap\Gamma_0(p)$) over  $R_0$.
\end{defn}

We have the following modular interpretation of (geometric) Hilbert modular forms \`a la Katz. For each $R_0$-algebra $R$, we consider  $5$-tuples $(A/R, \lambda, \psi_N, H, \omega)$, where  $(A/R, \lambda, \psi_N, H)$ is an element  in $Y(R)$ and $\omega$ is a generator of $\omegab_{A/R}=e^*\Omega^1_{A/R}$ as an $R\otimes \cO_F$-module. Then a Hilbert modular form $f$  of level $\Gamma_{00}(N)\cap \Gamma_0(p)$ of weight $\vk$ over $R_0$ is equivalent to a rule that assigns to each $R_0$-algebra $R$, each $5$-tuple $(A/R,\lambda, \psi_N, H, \omega)$ as above, an element $f(A/R,H,\omega)\in R$ satisfying the following properties:
\begin{itemize}
\item $f(A/R, H, a\cdot \omega)=\chi_{\vk}(a)^{-1}f(A/R,H,\omega)$ for $a\in (R\otimes \cO_F)^\times$;

\item if $\phi:R\ra R'$ is a homomorphism of $R_0$-algebras and $(A'/R',\lambda',\psi_N',H',\omega')$ is the base change to $R'$ of $(A/R,\lambda, \psi_N,H,\omega)$, then  $f(A'/R',H',\omega')=\phi(f(A/R,H,\omega))$.
\end{itemize}

We have a similar description of Hilbert modular forms of level $\Gamma_{00}(N)$ over $R_0$, and we leave the details to the reader.

\subsection{} We just recall some well known facts that we need on the minimal compactifications of Hilbert-Blumenthal moduli spaces. For more information, the reader may consult \cite{Rap, Ch, KL}. Let $(\fc,\fc^{+})$ be a prime-to-$p$ fractional ideal of $\cO_F$. A $\Gamma_{00}(N)$-cusp $C$ of $X_{\fc}$ consists of the following data:
\begin{enumerate}
\item Projective $\cO_F$-modules $\fa$ and $\fb$ of rank $1$.

\item An isomorphism of $\cO_F$-modules $\fb^{-1}\fa\xra{\sim}\fc$.

\item An exact sequence of projective $\cO_F$-modules
\[0\ra \fdelta_{F}^{-1}\fa^{-1}\ra L\ra \fb\ra 0.\]

\item An embedding of abelian fppf-sheaves of $\cO_{F}$-modules over $\Spec(W)$
$$i_{C}:\mu_{N}\otimes\fdelta_{F}^{-1}\hra L/NL.$$
\end{enumerate}
Set $M=\fa\fb=\fb^2\fc$, and $M^*=\Hom_{\Z}(M,\Z)\simeq \fa^{-1}\fb^{-1}\fdelta^{-1}_F=\fdelta_F^{-1}\fa^{-2}\fc$. The positivity on $\fc$ and that on $\fdelta_F$ induce natural positivities on $M$ and $M^*$.
For each  $\Gamma_{00}(N)$-cusp $C$, we choose a rational polyhedral cone decomposition of $M^{*+}\cup\{0\}$. We can construct a toroidal compactification  $\overline{X}_{\fc}$ of $X_{\fc}$ as in \cite{KL}. This is a proper smooth scheme over $W$ that contains $X_{\fc}$ as a Zariski open dense subset, and the boundary $\overline{X}_{\fc}-X_{\fc}$ is a relative normal crossing divisor in $\overline{X}_{\fc}$ over $\Spec(W)$. There exists a semi-abelian scheme $\cAb$ with real multiplication by $\cO_F$ over $\overline{X}_{\fc}$ which extends the universal HBAV $\cA$ over $X_{\fc}$.   Finally, we can construct the minimal compactification $X_{\fc}^*$ of $X$ from $\overline{X}_{\fc}$ \cite[1.8]{KL}. The scheme $X^{*}_{\fc}$  is normal and projective over $\Spec(W)$ with normal fibers, and it contains also $X_{\fc}$ as a Zariski open dense subset. There exists a canonical map $\overline{X}_{\fc}\ra X^{*}_{\fc}$ such that the following diagram
\[\xymatrix{&\overline{X}_{\fc}\ar[d]\\
X_{\fc}\ar@{^(->}[ru]\ar@{^(->}[r]&X^{*}_{\fc}
}\]
is commutative. Even though the toroidal compactification $\overline{X}_{\fc}$ depends on the choice of the polyhedral decompositions, the minimal compactification $X^*_{\fc}$ is canonical. The boundary $X^*_{\fc}-X_{\fc}$ consists of finitely many copies of $\Spec(W)$ indexed by the $\Gamma_{00}(N)$-cusps of $X_{\fc}$. We put
\[X^*=\coprod_{(\fc,\fc^+)}X_{\fc},\]
where $(\fc,\fc^+)$ runs through a set of prime-to-$p$ representatives for the strict ideal class group of $\cO_F$. The $(\cO_{X}\otimes \cO_F)^{\times}$ torsor $\omegab$ extends to $X^{*}$, hence, for any $\vk\in \Z^{\bB}$, the line bundle $\omegab^{\vk}$ extends uniquely to a line bundle, denoted still by $\omegab^{\vk}$, on $X^{*}$. Moreover, the line bundle $\omegab^{\vec{1}}$ is ample on $X^*$, where $\vec{1}=(1,\cdots,1)\in \Z^{\bB}$ \cite[Thm. 4.3 (ix)]{Ch}.

 Following \cite[4.5.2]{Ch},  we put $Y^*$ to be the normalization of $Y$ over $X^*$. This is a proper normal scheme over $\Spec(W)$. By a result of \cite{St} (cited as \eqref{equ-loc-ring}), the scheme $Y$ is normal. Hence, $Y$ is naturally a  Zariski open dense subset of $Y^*$. We call $Y^*$ the minimal compactification of $Y$. The boundary $Y^*-Y$ has a similar interpretation in terms of cusps. We define a $\Gamma_{00}(N)\cap\Gamma_0(p)$-cusp $(C,H)$ on $Y$ to be a $\Gamma_{00}(N)$ cusp $C=(\fa,\fb,L,i_C)$ together with an $(\cO_F/p)$-cyclic subgroup $H\subset L/pL$. We say a $\Gamma_{00}(N)\cap\Gamma_0(p)$-cusp $(C,H)$ is unramified if $H=\fa^{-1}\fdelta_{F}^{-1}/p\fa^{-1}\fdelta_{F}^{-1}\subset L/pL$. It's clear that there are $\prod_{\fp|p}(p^{f_{\fp}}+1)$ $\Gamma_{00}(N)\cap\Gamma_0(p)$-cusps above a $\Gamma_{00}(N)$-cusp, and one of them is unramified.  The boundary $Y^*-Y$ is a finite flat scheme over $\Spec(W)$, whose generic geometric fibers are indexed by the $\Gamma_{00}(N)\cap\Gamma_0(p)$-cusps. Note that since unramified $\Gamma_{00}\cap\Gamma_0(p)$-cusps are defined over $\Spec(W)$, each of them corresponds to a connected component of $Y^*-Y$ isomorphic to $\Spec(W)$. However, in general, several $\Gamma_{00}(N)\cap\Gamma_0(p)$-cusps on the generic fibers may specialize to the same point on the special fiber of $Y^*$. Another important fact for us is that the local ring of the special fiber of $Y^*$ at the specialization of a cusp is normal \cite[Proof of 7.1.1]{GK}. As usual, for $\vk\in \omegab^{\vk}$, we still denote by $\omegab^{\vk}$ the pull-back of the line bundle $\omegab^{\vk}$ on $X^{*}$. Note that $\omegab^{\vec{1}}$ is just ample on the generic fiber $Y^*_{\Q_{\kappa}}$, but not  on the special fiber $Y^*_{\kappa}$, because the canonical projection $Y^*_{\kappa}\ra X_{\kappa}^*$ is not quasi-finite.

\subsection{} Let $\fX$ and $\fY$ be the respectively the formal completions of $X$ and  $Y$ along their special fibers.  The formal scheme $\fX$ represents the functor that attaches, to each admissible $p$-adic formal scheme over $\Spf(W)$, the set of polarized HBAV with a $\Gamma_{00}(N)$-level structure; and we have a similar interpretation for $\fY$. Let $\fX_{\rig}$ and $\fY_{\rig}$ be the associated rigid analytic spaces in the sense of Raynaud, $X_{\Q_\kappa}^{\an}$ and $Y_{\Q_{\kappa}}^{\an}$  be the analytic spaces over $\Q_{\kappa}$ associated with the $\Q_\kappa$-schemes $X_{\Q_\kappa}$ and $Y_{\Q_\kappa}$. Similarly, we have formal schemes  $\fX^{*}$, $\fY^*$, and their associated rigid analytic spaces $\fX^*_{\rig}$, $\fY^*_{\rig}$. Then  we have natural inclusions of rigid analytic spaces
\[\fX_{\rig}\subset X^{\an}_{\Q_{\kappa}}\subset \fX^*_{\rig},\quad\quad \fY_{\rig}\subset Y_{\Q_{\kappa}}^{\an}\subset \fY_{\rig}^*.\]
For any extension of valuation fields $L/\Q_{\kappa}$, we use a subscript $L$ to denote the base change of a rigid space over $\Q_{\kappa}$ to $L$, e.g.  $\fX_{\rig,L}$, $X^{\an}_{\Q_{\kappa},L}=X^{\an}_{L}$, $\fX^{*}_{\rig,L}$...
For any weight $\vk\in \Z^{\bB}$, by an obvious abuse of notation, we still denote by $\omegab^{\vk}$ the modular line bundles of weight $\vk$ on the formal schemes $\fX^*$ and $\fY^*$, and on  rigid spaces $\fX^*_{\rig}$, $\fY^*_{\rig}$.

\begin{prop}[Koecher's Principle]\label{prop-koecher}For any finite extension  $L$ over $\Q_{\kappa}$,  we have a commutative diagram of canonical isomorphisms
\[\xymatrix{
H^0(Y^*_{L},\omegab^{\vk})\ar[r]^{\sim}\ar[d]^{\sim}& H^0(Y_{L},\omegab^{\vk})\ar[d]^{\sim} &\\
H^0(\fY^*_{\rig,L},\omegab^{\vk})\ar[r]^{\sim}&H^0(Y_{L}^{\an},\omegab^{\vk})\ar[r]^{\sim}& H^0(\fY_{\rig,L},\omegab^{\vk}),
}
\]
where the horizontal arrows are natural restriction map, and the vertical arrows are analytification maps.
\end{prop}
\begin{proof}
First, the diagram  above is clearly commutative. The top horizontal isomorphism is the classical Koecher principle \cite[Thm. 4.3(i)]{Ch}. The left vertical isomorphism follows from rigid GAGA \cite[7.6.11]{Ab}, since $Y^*$ is proper over $\Spec(W)$. The lower horizontal arrows are clearly injective. By Corollary \ref{cor-ext-1}, the composite map
\[H^0(\fY^*_{\rig,L},\omegab^{\vk})\ra H^0(\fY_{\rig,L},\omegab^{\vk})\]
is an isomorphism, because the special fiber $Y^*_{\kappa}$ is normal at the cusps. It follows that the lower horizontal arrows are isomorphisms,  hence so is the right vertical one. This finishes the proof of the Proposition.
\end{proof}

Of course, the same proposition is still true with $Y$ replaced by $X$. We leave the details to the reader.

\subsection{$U_{\fp}$-operators}\label{sect-U_p} We follow the treatment in \cite[1.11]{KL}. If  $A$ is a HBAV over a $W$-algebra $R$, we have a decomposition of finite and locally free group schemes over $R$
\[A[p]=\prod_{\fp | p}A[\fp],\]
 where $\fp$ runs through all the prime ideals of $\cO_F$ dividing $p$, and $A[\fp]$ is the subgroup scheme killed by all $a\in \fp$.  Then $A[\fp]$ is a group scheme of $(\cO_F/\fp)$-vector spaces of dimension $2$. We fix a prime ideal $\fp$ of $\cO_F$ above $p$, and put $\kappa(\fp)=\cO_F/\fp$. Let $\cC(\fp)$ be the scheme over $\Q_{\kappa}$ which represents the functor that attaches to a $\Q_{\kappa}$-algebra $R$ the set of isomorphism classes of  $5$-tuples $(A/R, \lambda, \psi_N, H, H')$ where:
\begin{itemize}
\item $(A/R,\lambda, \psi_N, H)$ is an object in $Y(R)$;

\item $H'\subset A[\fp]$ is a closed $(\cO_F/\fp)$-cyclic isotropic subgroup scheme such that $H'\cap H=\{0\}$.
\end{itemize}
 We have two canonical maps
\begin{equation}\label{diag-U_p}
\xymatrix{
& \cC(\fp)\ar[ld]_{\pi_1}\ar[rd]^{\pi_2}\\
Y_{\Q_{\kappa}}&& Y_{\Q_{\kappa}}
}
\end{equation}
given  respectively by
\begin{align*}\pi_1(A/R, \lambda, \psi_N,H,H')&= (A/R, \lambda, \psi_N, H)\\
\pi_2(A/R,\lambda, \psi_N, H,H')&= (B/R, \lambda_B, \psi_{B,N}, (H'+H)/H'),
 \end{align*}
 where $B$ is the quotient abelian scheme $A/H'$, $\lambda_B$ and $\psi_{B,N}$ are the naturally induced  polarization and level structure on $B$. Note that since we are in characteristic 0, both $\pi_1$ and $\pi_2$ are finite flat of degree $|\kappa(\fp)|$. Let $(\cA,\lambda, \psi_N,H,H')$ be the universal object on $\cC(\fp)$, and $\phi:\cA\ra \cB=\cA/H'$ be the canonical isogeny. We have a commutative diagram
\[\xymatrix{
\cA\ar[d]^{f}&\cA\ar[d]_{f_1}\ar[l]\ar[r]^{\phi}&\cB\ar[d]^{f_2}\ar[r]&\cA\ar[d]\\
Y_{\Q_{\kappa}}
&\cC(\fp)\ar[l]_{\pi_1}\ar@{=}[r]&\cC(\fp)\ar[r]^{\pi_2}&Y_{\Q_{\kappa}},
}\]
where both the left and right squares are cartesian. We have a natural morphism of $(\cO_{\cC(\fp)}\otimes \cO_F)^{\times}$-torsors:
\[\pi_2^*(\omegab_{\cA/Y_{\Q_{\kappa}}})\xra{\sim} f_{2*}(\Omega^1_{\cB/\cC(\fp)})\xra{\phi^*}f_{1*}(\Omega^1_{\cA/\cC(\fp)})\xla{\sim}\pi_1^*(\omegab_{\cA/Y_{\Q_\kappa}}),\]
which induces a natural homomorphism of line bundles
\[\phi^*: \pi_2^*(\omegab^{\vk})\lra\pi_1^*(\omegab^{\vk})\]
for any $\vk\in \Z^\bB$.
We define the $U_{\fp}$-operator on the space $H^0(Y_{\Q_\kappa}, \omegab^{\vk})$ as the composite
\begin{equation}\label{defn-U_p}
U_{\fp}:H^0(Y_{\Q_\kappa}, \omegab^{\vk})\xra{\pi_2^*}H^0(\cC(\fp),\pi_2^*(\omegab^{\vk}))\xra{\phi^*}H^0(\cC(\fp), \pi_1^*(\omegab^{\vk}))\xra{\frac{1}{|\kappa(\fp)|}\tr} H^0(Y_{\Q_{\kappa}},\omegab^{\vk}),
\end{equation}
where ``$\tr$'' is  induced by the trace map $\pi_{1*}\pi_1^*(\omegab^{\vk})\ra \omegab^{\vk}$. Explicitly, if $L$ is a finite extension of $\Q_\kappa$, $(A/L,H)=(A/L,\lambda, \psi_N, H)$ is an element in $Y(L)$, and $\omega$ is a generator of $\omegab_{A}=e^*\Omega^1_{A/K}$ as an $(L\otimes \cO_F)$-module, then we have
\begin{equation}\label{equ-U_p}
(U_{\fp}f)(A/L,H,\omega)=\frac{1}{|\kappa(\fp)|}\sum_{\substack{H'\subset A[\fp]\\ H'\cap H=0}}f(A/H',(H'+H)/H',p^{-1}\hat{\phi}^*(\omega)),
\end{equation}
where $H'$ runs over all the $\cO_F$-subgroup of $A[\fp]$ of order $|\kappa(\fp)|$ with  $H\cap H'=0$, and $\hat{\phi}:A/H'\ra A$ is the canonical isogeny with kernel $A[p]/H'$.

\subsection{Norms}\label{norms} Let $\vk=\sum_{\beta\in \bB}k_{\beta}\in \Z^{\bB}$ be a multi-weight. For any quasi-compact admissible open subset $V\subset \fY_{\rig}$, we define a norm on the space $H^0(V,\omegab^{\vk})$ as follows. Let $L$ be a finite extension of $\Q_\kappa$, and $Q=(A,H)\in \fY_{\rig}(L)$ be a rigid point, i.e.  a morphism of formal schemes $Q:\Spf(\cO_L)\ra \fY$ such that $Q^*(\fA,\fH)=(A,H)$, where $(\fA,\fH)$ is the universal formal HBAV together with its universal isotropic $(\cO_F/p)$-cyclic subgroup over $\fY_{\rig}$. Let $\omega$ be a generator of the free $(\cO_K\otimes\cO_F)$-module $\omegab_{A}$, and $\omega_{\beta}$ be its $\beta$-component for any $\beta\in \bB$. Then
\[\omega^{\vk}=\otimes_{\beta\in \bB}\,\omega^{k_{\beta}}_{\beta}\]
is a basis of the $\cO_K$-module $Q^*(\omegab^{\vk})$. For any element $f\in Q^*(\omegab^{\vk})\otimes_{\cO_L}L$, we write $f=f(A,H,\omega)\omega^{\vk}$ with $f(A,H,\omega)\in L$, and define
$$|f|=|f(A,H,\omega)|_p.$$
For any admissible open subset $V\subset \fY_{\rig}$, and a section $f\in H^0(V,\omegab^{\vk})$, we define
\[|f|_{V}=\sup_{Q\in |V|}|f(Q)|\in \R_{>0}\cup\{\infty\}.\]
If $V$ is  quasi-compact, then $|f|_V\in \R_{>0}$ by the maximum modulus principle in rigid analysis, and $H^0(V,\omegab^{\vk})$ is a $\Q_k$-Banach space with the norm   $|\cdot |_V$.  If $V$ is clear from the context, we usually omit the subscript $V$ from the notation.

\subsection{Hasse Invariants} Let $R$ be a $\kappa$-algebra, and $A$ be a HBAV over $R$. Let $\omega_{A/R}=H^0(A,\Omega^1_{A/R})$ be the module of invariant $1$-differentials of $A$ relative to $R$, and $\Lie(A)$ be the Lie algebra of $A$; so we have $\Lie(A)=\Hom_{R}(\omega_{A/R},R)$. The Verschiebung homomorphism $V_{A}:A^{(p)}\ra A$ induces a map of $R$-modules $\HW: \Lie(A)^{(p)}\ra \Lie(A)$, where $\Lie(A)^{(p)}$ is the base change of $\Lie(A)$ via the absolute Frobenius endomorphism $F_R:a\mapsto a^p$ of $R$. Equivalently, we have a canonical map
 \begin{align}
h:\omega_{A/R}\ra \omega_{A/R}^{(p)}.\nonumber
\end{align}
Note that $\omegab_{A/R}$ is a locally free $R\otimes\cO_F$-module of rank 1, and let $\omegab_{A/R}=\bigoplus_{\beta\in \bB}\omegab_{A/R,\,\beta},$
where $\omega_{A/R,\beta}$ is the direct summand on which $\cO_F$ acts via the character $\chi_{\beta}$. Thus we have a decomposition $h=\oplus_{\beta\in \bB}h_{\beta}$, where
\begin{equation}\label{part-Ha}
h_\beta: \omegab_{A/R,\,\beta}\ra \omegab^{(p)}_{A/R,\,\sigma^{-1}\circ \beta}.
\end{equation}
The morphism $h_\beta$ thus defines a Hilbert modular form (with full level) of weight $p\cdot \sigma^{-1}\circ \beta-\beta$ over $\kappa$, and we call it the \emph{$\beta$-partial Hasse invariant}. The product
$E=\prod_{\beta\in \bB}h_{\beta}$ is thus a Hilbert modular form of weight $(p-1)\sum_{\beta\in\bB}\beta$ over $\kappa$, called simply \emph{the Hasse invariant}. If $A$ is a HBAV over an algebraically closed field containing $\kappa$, the Hasse invariant $E(A)\neq 0$ if and only if $A$ is ordinary in the usual sense, \ie, the finite group scheme $A[p]$ is  isomorphic to $\mu_p^g\times (\Z/p\Z)^g$.

\subsection{}\label{sect-ord} Let $X_{\kappa}$, $Y_{\kappa}$ be the special fibers of $X$ and $Y$, and $X_{\kappa}^{\ord}$ be the locus where the Hasse invariant $h$ does not vanish, or equivalently the open subscheme of $X_{\kappa}$ parametrizing  polarized ordinary HBAV. For a HBAV $A$ over a $\kappa$-algebra $R$, the kernel of the Frobenius homomorphism of $A$ is naturally an $(\cO_F/p)$-cyclic isotropic subgroup of $A[p]$. In other words, the kernel of Frobenius defines a section $s:X_{\kappa}\ra Y_{\kappa}$ of the projection $\pi: Y_{\kappa}\ra X_{\kappa}$. We put $Y_{\kappa}^\ord=s(X_{\kappa}^{\ord})$. In particular, $Y_{\kappa}^{\ord}$ is isomorphic to $X_{\kappa}^{\ord}$.

Let $\fX^{\ord}$ and $\fY^{\ord}$ be respectively the open formal subschemes of $\fX$ and $\fY$ corresponding to the open subsets $X_{\kappa}^\ord\subset X_{\kappa}$ and $Y_{\kappa}^\ord\subset Y_{\kappa}$, and $\fX_{\rig}^{\ord}$ and $\fY^{\ord}_{\rig}$ be the associated rigid analytic generic fibers. Then $\fX^{\ord}_{\rig}$ and $\fY^{\ord}_{\rig}$ are respectively  quasi-compact admissible open subsets of $\fX_{\rig}$ and $\fY_{\rig}$. Let $(\fA,\lambda, \psi_N)$ be the  universal  formal HBAV over $\fX$.   Over the ordinary locus $\fX^{\ord}$, we have an extension of finite flat $\cO_F$-group schemes
  \[0\ra \fA[p]^{\mu}\ra \fA[p]\ra \fA[p]^\et\ra 0,\]
  where $\fA[p]^\et$ is \'etale of order $p^g$, and $\fA[p]^\mu$ is of multiplicative type and lifts the kernel of Frobenius of $\fA\otimes_{W} \kappa$. The finite flat subgroup $\fA[p]^\mu$ is $(\cO_F/p)$-cyclic and isotropic for the Weil pairing induced by the polarization $\lambda$,  and it defines thus an section $s^{\circ}:\fX^{\ord}\ra \fY^{\ord}$ lifting the section $s:X_{\kappa}^{\ord}\ra Y_{\kappa}^{\ord}$ given by the kernel of Frobenius. In particular, the natural projections  $\fY^{\ord}\ra \fX^{\ord}$ and $\fY^{\ord}_{\rig}\ra \fX^{\ord}_{\rig}$ are canonical isomorphisms.

\begin{defn}\label{defn-CO-HMF}
Let $L$ be a finite extension of $\Q_{\kappa}$.

(i) For $\vk\in \Z^{\bB}$,  an element of $H^0(\fX^{\ord}_{\rig,L},\omegab^{\vk})$ is called a \emph{$p$-adic Hilbert modular form} of level $\Gamma_{00}(N)$ and weight $\vk$ with coefficients in $L$. \\

(ii) We say an element $f\in H^0(\fX^{\ord}_{\rig,L}, \omegab^{\vk})$ is \emph{overconvergent} if $f$ extends to a strict neighborhood  $V$ of $\fX^{\ord}_{\rig,L}$ in $\fX_{\rig,L}$. We put
\[\rM^{\dagger}_{\vk}(\Gamma_{00}(N),L )=\varinjlim_{V} H^0(V, \omegab^{\vk})\]
 where $V$ runs over the  strict neighborhoods of $\fX^{\ord}_{\rig,L}$ in $\fX_{\rig,L}$, and we call it the space of overconvergent $p$-adic Hilbert modular forms of level $\Gamma_{00}(N)$ and weight $\vk$.
\end{defn}

\begin{rem}\label{rem-overcon}
 By the theory of canonical subgroups (cf. \cite[\S3]{KL} and \cite[Thm. 5.3.1]{GK}), the isomorphism of ordinary loci $\pi:\fY^{\ord}_{\rig}\ra \fX^{\ord}_{\rig}$ extends to a strict neighborhood of $\fY^{\ord}_{\rig}$ in $\fY_{\rig}$. Therefore, the natural notion of (overconvergent) $p$-adic Hilbert modular forms of level $\Gamma_{00}(N)\cap \Gamma_{0}(p)$ is the same as its counterpart of level $\Gamma_{00}(N)$. Hence, we can  always consider an element $f\in \rM^{\dagger}_{\vk}(\Gamma_{00}(N),L)$ as a section of $\omegab^{\vk}$ over a strict neighborhood of $\fY^{\ord}_{\rig,L}$.
\end{rem}

We refer the reader to \cite{Be} for the definition of strict neighborhood. Here we construct an explicit fundamental system of strict neighborhoods of $\fX^{\ord}_{\rig,L}$ in $\fX_{\rig,L}$ by using the Hasse invariant. Let $\tilde{E^{k_0}}$ be a lift in $H^0(\fX^*, \omegab^{k_0(p-1)\vec{1}})$, where $\vec{1}=(1,\cdots, 1)\in \Z^\bB$, of $k_0$-th power of the Hasse invariant   $E^{k_0}$ for some integer $t\geq 1$. The existence of such a lift follows from Koecher's principle and the fact that $\omegab^{\vec{1}}$ is ample on the minimal compactification $X^*$. For any rational number $0<r\leq 1$, we denote by $\fX^{\ord}_{\rig,L}(r)$ the admissible open subset  of $\fX_{\rig,L}$ where $|\tilde{E^t}|\geq r^t$. Since the Hasse invariant is well-defined modulo $p$,  the subset $\fX_{\rig,L}(r)$ does not depend on the choice of the lift $\tilde{E^{k_0}}$ if  $p^{-1/k_0}<r\leq 1$. It is clear that  $\fX_{\rig,L}(1)=\fX^{\ord}_{\rig,L}$, and the $\fX_{\rig,L}(r)$'s form a fundamental system of strict neighborhoods of $\fX_{\rig,L}^{\ord}$ in $\fX_{\rig,L}$. Hence, we have
\[\rM^{\dagger}_{\vk}(\Gamma_{00}(N),L)=\lim_{r\ra 1^{-}}H^0(\fX^{\ord}_{\rig,L}(r),\omegab^{\vk}).\]
Note that  each $H^0(\fX_{\rig,L}^{\ord}(r),\omegab^{\vk})$ is a Banach space over $L$, and the natural restriction map
$$H^0(\fX_{\rig,L}^{\ord}(r),\omegab^{\vk})\ra H^0(\fX_{\rig,L}^{\ord}(r'),\omegab^{\vk})$$
for $0<r<r'<1$ are compact \cite[2.4.1]{KL}. Therefore, $\rM^{\dagger}_{\vk}(\Gamma_{00}(N),L)$ is (compact) direct limit of Banach spaces over $L$.

By Remark \ref{rem-overcon}, we have a natural injective map
\[H^0(Y_{L},\omegab^{\vk})\simeq H^0(\fY_{\rig,L},\omegab^{\vk})\ra \rM^{\dagger}_{\vk}(\Gamma_{00}(N),L),\]
where we have used Prop. \ref{prop-koecher} for the first isomorphism.
We denote its image  by $\rM_{\vk}(\Gamma_{00}(N)\cap\Gamma_0(p),L)$, and call them the \emph{classical} Hilbert modular forms.

\subsection{}\label{sect-rig-U_p} For a prime ideal $\fp$ of $\cO_F$ above $p$, let $\cC(\fp)^{\an}$ be the rigid analytification of the scheme $\cC(\fp)$ over $\Q_\kappa$ considered in \ref{sect-U_p}. Then just as $\cC(\fp)$, the rigid space $\cC(\fp)^{\an}$ represents an analogous functor in the rigid analytic setting, and we still have a universal object $(\cA^{\an}, \lambda,\psi_N,H, H')$  over $\cC(\fp)$. We have analogous morphisms $\pi_1,\pi_2:\cC(\fp)^{\an}\ra Y_{\Q_{\kappa}}^{\an}$.  We put
$$\cC(\fp)_{\rig}=\pi^{-1}_1(\fY_{\rig})=\pi^{-1}_{2}(\fY_{\rig}).$$
   The rigid analytic space $\cC(\fp)_{\rig}$ is  the locus of $\cC(\fp)^{\an}$ where $\cA^\an$ has good reduction, it  classifies the objects $(A,H,H')$, where $(A,H)$ is   a rigid point of $\fY_{\rig}$, and $H'\subset A[p]$ is a group scheme of  $(\cO_F/\fp)$-vector space of dimension $1$ with $H\cap H'=0$.  We have a rigid version of the Hecke correspondence:
 \[
 \xymatrix{
& \cC(\fp)_{\rig}\ar[ld]_{\pi_1}\ar[rd]^{\pi_2}\\
\fY_{\rig}&& \fY_{\rig}
}\]
given by $\pi_1(A,H,H')=(A,H)$ and $\pi_2(A,H,H')=(A/H',(H+H')/H')$. We have also  a set theoretical Hecke correspondence  between the rigid points of $\fY_{\rig}$
\begin{align}\label{defn-set-U_p}
U_{\fp}: \fY_{\rig}&\ra \fY_{\rig}\\
Q&\mapsto \pi_2(\pi_1^{-1}(Q)).\nonumber
\end{align}
Here, it's an obvious notation for convenience, because $U_{\fp}$ is not really a morphism of rigid analytic spaces. If $U$ and $V$ are admissible open subsets of $\fY_{\rig}$ such that $U_{\fp}(U)\subset V$, i.e. $\pi_1^{-1}(U)\subset \pi_2^{-1}(V)$. A rigid version of the formula \eqref{equ-U_p} defines the $U_{\fp}$-operator
\[U_\fp:H^0(V,\omegab^{\vk})\xra{\pi_2^*}H^0(\pi_2^{-1}(V),\pi_2^*\omegab^{\vk})\xra{\phi^*} H^0(\pi^{-1}(U),\pi_1^*\omegab^{\vk})\xra{\frac{1}{|\kappa(\fp)|}\tr}H^0(U,\omegab^{\vk}).\]

\begin{lemma} The ordinary locus $\fY_{\rig}^{\ord}$ is stable under the Hecke correspondence $U_{\fp}$, i.e. we have $\pi_2(\pi_1^{-1}(\fY_{\rig}^{\ord}))\subset \fY_{\rig}^{\ord}$.
\end{lemma}
\begin{proof}
Let $L$ be a finite extension of $\Q_{\kappa}$,  $(A,H)\in \fY_{\rig}^{\ord}(L)$ be a rigid point, i.e. $A$ is a HBAV over $\cO_L$ with ordinary reduction, and $H\subset A[p]$ is the multiplicative part. We have to show that $(A/H',(H+H')/H')$ still lies in $\fY_{\rig}^{\ord}$ for any  isotropic $(\cO_F/\fp)$-subgroup $H'\subset A[\fp]$ with $H'\cap H=0$. Actually, such a $H'$ is necessarily \'etale over $\cO_L$. Therefore, the isogeny $A\ra A/H'$ is \'etale, and the subgroup $(H+H')/H'$ is the multiplicative part of the HBAV $A/H'$.
\end{proof}

This easy Lemma implies immediately that a $U_{\fp}$-operator analogous to the classical case can be defined on the space $H^0(\fY_{\rig,L}^{\ord},\omegab^{\vk})$ for any weight $\vk\in \Z^{\bB}$ and any finite extension $L$ of $\Q_{\kappa}$.  In order to show that overconvergent $p$-adic Hilbert modular forms are stable under  $U_{\fp}$, we need to extend canonically the section $s^{\circ}:\fX_{\rig}^{\ord}\ra \fY_{\rig}^{\ord}$ to a strict neighborhood of $\fX_{\rig}^{\ord}$. As already mentioned in Remark \ref{rem-overcon}, this is the theory of canonical subgroups, and it  has been developed  by many authors (cf. for instance \cite{KL} and \cite{GK}).   The main result of this paper is the following

\begin{thm}\label{thm-main}
 Let $f$ be an element of $\rM^{\dagger}_{\vk}(\Gamma_{00}(N),L)$. Assume that for every prime ideal $\fp$ of $\cO_F$ above $p$, we have $[\kappa(\fp):\F_p]\leq 2$ and $U_{\fp}(f)=a_{\fp}f$ with $v_p(a_\fp)< k_{\beta}-[\kappa(\fp):\F_p]$ for all $\beta\in \bB_{\fp}$, then $f$ is classical, i.e., $f\in \rM_{\vk}(\Gamma_{00}(N)\cap\Gamma_0(p), L)$.
\end{thm}
\begin{rem} It is reasonable to expect that the theorem is also true without the restriction  $[\kappa(\fp):\F_p]\leq 2$.  The main obstacle to this generalization  is that the geometry of $Y$ in the higher dimensional case is too complicated, and we don't well  understand  the dynamics of the $U_{\fp}$-operator.
\end{rem}

In the reminder of this section, we suppose $p\geq 3$, and  indicate some consequences of our results on eigencurves for overconvergent Hilbert eigenforms.

\subsection{}\label{over-family}
We follow the treatments in \cite{KL}.
Let $L$ be a finite extension of $\Q_{\kappa}$, and $R$ be a Banach algebra over $L$ with a submultiplicative norm $|\cdot|$ and $Z\in R$ such that $|Z|<p^{-\frac{2-p}{p-1}}$. We fix an integer $k_0>0$ coprime to $p$ such that the $k_0$-th power of the Hasse invariant lifts to $\widetilde{E^{k_0}}\in H^0(X,\omegab^{k_0(p-1)\vec{1}})$. For $\vk\in \Z^{\bB}$, Kisin and Lai defined in \cite[4.2.3]{KL} the space of overconvergent Hilbert modular forms over $R$ of level $\Gamma_{00}(N)$ and weight $\vk+Z$ to be space
\begin{equation}\label{defn-KL}
\rM^{\dagger}_{\vk+Z}(\Gamma_{00}(N),R)
=\varinjlim_{V}H^0(V,\omegab^{\vk})\hat{\otimes}_{\Q_{\kappa}}R,
\end{equation}
where $V$ runs over a fundamental system of quasi-compact strict neighborhoods of $\fX_{\rig}^{\ord}$ in $\fX_{\rig}$. This space is equipped with an action of the Hecke operators $T_{\fa}$ (resp. $U_{\fa}$) for each ideal $\fa\subset \cO_F$ coprime to  $pN$ (resp. not coprime to $pN$). We point out that Kisin-Lai's definition of these operators involves the lift $\widetilde{E^{k_0}}^t$, and if $Z=0$, we come back to the definition \eqref{defn-CO-HMF}. We denote by $\bT_{\vk+Z}^{\dagger}(\mu_{N})$ the ring of endomorphisms generated by these operators as $\fa$ runs over the ideals of $\cO_F$. Let $f\in \rM^{\dagger}_{\vk+Z}(\Gamma_{00}(N),R)$. We say $f$ is an \emph{eigenform} if it is non-zero and a simultaneous eigenvector of all the operators in $\bT^{\dagger}_{\vk+Z}(\mu_N)$. We say an eigenform  $f$  has \emph{ finite slope} if it's an eigenform and the eigenvalue of $U_{(p)}=\prod_{\fp|p}U_{\fp}$ is non-zero, i.e. $U_{(p)}(f)\neq 0$.

 According to \cite[4.2.8]{KL}, the space $\rM^{\dagger}_{\vk+Z}(\Gamma_{00}(N),R)$ interpolates the $p$-adic overconvergent modular forms of integer weights in the following sense. Let $L'/L$ be a finite extension, $\psi:R\ra L'$ be a homomorphism of $L$-algebras sending $Z$ to $(p-1)k_0t$ for some $t\in \Z_{>0}$. Then for each Hecke operator $T=T_{\fa}$ or $T=U_{\fa}$, $\psi$ induces a commutative diagram
\[
\xymatrix{
\rM^{\dagger}_{\vk+Z}(\Gamma_{00}(N),R)\ar[r]^{\psi}\ar[d]^{T}&\rM^{\dagger}_{\vk}(\Gamma_{00}(N),L')
\ar[r]^-{\cdot\widetilde{E^{k_0}}^t}&\rM^{\dagger}_{\vk+(p-1)k_0t\vec{1}}(\Gamma_{00}(N),L')\ar[d]^{T}\\
\rM^{\dagger}_{\vk+Z}(\Gamma_{00}(N),R)\ar[r]^{\psi}&\rM^{\dagger}_{\vk}(\Gamma_{00}(N),L')\ar[r]^-{\cdot \widetilde{E^{k_0}}^t}&\rM^{\dagger}_{\vk+(p-1)k_0t\vec{1}}(\Gamma_{00}(N),L').
}\]

 \subsection{} Let $S$ be set of infinity places  and all the finite places of $F$ dividing $pN$, and $G_{F,S}$ be the Galois group of the maximal algebraic extension of $F$ in $\C$ which is unramified outside $S$. Let  $f$ be a Hilbert cusp eigenform of level $\Gamma_{00}(N)\cap \Gamma_0(p)$ and weight $\vk\in \Z^{\bB}$, where  the integers $k_{\beta}$'s are $\geq 2$ and have the same parity. We may suppose that the Fourier coefficients of $f$ at cusps are in $\cO_{f}\subset \cO_{\C_p}$, where $\cO_{f}$ is the normalization of $\Z_p$ in a finite extension $K_f$ of $\Q_p$. Let  $\F_f=\cO_f/\m_{\cO_f}$ be the residue field of $\cO_f$.  Then by the work of many people \cite{Ca,Ta,BR}, we know how to associate to $f$  a $2$-dimensional Galois representation $(\rho_f, V_f)$  of $G_{F,S}$ over $K_{f}$. The representation $\rho_f$ is characterized by the condition that, for any prime ideal $\fq\notin S$ of $\cO_F$, $\Tr(\rho_f(\Frob_{\fq}))$ coincides with the eigenvalue of $T_{\fq}$ on $f$, where $\Frob_{\fq}$ denotes  an arithmetic Frobenius element in $G_{F,S}$ at $\fq$.

  If $L_{f}\subset V_{f}$ be a $G_{F,S}$-stable $\cO_f$-lattice of $V_f$, we put $\overline{V}_f=L_{f}\otimes_{\cO_f}\F_{f}$ by the abuse of notation, and denote by $\rhob_f$ the resulting representation of $G_{F,S}$ over $\F_{f}$. Note that the isomorphism class of $\rhob_f$ is only determined by $\rho_f$ up to semi-simplification. We call such a $\rhob_f$  $p$-modular residual representation of $G_{F,S}$, and call the pseudo-representation associated with the semi-simplification of $\rhob_f$ a \emph{$p$-modular pseudo-representation}.\\

 Let $\rhob$ be a $p$-modular pseudo-representation of $G_{F,S}$ over a finite field $\F$. We denote by $R^{\univ}(\rhob)$ the universal deformation ring of $\rhob$, whose existence is proved in \cite[5.1.3]{CM}, and by $\rho^{\univ}$ the universal pseudo-representation of $G_{F,S}$ over $R^{\univ}(\rhob)$. Let $\cZ(\rhob)$ be the rigid analytic space over $L$ attached to $R^{\univ}(\rhob)\otimes_{W(\F)}L$, and $\cW$ be the weight space over $L$ of $\mathrm{Res}_{\cO_{F}/\Z}\G_{m}$, i.e. the rigid space over $L$ which  to an affinoid algebra $A$ over $L$ assigns the set of continuous $A^{\times}$-valued characters of $(\cO_{F}\otimes_{\Z} \Z_p)^{\times}$. By the local class field theory, the determinant $\det(\rho^{\univ})$ defines a  character $(\cO_F\otimes_{\Z}\Z_p)^{\times}\ra R^{\univ}(\rhob)$, i.e. a map of rigid analytic spaces $\cZ(\rhob)\ra \cW$.

   Fix a weight $\vk\in \Z^{\bB}$ such that all the $k_{\beta}$'s have the same parity. We denote by $\cW_{\vk}$ the subspace of $\cW$ whose  points in a complete subfield $L'\subset \C_{p}$ containing $L$, correspond to characters $\chi: (\cO_F\otimes_{\Z}\Z_{p})^{\times }\ra L'^{\times}$ such that there exists $z_0\in L'$ with $v_p(z_0)>\frac{1}{p-1}-1$ and $\chi=\chi_{\vk}\cdot \langle\mathrm{Nm}\rangle^{z_0}$, where $\mathrm{Nm}=\chi_{\vec{1}}:(\cO_F\otimes \Z_p)^{\times}\ra \Z_{p}^{\times}$ is the natural norm map, and $a\mapsto \langle a\rangle$ is the  projection $\Z_{p}^{\times}\ra (1+p\Z_p)$ defined by the canonical isomorphism
  \[\Z_{p}^{\times}\simeq(\Z/p\Z)^{\times}\times (1+p\Z_p).
\] Then  $\cW_{\vk}$ is a one-dimensional rigid closed subspace of $\cW$.
    We denote by
 $$\chi^{\univ}:(\cO_F\otimes \Z_p)^{\times}\ra \cO(\cW_{\vk})^{\times}$$ the universal character, and by $Z$ the rigid analytic function on $\cW_{\vk}$ such that $\chi^{\univ}=\chi_{\vk}\cdot \langle \mathrm{Nm}\rangle^Z$. We have $|Z|<p^{-\frac{2-p}{p-1}}$, and  $Z$ establishes an isomorphism of rigid spaces
 over $L$
 \[Z:\cW_{\vk}\xra{\sim}\DD_{L}(0, p^{-\frac{2-p}{p-1}})=\{x\in \C_p|v_p(x)> \frac{1}{p-1}-1\}.\] We denote by $\cZ_{\vk}(
 \rhob)$ the inverse image of $\cW_{\vk}$ in $\cZ(\rhob)$. We put $\cY_{\vk}(\rhob)=\cZ_{\vk}(\rhob)\times \G_{m,L}\times \prod_{\fq\in S}\bA^1_{L}$, and denote by $x_p$ the canonical coordinate on $\G_{m,L}$, $x_{\fq}$ the coordinate on $\bA^1_{L}$. We denote also by $Z$ the rigid analytic function on $\cY_{\vk}(\rhob)$ induced by the canonical projection $\cY_{\vk}(\rhob)\ra \cW_{\vk}$, and by the same notation the analytic function on $\cY_{\vk}(\rhob)$ induced by that on $\cZ_{\vk}(\rhob)$.

 Kisin and Lai proved in  \cite[4.5.4]{KL} that there exists a  rigid analytic closed subspace $\cC_{\vk}(\rhob)\subset \cY_{\vk}(\rhob)$ that interpolates $p$-adic overconvergent Hilbert eigenforms of finite slope. More precisely, it satisfies the following properties:
  \begin{enumerate}
  \item For  any closed subfield $L'\subset \C_p$ containing $L$ and any $c\in \cC_{\vk}(\rhob)(\C_p)$, there exists an eigenform  $f_c\in \rM^{\dagger}_{\vk+Z(c)}(\Gamma_{00}(N),L')$ of finite slope, such that, if $\lambda_{c,T}$ denotes the eigenvalue of $T\in \bT^{\dagger}_{\vk+Z(c)}(\mu_N)$ on $f_c$, we have $\lambda_{c, U_{(p)}}=x_p(c)$, and for all primes $\fq$ of $\cO_F$
      \[\lambda_{c,T_{\fq}}=\Tr(\rho^{\univ}(\Frob_{\fq}))\quad \text{if } \fq\notin S\quad \text{and } \lambda_{c,U_{\fq}}=x_{\fq}(c)\quad \text{if }\fq\in S. \]

      \item For any $z_0\in L'$ with $v_p(z_0)> \frac{1}{p-1}-1$, and $c\in \cC_{\vk}(\rhob)(L')$, the association $c\mapsto \{\lambda_{T,c}\}_{T\in \bT^{\dagger}_{\vk+Z(c)}(\mu_N)}$ induces a bijection between $c\in \cC_{\vk}(\rhob)(L')$ with $Z(c)=z_0$, and systems of $\bT^{\dagger}_{\vk+z_0}(\mu_N)$-eigenvalues arising from the eigenforms of finite slope in $ \rM^{\dagger}_{\vk+z_0}(\Gamma_{00}(N),\C_p)$,

\end{enumerate}

We say a point $c\in \cC_{\vk}(\rhob)(L')$ is \emph{classical}, if $Z(c)=(p-1)k_0t$ for some integer $t\in \Z_{\geq0}$, and the image of $f_c$ under the composite map
\[\rM^{\dagger}_{\vk+Z(c)}(\Gamma_{00}(N), L')\simeq \rM^{\dagger}_{\vk}(\Gamma_{00}(N),L')\xra{\cdot \widetilde{E^{k_0}}^{t}}\rM^{\dagger}_{\vk+(p-1)k_0t\vec{1}}(\Gamma_{00}(N),L')\]
comes from an element of $\rM_{\vk+(p-1)k_0t\vec{1}}(\Gamma_{00}(N)\cap \Gamma_0(p),L')$.

\begin{thm}\label{thm-density} Assume $p\geq 3$ and $[\kappa(\fp):\F_p]\leq 2$ for all prime ideal $\fp\subset \cO_F$ dividing $p$. Then the  classical points are  Zariski-dense in $\cC_{\vk}(\rhob)$.
\end{thm}
\begin{proof}Let $C$ be an irreducible component of $\cC_{\vk}(\rhob)$. It suffices to prove that $C$ contains infinitely many classical points. Let $\pi_{\vk}(\rhob):\cC_{\vk}(\rhob)\ra \cW_{\vk}$ be the natural projection onto the weight space. By the same argument of \cite[Thm. B]{CM}, the morphism $\pi_{\vk}(\rhob)$ is component-wise almost surjective in the sense that, for every irreducible component of $\cC_{\vk}(\rhob)$, the complement of its image in $\cW_{\vk}$ consists of at most a finite number of weights. Therefore, there exists an admissible affinoid subdomain $B
\subset C$  such that its image $W_0$ in $\cW_{\vk}$ is an affinoid domain containing a closed disk
 with center $z_0=(p-1)k_0t_0\in W_0$ of radiu $p^{-n}$ for certain $n\in \Z_{>0}$. By the maximum modulus principle, there exists a real number $\alpha>0$ such that the slopes $v_p(x_{\fp}(c))<\alpha$ for  any prime $\fp$ of $F$ dividing $p$ and any  $c\in B(\Lb)$, where $\Lb$ denotes the algebraic closure of $L$ in $\C_p$.  Hence, there are infinitely many points $c\in B(\Lb)$ such that we have $Z(c)=(p-1)k_0 t$, and  $k_{\beta}+(p-1)k_0t>\alpha+2$ for all $\beta\in \bB$. By property (1) of the eigencurve $\cC_{\vk}(\rhob)$ and Theorem \ref{thm-main}, such  a point $c$ corresponds to classical Hilbert modular forms of level $\Gamma_{00}(N)\cap \Gamma_0(p)$ and weight $\vk+(p-1)k_0 t\vec{1}$.
\end{proof}

\bigskip
\section{Finite flat group schemes with RM and Breuil-Kisin modules}

\subsection{} We recall first the theory of Breuil-Kisin modules for finite flat group schemes. Let $k$ be a  perfect field of characteristic $p>0$, $W(k)$ be its ring of Witt vectors, and $K$ be a finite totally ramified extension of $K_0=W(k)[1/p]$ of degree $e$, and $\cO_K$ be its ring of integers. Fix a uniformizer $\pi$ of $K$ with Eisenstein polynomial $E(u)$, and put $\fS=W(k)[[u]]$. We equip $\fS$ with the endomorphism $\varphi$ which acts on $W(k)$ via Frobenius and sends $u$ to $u^p$. A  finite torsion Breuil-Kisin module (of height $1$) is a finite $\fS$-module $\fM$ equipped with a $\varphi$-linear endomorphism $\varphi:\fM\ra \fM$  verifying the following properties:
\begin{enumerate}
\item $\fM$ has $p$-power torsion.

\item $\fM$ has projective dimension $1$ as $\fS$-module, i.e. there is a two term resolution of $\fM$ by finite free $\fS$-modules.

    \item The cokernel of the linearized map
    \[1\otimes\varphi:\varphi^*(\fM)=\fS\otimes_{\fS,\varphi}\fM\xra{1\otimes\varphi} \fM\]
is killed by $E(u)$.
\end{enumerate}
We denote by $\Modd^{\tor}$ the category of finite torsion Breuil-Kisin modules (of height $1$).  Note that $\fS$ is regular local ring of dimension $2$, a finite $\fS$-module has projective dimension $1$ if and only if it has depth $1$. Therefore, condition (2) in the definition above is equivalent to saying that $\fM$ has no $u$-torsion.  Similarly, a finite free  Breuil-Kisin module (of height $1$) is a finite free $\fS$-module $\fM$ equipped with a $\varphi$-linear endomorphism $\varphi:\fM\ra \fM$ such that the third condition above is satisfied. We denote by $\Modd^{\fr}$ the category of finite free Breuil-Kisin modules (of height $1$). By a Breuil-Kisin module, we mean an object $\fM$ in either $\Modd^{\tor}$ or $\Modd^{\fr}$ depending on the situation. For a Breuil-Kisin module $\fM$, the map $1\otimes\varphi:\varphi^*(\fM)\ra \fM$ is necessarily injective \cite[1.1.9]{Ki09}, and we denote its image by $(1\otimes\varphi)\varphi^*(\fM)$.

The main motivation of studying Breuil-Kisin modules in this paper is the following theorem due to Kisin \cite[0.5]{Ki} when $p\geq 3$, to Eike Lau \cite[7.6, 7.7]{La10} and Tong Liu \cite[1.0.1,  1.0.2]{L10} independently when $p=2$.
 \begin{thm}\label{thm-kisin-1}
 There is a natural anti-equivalence between the category of commutative finite and flat group schemes over $\cO_K$ of $p$-power order and the category $\Modd^{\tor}$. Similarly, the category of $p$-divisible groups is naturally anti-equivalent to the category $\Modd^{\fr}$.
\end{thm}

 The following Proposition will be fundamental for our application of Breuil-Kisin modules to the analytic continuation of Hilbert modular forms.

\begin{prop}\label{prop-deg-kisin}
Let $G$ be a finite and flat group scheme (or a $p$-divisible group) over $\cO_K$. Let $\fM$ be the Breuil-Kisin module associated with $G$.
Then  there is a canonical isomorphism of $\cO_K$-modules
\[\omega_{G}\xra{\sim}\fM/(1\otimes\varphi)\varphi^*(\fM),\]
where $\omega_{G}$ is the module of invariant differentials of $G$.
\end{prop}

To prove this proposition, we need Breuil's filtered $S$-modules. Let $S$ be the $p$-adic completion of the divided power envelop of $W(k)[u]$ with respect to the principal ideal $(E(u))$, i.e.  $S$ is the completion of subring $W(k)[u, E(u)^i/i!:i\geq1]$ of $K_0[u]$, and $\Fil^1S\subset S$ be the kernel of the natural surjection $S\ra \cO_{K}$ sending $u\mapsto\pi$. We note that $\fS$ is naturally a subring of $S$,  and the endomorphism $\varphi$ on $\fS$ extends to $S$. We check easily that $\varphi(\Fil^1S)\subset pS$, and we put $\varphi_1=\frac{1}{p}\varphi\,|_{\Fil^1S}$ and $c=\varphi_1(E(u))$. A filtered $S$-module $(\cM,\Fil^1\cM,\varphi_1)$ consists of the following data:
 \begin{enumerate}
 \item A finite generated $S$-module  $\cM$ and a submodule $\Fil^1\cM$ with $E(u)\cM\subset\Fil^1\cM$.

 \item A $\varphi$-linear morphism $\varphi_1:\Fil^1\cM\ra \cM$ such that, for $s\in \Fil^1S$ and $x\in \cM$, we have $\varphi_1(sx)=\frac{1}{c}\varphi_1(s)\varphi_1(E(u)x)$, and the  image  of $\varphi_1$ generates $\cM$ as an $S$-module.
  \end{enumerate}

 We denote by $\MF$ the category of filtered $S$-modules. It has a natural structure of an exact category. A sequence is short exact if it is short exact as a sequence of $S$-modules, and induces a short exact sequence on $\Fil^1$'s.

     Let $\fM$ be an object in $\Modd^{\tor}$ or $\Modd^{\fr}$. We can associate covariantly with $\fM$ a filtered $S$-module as follows. We put $\cM(\fM)=S\otimes_{\fS,\varphi}\fM$, and define $\Fil^1\cM(\fM)$ to be the submodule of $\cM(\fM)$ whose image under the morphism of $S$-modules
     \[S\otimes_{\fS,\varphi}\fM\xra{1\otimes \varphi}S\otimes_{\fS}\fM\]
     lies in $\Fil^1S\otimes_{\fS}\fM$. The morphism $\varphi_1:\Fil^1\cM(\fM)\ra \cM(\fM)$ is defined to be the composite
     \[\Fil^1\cM(\fM)\xra{1\otimes \varphi}\Fil^1S\otimes_{\fS}\fM\xra{\frac{1}{p}\varphi\otimes1}S\otimes_{\fS,\varphi}\fM=\cM(\fM).\]
      By \cite[8.1]{La10}, $(\cM(\fM),\Fil^1\cM(\fM),\varphi_1)$ is an object in $\MF$, and the functor $\fM\mapsto \cM(\fM)$ is exact.  By definition of $\Fil^1\cM(\fM)$, we have an embedding \cite[1.1.15]{Ki09}
     \[\cM(\fM)/\Fil^1\cM(\fM)\xra{1\otimes \varphi}S\otimes_{\fS}\fM/(\Fil^1S\otimes_{\fS}\fM)\xra{\sim} \fM/E(u)\fM,\]
     which induces an isomorphism
     \beq\label{iso-br-kis}
     \cM(\fM)/\Fil^1\cM(\fM)\xra{\sim} (1\otimes\varphi)\varphi^*(\fM)/E(u)\fM.
     \eeq
If $\fM$ is an object in $\Modd^{\fr}$, then $(1\otimes\varphi)\varphi^*(\fM)$ is free over $\fS$. So the $\fS$-module $(1\otimes\varphi)\varphi^*(\fM)/E(u)\fM$ has projective dimension $1$, hence depth $1$.  It follows that $\cM(\fM)/\Fil^1\cM(\fM)$ and $(1\otimes\varphi)\varphi^*(\fM)/E(u)\fM$ are actually  finite free $\cO_K$-modules.

\begin{lemma}\label{lem-br-ki} Let $\fM$ be an object in $\Modd^{\tor}$ or $\Modd^{\fr}$ as above.
We have a canonical isomorphism
\begin{align*}
\Fil^1\cM(\fM)/\Fil^1S\cM(\fM)\simeq \fM/(1\otimes\varphi)\varphi^*(\fM).
\end{align*}

\end{lemma}
\begin{proof}

The following   construction was indicated to me  by Tong Liu. Consider the map $
1\otimes \varphi: \varphi^*(\fM)\ra \fM$.  Motivated by the definition of $\cM(\fM)$, we put
\[\Fil^1\varphi^*(\fM)=\{x\in \varphi^*(\fM)\,|\, (1\otimes\varphi)(x)\in E(u)\fM\}.\]
Thus $1\otimes \varphi$ induces an isomorphism
\beq\label{equ-b-1}\varphi^*(\fM)/\Fil^1\varphi^*(\fM)\xra{\sim}(1\otimes\varphi)\varphi^*(\fM)/E(u)\fM.\eeq
We denote by
\[\varphi_1^{\#}:\Fil^1\varphi^*(\fM)\ra \fM\]
 the natural map given by $x\mapsto (1\otimes\varphi)(x)/E(u)$. Then $\varphi^{\#}_1$ induces an isomorphism \beq\label{equ-b-2}\Fil^1\varphi^*(\fM)/E(u)\varphi^*(\fM)\xra{\sim} \fM/(1\otimes\varphi)\varphi^*(\fM).\eeq
 On the other hand, the natural inclusion $\fS\hra S$ induces an inclusion $\iota:\varphi^*(\fM)\hra \cM(\fM)$. We check easily that $\iota(\Fil^1\varphi^*(\fM))\subset \Fil^1\cM(\fM)$, and we have a commutative diagram of exact sequences
 \[\xymatrix{0\ar[r]&\Fil^1\varphi^*(\fM)/E(u)\varphi^*(\fM)\ar[r]\ar[d]&\varphi^*(\fM)/E(u)\varphi^*(\fM)\ar[r]\ar[d]^{\sim}
 &\varphi^*(\fM)/\Fil^1\varphi^*(\fM)\ar[d]^{\sim}\ar[r]&0\\
 0\ar[r]&\Fil^1\cM(\fM)/\Fil^1S\cM(\fM)\ar[r]&\cM(\fM)/\Fil^1S\cM(\fM)\ar[r]&\cM(\fM)/\Fil^1\cM(\fM)\ar[r]&0,
 }\]
where the vertical arrows are induced by $\iota$. The middle vertical arrow  is easily seen to be an isomorphism, and so is  the right vertical map  because of \eqref{iso-br-kis} and \eqref{equ-b-1}. It follows that the left vertical one is also an isomorphism. In view of \eqref{equ-b-2}, the lemma follows.

\end{proof}

\begin{proof}[Proof of Prop. \ref{prop-deg-kisin}]
 Let $\cM=\cM(\fM)$ be the filtered $S$-module associated with $G$. By Lemma \ref{lem-br-ki}, it suffices to prove that we have a canonical isomorphism
  \[\omega_{G}\xra{\sim}\Fil^1\cM/\Fil^1S\cM.\]
  If $G$ is a $p$-divisible group, this follows from \cite[3.3.5]{BBM} or \cite[8.1]{La10} and the fact that $\cM$ is the evaluation of the Dieudonn\'e crystal of $G$ at the PD-thickening $\Spec(\cO_K)\hra \Spec(S)$. If $G$ is finite flat group scheme over $\cO_K$, then $G$ can be embedded as a closed subgroup scheme into a  $p$-divisible group $G_0$ over $\cO_K$ \cite[3.1.1]{BBM}; we put $G_1=G_0/G$.
  We denote respectively by $\cM_0$ and $\cM_1$  the filtered $S$-modules associated with $G_0$ and $G_1$. Since all the constructions are functorial in $G$, the exact sequence of  groups $0\ra G\ra G_0\ra G_1\ra0$ induces a commutative diagram of exact sequences of $\cO_K$-modules
 \[\xymatrix{0\ar[r] &\omega_{G_1}\ar[r]\ar[d]^{\sim} &\omega_{G_0}\ar[r]\ar[d]^{\sim}&\omega_{G}\ar[r]\ar[d]& 0\\
 0\ar[r] & \Fil^1\cM_1/\Fil^1S\cM_1\ar[r] &\Fil^1\cM_0/\Fil^1S\cM_0\ar[r]&\Fil^1\cM/\Fil^1S\cM\ar[r]&0.}\]
 Since the left two vertical arrows are isomorphisms, it follows that so is the right one.

\end{proof}

\subsection{$\Z_{p^g}$-groups}  Let $g>0$  be an integer, $\Q_{p^g}$ be the  unramified extension of $\Q_p$ of degree $g$, $\Z_{p^g}$ be its ring of integers. We assume $k$ contains $\F_{p^g}$. We identify $\mathrm{Emd}_{\Z_p}(\Z_{p^g},\cO_K)$, the set of embeddings of $\Z_{p^g}$ into $\cO_K$, with $\Z/g\Z$, and the natural action of Frobenius on $\mathrm{Emd}_{\Z_p}(\Z_{p^g},\cO_K)$ is identified with $i\mapsto i+1$. For an $(\cO_K\otimes \Z_{p^g})$-module $M$, we have a canonical splitting $M=\oplus_{i\in \Z/g\Z}M_i$, where $\Z_{p^g}$ acts on $M_i$ via the $i$-th embedding. If $N$ is a finite torsion $\cO_K$-module, we choose an isomorphism $N\simeq \oplus_{i=1}^d\cO_K/(a_i)$ with $a_i\in \cO_K$, and define the degree of $N$ to be
\[\deg(N)=\sum_{i=1}^dv_p(a_i).\]

  A \emph{$\Z_{p^g}$-group} over $\cO_K$ is a commutative finite and flat group scheme over $\cO_K$ endowed  with an action of $\Z_{p^g}$. Fargues \cite{Fa} defined the degree function of a finite flat group scheme over $\cO_K$. We give a refinement of this function for $\Z_{p^g}$-groups over $\cO_K$.

\begin{defn}\label{defn-deg} Let $G$ be a $\Z_{p^g}$-group over $\cO_K$, and $\omega_G=\oplus_{i\in \Z/g\Z}\, \omega_{G,i}$ be its module of invariant differential forms. We put
$$\deg_i(G)=\deg(\omega_{G,i}),$$
 and we call it the $i$-th degree of $G$.
  \end{defn}

Hence, the degree function of Fargues is $\deg(G)=\sum_{i\in \Z/g\Z}\deg_i(G)$.
If $0\ra G_1\ra G\ra G_2\ra 0$ is an exact sequence of $\Z_{p^g}$-groups over $\cO_K$, we have an exact sequence of $\cO_K\otimes \Z_{p^g}$-modules
\[0\ra \omega_{G_2}\ra \omega_{G}\ra \omega_{G_1}\ra 0;\]
hence we have $\deg_{i}(G)=\deg_{i}(G_1)+\deg_{i}(G_2)$ for any $i\in \Z/f\Z$.

Recall that a scheme of $1$-dimensional $\F_{p^{g}}$-vector spaces over $\cO_K$ is a $\Z_{p^g}$-group $G$ over $\cO_K$ such that $G(\Kb)$ is an $\F_{p^g}$-vector space of dimension $1$, where $\Kb$ is an algebraic closure of $K$. According to Raynaud's classification of such finite flat group schemes \cite[1.5.1]{Ra}, we have an isomorphism of schemes
\begin{equation}\label{Ray-class}
G\simeq\Spec\bigl(\cO_K[T_{i}:i\in \Z/g\Z]/(T^p_{i-1}-a_{i}T_{i})_{i\in\Z/g\Z}\bigr),
\end{equation}
for some $a_i\in \cO_K$ with $0\leq v_p(a_i)\leq 1$.
Using this isomorphism, we have $\deg_{i}(G)=v_p(a_{i})$ for $i\in \Z/g\Z$. This following Lemma is a refinement of \cite[Cor. 3]{Fa}.

\begin{lemma}\label{lem-ray}
Let $\phi:H\ra G$ be a homomorphism of schemes of $1$-dimensional $\F_{p^g}$-vector spaces over $R$ that induces an isomorphism on the generic fibers. Then  for any $i\in \Z/g\Z$, we have
\[\sum_{j=0}^{g-1}p^{j}\deg_{i-j}(G)\geq \sum_{j=0}^{g-1}p^{j}\deg_{i-j}(H).\]
Moreover, all the equalities hold if and only if $\phi$ is an isomorphism.
\end{lemma}
\begin{proof}
Let $(a_{i})_{i\in \Z/g\Z}$ (resp. $(b_{i})_{i\in \Z/g\Z}$) be respectively the elements in $\cO_K$ appearing in   an isomorphism as \eqref{Ray-class} for $G$  (resp. for $H$).  We have $\deg_{i}(G)=v_p(a_i)$ and $\deg_{i}(H)=v_p(b_{i})$. The existence of $\phi$ implies that there exist $u_{i}\in \cO_K$ for all $i\in \Z/g\Z$ such that
$a_{i}u_{i}=b_{i}u^p_{i-1}$  \cite[1.5.1]{Ra}.
Hence, we have
\[\prod_{j=0}^{g-1}\biggl(\frac{a_{i-j}}{b_{i-j}}\biggr)^{p^j}=\prod_{j=0}^{g-1}
\frac{u_{i-j-1}^{p^{j+1}}}{u_{i-j}^{p^j}}=u_{i}^{p^{g}-1}.\]
  The lemma follows immediately from the fact that $v_p(u_i)\geq 0$, and that $\phi$ is an isomorphism if and only if $v_p(u_{i})=0$ for all $i\in\Z/g\Z$.
\end{proof}

\subsection{} We describe the  $\Z_{p^g}$-groups over $\cO_K$ in terms of Breuil-Kisin modules.  A \emph{$\Z_{p^g}$-Breuil-Kisin module} is an object $\fM$ in $\Modd^{\tor}$ together with an action of $\Z_{p^g}$ commuting with $\varphi$. Equivalently, a $\Z_{p^g}$-Breuil-Kisin module $\fM$ is an $(\fS\otimes\Z_{p^g})$-module $\fM=\oplus_{i\in\Z/g\Z}\fM_i$ satisfying the following properties:
\begin{enumerate}
\item each $\fM_i$ is killed by some power of $p$;

\item each $\fM_i$ has projective dimension $1$, i.e. $\fM_i$ has a two term resolution by finite free $\fS$-modules;

 \item there is a $\varphi$-linear endomorphism $\varphi:\fM\ra \fM$ such that $\varphi(\fM_i)\subset \fM_{i+1}$ and the cokernel of the linearization $1\otimes \varphi: \varphi^*(\fM_{i})\ra \fM_{i+1}$ is killed by $E(u)$.
     \end{enumerate}

We denote by $\Mod^{\tor}$ the category of $\Z_{p^g}$-Breuil-Kisin modules, and the morphisms in $\Mod^{\tor}$ are homomorphisms of $(\fS\otimes\Z_{p^g})$-modules commuting with $\varphi$. Let $\fM$ be an object of $\Mod^{\tor}$.  We define the $i$-th degree of $\fM$ as
\[\deg_i(\fM)=\frac{1}{e}\leng \biggl(\fM_i/(1\otimes\varphi)\varphi^*(\fM_{i-1})\biggr),\]
Here, ``$\leng$'' denotes the length, and the factor $\frac{1}{e}$ will be justified in Lemma \ref{lem-deg-kisin}. If $0\ra \fL\ra \fM\ra \fN\ra 0$
 is an exact sequence in $\Mod^{\tor}$, it follows from an easy diagram chasing that
$$\deg_i(\fM)=\deg_i(\fL)+\deg_i(\fN)$$ for any $i\in \Z/g\Z$.

 From Theorem \ref{thm-kisin-1}, it follows easily  that \emph{ the category of $\Z_{p^g}$-groups over $\cO_K$ is anti-equivalent to the category $\Mod^{\tor}$.}

\begin{lemma}\label{lem-deg-kisin}
Let $G$ be $\Z_{p^g}$-group over $\cO_K$, and $\fM$ be its corresponding $\Z_{p^g}$-Breuil-Kisin module. Then we have $\deg_{i}(G)=\deg_{i}(\fM)$ for $i\in \Z/g\Z$.
 \end{lemma}
 \begin{proof} Since  the isomorphism in Prop. \ref{prop-deg-kisin} is canonical, it necessarily commutes with $\Z_{p^g}$-actions. We have an isomorphism of $(\cO_{K}\otimes\Z_{p^g})$-modules
 $$\omega_{G}=\bigoplus_{i\in \Z/g\Z}\omega_{G,i}\simeq \fM/(1\otimes\varphi)\varphi^*(\fM)=\bigoplus_{i\in \Z/g\Z}\fM_{i}/(1\otimes\varphi)\varphi^*(\fM_{i-1}).$$
 The lemma follows immediately.
 \end{proof}
\begin{defn}\label{defn-gp-RM} Let $n\geq 1$ be an integer, $G$ be a truncated Barsotti-Tate group of level $n$ over $\cO_K$ equipped with an action of $\Z_{p^g}$. We say $G$ has \emph{formal real multiplication} (or just RM for short) by $\Z_{p^g}$ if $G$ has dimension $g$ and height $2g$ and  $\omega_G$ is a free $(\cO_K/p^n\otimes \Z_{p^g})$-module of rank $1$; in particular,  we have $\deg_i(G)=n$ for $i\in \Z/g\Z$.
\end{defn}


By \cite[2.3.6]{Ki}, an object $\fM$ of $\Mod^{\tor}$
correpsonds to a truncated Barsotti-Tate group  of level $n$ over $\cO_K$ with RM by $\Z_{p^g}$ if and only if

(a) $\fM$ is a free $(\fS/p^n\otimes \Z_{p^g})$-module of rank $2g$;

(b) $\fM/(1\otimes\varphi)\varphi^*(\fM)$ is a free $(\cO_K/p^n\otimes\Z_{p^g})$-module of rank $g$.\\

Let $\fM=\oplus_{i\in \Z/g\Z}\fM_{i}$ be such a $\Z_{p^g}$-Breuil-Kisin module. For each $i\in \Z/g\Z$, we have a filtration of free $\cO_K/p^n$-modules
\[0\ra (1\otimes\varphi)\varphi^*(\fM_{i-1})/E(u)\fM_i\ra \fM_i/E(u)\fM_{i}\ra \fM_i/(1\otimes \varphi)\varphi^*(\fM_{i-1})\ra 0.\]
We say a basis  $(\delta_i,\epsilon_i)_{i\in \Z/g\Z}$ of $\fM$ over $\fS/p^n$ is \emph{adapted} if $\delta_i\in (1\otimes\varphi)\varphi^*(\fM_{i-1})$ and  the image of $\epsilon_i$ generates $\fM_i/(1\otimes\varphi)\fM_{i-1}$ over $\cO_K/p^n$. Then under such an adapted basis, there exists $\begin{bmatrix}a_i& b_i\\
c_i&d_i\end{bmatrix}\in \GL_{2}(\fS/p^n)$ such that
\begin{equation}\label{equ-kisin}
\varphi(\delta_{i-1},\epsilon_{i-1})=(\delta_i,\epsilon_i)\begin{bmatrix}a_i& b_i\\
E(u)c_i& E(u)d_i\end{bmatrix}.
\end{equation}

 \subsection{}\label{sect-hodge} Let $G$ be a truncated Barsotti-Tate group of level $1$ over $\cO_K$ with RM by $\Z_{p^g}$, and $G_1=G\otimes_{\cO_K}\cO_K/p$ be its reduction modulo $p$.  The Lie algebra of the Cartier dual of $G_1$, denoted by $\Lie(G^\vee_1)$, is a free ($\cO_K/p\otimes\Z_{p^g})$-module of rank $1$. Let
$$\Lie(G_1^\vee)=\bigoplus_{i\in\Z/g\Z} \Lie(G_1^\vee)_i$$
 be the decomposition according to the action of $\Z_{p^g}$. The Frobenius homomorphism $F_{G_1}:G_1\ra G_1^{(p)}$ induces a Frobenius linear endomorphism
 $$\HW: \Lie(G_1^\vee)\ra \Lie(G_1^\vee)$$ with $\HW:\Lie(G_1^\vee)_{i-1}\subset \Lie(G_1^\vee)_{i}$. We choose a basis $\delta_i$ for each $\Lie(G_1^\vee)_i$ over  $\cO_K/p$, and write  $\HW(\delta_{i-1})=t_i\delta_i$. Let $v_p:\cO_K/p\ra [0,1]$ be the truncated $p$-adic valuation.
 We define the \emph{$i$-th partial Hodge height of $G$} to be $w_i(G)=v_p(t_i)\in [0,1]$. It's clear that the definition does not depend on the choice of the basis $\delta_i$. Note that $G$ is ordinary if and only if $w_i(G)=0$ for all $i\in \Z/g\Z$.

\begin{lemma}\label{part-hodge}
Let $G$ be as above, and $\fM$ be its corresponding Breuil-Kisin module. We choose an adapted basis of $\fM$ so that $\varphi$ is represented by matrices of the form \eqref{equ-kisin}. Then we have $w_i(G)=\frac{1}{e}\min\{e,v_u(\overline{a}_i)\}$, where $\overline{a}_i$ is the image of $a_i$ in $\fS_1=k[[u]]$, and $v_u$ denotes the $u$-adic valuation.
\end{lemma}
\begin{proof} Let $\cM=\cM(\fM)$ be the filtered $S$-module associated with $G$. By \cite[8.1]{La10}, there is a canonical isomorphism of $(\cO_K/p)$-modules
$\Lie(G_1^\vee)\simeq \cM/\Fil^1\cM.$
Combining with \eqref{iso-br-kis}, we have an isomorphism $$\Lie(G^\vee_1)\simeq (1\otimes\varphi)\varphi^*(\fM)/E(u)\fM,$$ where the second term is considered as an $\cO_K/p$-module via the isomorphism $\fS_1/E(u)\fS_1\simeq \cO_K/p$ given by $u\mapsto \pi$. Since everything is functorial in $G$, this is actually an isomorphism of $(\cO_K/p\otimes\Z_{p^g})$-modules. Since the endomorphisms $\HW$ on $\Lie(G_1^\vee)$ and $\varphi$ on $(1\otimes\varphi)\varphi^*(\fM)/E(u)\fM$ are both induced by Frobenius homomorphism of $G_1$ \cite[1.1.2]{K09}, one checks easily that $\HW$ and $\varphi$ coincide with each other via the canonical isomorphism above. The Lemma follows immediately.
\end{proof}

Let $G$ be a truncated Barsotti-Tate group  of level $1$ over $\cO_K$ with RM by $\Z_{p^g}$. We say a finite flat closed subgroup scheme $H\subset G$ is \emph{$\Z_{p^g}$-cyclic or $\F_{p^g}$-cyclic}, if $H(\Kb)$ is a one-dimensional $\F_{p^g}$-subspace of $G(\Kb)$.

\begin{lemma}\label{lem-two-cyclic} If $H,H'$ are two distinct $\Z_{p^g}$-cyclic subgroups of $G$, then for all $i\in \Z/g\Z$, we have
\[\sum_{j=0}^{g-1}p^j(\deg_{i-j}(H)+\deg_{i-j}(H'))\leq \frac{p^g-1}{p-1}.\]
\end{lemma}
\begin{proof}
The Lemma follows from \ref{lem-ray} applied to the homomorphism $H\hra G\ra G/H'$.
\end{proof}

The following theorem is a slightly generalized version of \cite[Thm. 5.4.3]{GK}. The proof is motivated by  \cite[3.4]{Ha}.

\begin{thm}\label{thm-can}
Let $G$ be a truncated Barsotti-Tate group of level $1$ with RM by $\Z_{p^g}$, and denote $w_i=w_i(G)$. Assume that $w_i+pw_{i-1}<p$ for all $i\in \Z/g\Z$. Then there exists a $\Z_{p^g}$-cyclic subgroup $C\subset G$ such that $\deg_i(C)=1-w_i$; moreover, $C$ is the unique $\Z_{p^g}$-cyclic subgroup of $G$ satisfying
$$\deg_{i}(C)+p\deg_{i-1}(C)>1\quad \text{for any }i\in \Z/g\Z.$$
\end{thm}

\begin{proof}

Let $\fM=\oplus_{i\in\Z/g\Z}\fM_i$ be the Breuil-Kisin module associated with $G$. We choose an adpated basis $(\delta_i,\epsilon_i)_{i\in \Z/g\Z}$ of $\fM$ so that $\varphi$ is represented by matrices \eqref{equ-kisin}. Note that $w_{i+1}+pw_{i}<p$ implies that $w_i<1$; so by Lemma \ref{part-hodge}, we have $v_u(a_i)=ew_i$. By Lemma \ref{lem-deg-kisin}, we have to show  that there exists  a quotient $\fN=\oplus_{i\in\Z/g\Z}\fN_i$ of $\fM$ such that $\deg_i(\fN)=1-w_i$ and it is the unique quotient satisfying $\deg_i(\fN)+p\deg_{i-1}(\fN)>1$ for $i\in\Z/g\Z$.

We prove first the existence of $\fN$. We construct a direct summand $\fL=\oplus_{i\in \Z/g\Z}\fL_i$ of $\fM$ such that $\fL_i$ is the submodule of $\fM_i$ generated by $$\eta_i=(\delta_i,\epsilon_i)\begin{bmatrix}1\\u^{e(1-w_i)}z_i\end{bmatrix}, $$ where $z_i\in \fS_1$ is some element to be determined. If we require that $\fL$ is an sub-oject of $\fM$,  there should exists a certain $A_i\in \fS_1$ such that
$\varphi(\eta_{i-1})=A_i\eta_i$. Using the equation \eqref{equ-kisin}, we get
\begin{equation}\label{equ-exist}
\begin{cases}a_i+b_iu^{ep(1-w_{i-1})}z_{i-1}^p&=A_i\\
u^e(c_i+u^{ep(1-w_{i-1})}d_iz_{i-1}^p)&=u^{e(1-w_i)}z_iA_i
\end{cases}
\end{equation}
Since $v_u(a_i)=ew_i$, there exists a unit $\hat{a}_i\in \fS_1$ such that $a_i\hat{a}_i=u^{ew_i}$. Then we have $z_{i}=g_i(z_{i-1})$, where 	
\[g_i(z)=\frac{c_i+u^{ep(1-w_{i-1})}d_iz^p}{\hat{a}_i+u^{e(p-pw_{i-1}-w_i)}b_iz^p}.\]
Note that $1-w_{i-1}>0$ and $p-pw_{i-1}-w_{i}>0$ by assumption. We get therefore
\[z_i=g_i\circ g_{i-1}\circ\cdots \circ g_{i-g+1}(z_i).\]
By iteration, it's easy to see that the equation admits  a unique solution in $\fS_1$ for $z_i$. This well defines the sub-oject $\fL\subset \fM$. From
$$ew_i=v_u(a_i)<ep(1-w_{i-1})\leq v_u(b_iu^{ep(1-w_{i-1})}z_{i-1}^p)$$ by assumption, we deduce that
$\deg_i(\fL)=\frac{1}{e}v_u(A_i)=w_i.$ We can take $\fN$ to be $\fM/\fL$.

For the uniqueness of $\fN$, we assume that  $\fN'$ is a quotient of $\fM$ with $\deg_{i+1}(\fN')+p\deg_{i}(\fN')>1$ for $i\in \Z/g\Z$. We have to show that $\fN=\fN'$. Since $\deg_{i+1}(\fN')\leq 1$, we have $\deg_{i}(\fN')>0$, i.e.  $\fN'_i/(1\otimes \varphi)\varphi^*(\fN'_{i-1})\neq 0$. Let $\fL'$ be the kernel of $\fM\ra \fN'$. We have for each $i\in \Z/g\Z$ an exact sequence of $\cO_K/p$-modules
\[0\ra \fL'_{i}/(1\otimes \varphi)\varphi^*(\fL'_{i-1})\ra \fM_{i}/(1\otimes\varphi)\varphi^*(\fM_{i-1})\ra \fN'_{i}/(1\otimes \varphi)\varphi^*(\fN'_{i-1})\ra 0.\]
Because $\fM_{i}/(1\otimes \varphi)\varphi^*(\fM_{i-1})\simeq \fS_1/u^e\epsilon_i$, we see that $\fN'_{i}$ is generated  by the image of $\epsilon_i$ in $\fN'_i$. Hence there exists  $x_i\in \fS_1^\times$ such that $\eta_i'=\delta_i+x_i\epsilon_i\in \fL'_i$.  We put $r_i=\frac{1}{e}v_u(x_i)$. As $\varphi(\fL'_{i-1})\subset \fL'_{i}$, there exists $A'_i\in \fS_1$ such that $\varphi(\eta_{i-1})=A'_i\eta_{i}$. We have
$$v_u(A'_i)=e\deg_i(\fL')= e(1-\deg_i(\fN')),$$
where the second equality comes from the additivity of the degree function and the fact that $\deg_i(\fM)=1$.
Therefore, we get
\begin{equation}\label{equ-f-1}
v_u(A'_i)+pv_u(A'_{i-1})=e(p+1-\deg_i(\fN')-p\deg_{i-1}(\fN'))<ep.
\end{equation}
On the other hand, using \eqref{equ-kisin} as above, we have equations
\begin{align}
a_i+b_ix_{i-1}^p&=A'_i\label{equ-a-1}\\
u^e(c_i+d_i x_{i-1}^p)&=x_iA'_i.\label{equ-a-2}
\end{align}
We claim that $r_{i}\geq 1-w_i$ for any $i\in \Z/g\Z$. Admitting this claim for the moment, we can write  $x_i=u^{e(1-w_i)}z_i'$. Then the $z_i'$'s will satisfy the equations \eqref{equ-exist}. But we have seen that \eqref{equ-exist} admits a unique solution $z_i$ for each $i\in \Z/g\Z$. So we have $z_i'=z_i$, and hence $\fL=\fL'$. It remains to prove the claim.
We deduce first from \eqref{equ-a-2} that $v_u(A'_i)\geq e(1-r_i)$. In view of \eqref{equ-f-1}, we get
\begin{equation}\label{equ-f-2}r_i+pr_{i-1}>1\quad\; \text{for all }i\in \Z/g\Z. \end{equation}
If $r_i<1-w_i$ for some $i\in \Z/g\Z$, we have $v_u(A'_i)\geq e(1-r_i)>ew_i$. Because of \eqref{equ-a-1}, we have
\[ew_i=v_u(a_i)=v_{u}(b_ix_{i-1}^{p})\geq epr_{i-1}.\]
So we have $1-r_i>w_i\geq pr_{i-1}$, i.e. $r_i+pr_{i-1}<1$, which contradicts with \eqref{equ-f-2}. This completes the proof.

\end{proof}

 \begin{rem} The subgroup $C\subset G$ given by the theorem is called the \emph{canonical subgroup} of $G$. By the same argument as in \cite[Thm. 5.4.2]{GK}, it's not hard to see that the subgroup $C$ verifies the ``Frobenius lifting property'':\emph{ If we denote $w=\max_{i\in \Z/g\Z}\{w_i\}<1$, then $C\otimes_{\cO_K}(\cO_{K}/p^{1-w}\cO_K)$ coincides with the kernel of Frobenius of $G\otimes_{\cO_K}(\cO_{K}/p^{1-w}\cO_K),$ where $p^{1-w}$ denotes any element in $\cO_K$ with $p$-adic valuation $1-w$.}
\end{rem}

Let $I$ be a subset of $\Z/g\Z$,  $I^c$ be its complementary, and $|I|$ be the cardinality. We denote by $\sigma(I)$ the image of $I$ under the action of Frobenius $\sigma: \Z/g\Z\ra \Z/g\Z$ given by $i\mapsto i+1$. Let $G$ be a truncated Barsotti-Tate group of level $1$ with RM by $\Z_{p^g}$ over $\cO_K$. We say that a $\Z_{p^g}$-cyclic subgroup $H$ of $G$ is \emph{ special of type $I$} if $\deg_i(H)=1$ for $i\in I$ and $\deg_i(H)=0$ for $i\in I^c$.

 \begin{prop}\label{prop-gen-split}
  Let the notation be as above.

  \emph{(a)} Assume that $G$ admits a special subgroup $H$ of type $I$.
  \begin{enumerate}

    \item[(1)] The group $H$ is necessarily a truncated Barsotti-Tate group of level $1$ of height $g$ and dimension $|I|$ over $\cO_K$. Moreover, we have $w_{i}(G)=1$ for $i\in \sigma(I)\cap I^{c}$, and $w_{i}(G)=0$ for $i\in (\sigma(I)\cap I)\cup (\sigma(I^c) \cap I^c)$.

        \item[(2)] If $H'$ be another $\Z_{p^g}$-cyclic subgroup of $G$ distinct from $H$, then we have
        $$\deg_{i}(H')\leq \frac{1}{p-1}(1-\frac{1}{p^{g-1}})\quad\text{for all }i\in I.$$ In particular, if $I\neq \emptyset$, then $G$ admits at most one special subgroup of type $I$.

        \item[(3)]  If $G$ admits  another special subgroup $H'$ of type  $I'$ with $H'\neq H$, then   either $I=I'=\emptyset$, or  $I'=I^c$ and the natural map $H\times H'\ra G$  is an isomorphism of finite flat group schemes over $\cO_K$. In the second case, if $H''$ is a $\Z_{p^g}$-cyclic subgroup of $G$ distinct from $H$ and $H'$, then we have
            $$\deg_{i}(H'')\leq \frac{1}{p-1}(1-\frac{1}{p^{g-1}})\quad \text{ for all }i\in \Z/g\Z.$$

    \end{enumerate}

 \emph{(b)} Conversely, assume that $\sigma(I^{c})\subset I$, $w_i(G)=1$ for $i\in \sigma(I)\cap I^{c}$, $w_{i}(G)=0$ for $i\in \sigma(I)\cap I$, and $w_i(G)>0$ for $i\in \sigma(I^c)$. Then $G$ admits a special subgroup of type $I$.
\end{prop}

\begin{proof}

 (a)  First, by Raynaud's explicit classification \eqref{Ray-class}, a group scheme  $H$ of $\F_{p^g}$-vector spaces of dimension $1$ over $\cO_K$ is a truncated Barsotti-Tate group of level $1$ if and only if $\deg_i(H)\in\{0,1\}$ for all $i\in \Z/g\Z$.  Let $\fM=\oplus_{i\in \Z/g\Z}\fM_i$ be the Breuil-Kisin module associated to $G$, and $(\delta_i,\epsilon_i)$ be an adapted  basis of $\fM_i$ so that  $(1\otimes \varphi)\varphi^*(\fM_{i-1})=\fS_1\delta_i\oplus E(u)\fS_1 \epsilon_i$. Let $\fL=\oplus_{i\in \Z/g\Z}\fL_{i}$ be the Breuil-Kisin submodule of $\fM$ attached to the quotient $G/H$, and $\fN=\fM/\fL$ be  associated to $H$. By Lemma \ref{lem-deg-kisin}, we have $\deg_i(\fN)=\deg_i(H)=1$ for $i\in I$, i.e.
 \[\fN_i/(1\otimes \varphi)\varphi^*(\fN_{i-1})=\fM_{i}/((1\otimes\varphi)\varphi^*(\fM_{i-1})+\fL_i)\simeq \fS_1/u^e\epsilon_i.\]
  Since $\fL_i$ is a direct summand of $\fM_i$, there exists $x_i\in \fS_1$ such that $\fL_i=\fS_1(\delta_i+u^e x_i\epsilon_i)$. Up to replacing $\delta_i$ by $\delta_i+u^ex_i\epsilon_i$, we may assume that $\fL_i=\fS_1\delta_i$ for  $i\in I$. Similarly, since $\deg_i(\fN)=\deg_i(H)=0$ for $i\in I^c$, we see that
  \[\fM_i=\fL_i+(1\otimes \varphi)\varphi^*(\fM_{i-1}).\]
  Up to modifying $\epsilon_i$, we may assume that $\fL_i=\fS_1\epsilon_i$ if $i\in I^c$.  The following facts follow  easily from the condition that $\varphi(\fL_{i-1})\subset \fL_i$.
  \begin{itemize}
\item If $i\in I\cap \sigma(I)$,  there exists $a_i\in \fS_1$ such that $\varphi(\delta_{i-1})=a_i\delta_{i}$. As $\deg_{i}(\fL)=1-\deg_i(H)=0$, we have $a_i\in \fS_1^{\times}$. In particular, we have $w_i(G)=0$ by \ref{part-hodge}.

\item  If $i\in I^c\cap \sigma(I)$, there exists $c_i\in \fS_1$ such that $\varphi(\delta_{i-1})=u^ec_i\epsilon_i$. In particular, $w_i(G)=1$ by \ref{part-hodge}.

 \item If $i\in I^c\cap \sigma(I^c)$, there exists $d_i\in \fS_1$ such that $\varphi(\epsilon_{i-1})=u^ed_{i}\epsilon_i$. As $(1\otimes \varphi)\varphi^*(\fM_{i-1})=\fS_1\delta_i+E(u)\fS_1\epsilon_i$, we see that if  $\varphi(\delta_{i-1})=a_i\delta_i+u^ec_i\epsilon_i$, then $a_i$ is a unit in  $\fS_1$. In particular, $w_i(G)=0$.
   \end{itemize}
 This prove statement (1).

 For (2), it follows from Lemma \ref{lem-two-cyclic} that
$$p^{g-1}(1+\deg_i(H'))=p^{g-1}(\deg_i(H)+\deg_i(H'))\leq \frac{p^g-1}{p-1}\quad \text{for } i\in I,$$ whence statement (2). For (3), we note first (2) implies $I'\subset I^c$, or equivalently $I\subset I'^c$. We have to show that if $I'\neq I^c$, then $I=I'=\emptyset$.  Let $i\in I^c\cap I'^c$. If  $i-1$ were in $ I\subset I'^c$, then  $i\in \sigma(I)\cap I^c$ and (1) would imply that $w_i(G)=1$. But we have also $i\in \sigma(I'^c)\cap I'^c$, so (1) applied to $H'$ implies that $w_i(G)=0$. This is a contradiction,  hence $i-1\in I^c$. In the same way, we have   $i-1\in I'^c$. Repeating this argument, we see that  $\Z/g\Z=I^c\cap I'^c$, i.e. $I=I'=\emptyset$. Note that the natural map $f: H\times H'\ra G$ is an isomorphism over the generic fibers. If $I'=I^c$, then $\deg(H)+\deg(H')=\deg(G)$. Therefore, $f$ is an isomorphism by \cite[Cor. 3]{Fa}. The second part of (3) follows directly from (2).

(b) We may  assume $\varphi$ is given by the matrices \eqref{equ-kisin}
such that $\bem a_i &b_i\\ c_i&d_i\enm$ is invertible for $i\in \Z/g\Z$. Under the assumption of statement (b), Lemma \ref{part-hodge} implies that $a_i\in \fS_1^{\times }$ for $i\in \sigma(I)\cap I$, $b_i, c_i\in \fS_1^{\times}$ if $i\in (\sigma(I)\cap I^c)\cup \sigma(I^c)$. Up to modifying the basis vectors, we may assume $\delta_i=\varphi(\delta_{i-1})$ if $i\in \sigma(I)\cap I$, and $\delta_i=\varphi(\epsilon_{i-1})$ if $i\in (\sigma(I)\cap I^c)\cup \sigma(I^c)$. Then the matrices of $\varphi$ can be simplified as
\begin{align*}
\varphi(\delta_{i-1},\epsilon_{i-1})&=(\delta_i,\epsilon_i)\bem1 & b_i\\0& u^e d_i\enm \quad \text{if } i\in \sigma(I)\cap I;\\
\varphi(\delta_{i-1},\epsilon_{i-1})&=(\delta_i,\epsilon_i)\bem a_i &1\\ u^e c_i &0\enm \quad \text{if } i\in (\sigma(I)\cap I^c)\cup \sigma(I^c).
\end{align*}
We write $a_i=u^e a_i'$ if $i\in \sigma(I)\cap I^c$.
The existence of $H$ is equivalent to the existence of a Breuil-Kisin submodule $\fL=\oplus_{i\in \Z/g\Z}\fL_i$ of $\fM$ such that $\deg_i(\fL)=0$ if $i\in I$, and $\deg_i(\fL)=1$ if $i\in I^c$. By the discussion in (a), we may assume $\fL_i=(\delta_i+u^e x_i\epsilon_i)\fS_1$ for $i\in I$ and $\fL_i=(\epsilon_i+x_i\delta_i)\fS_1$ for $i\in I^c$, where the  $x_i$'s are some elements in $\fS_1$ to be determined later.
\begin{itemize}
\item If $i\in \sigma(I)\cap I$, then $\varphi(\delta_{i-1}+u^e x_{i-1}\epsilon_{i-1})=(1+u^{ep}x_{i-1}^pb_i)\delta_i +u^{ep+1}x_{i-1}^pd_i \epsilon_i$. The condition $\varphi(\fL_{i-1})\subset \fL_i$ implies that
    \[x_i=F_{i}(x_{i-1})=\frac{u^{ep}x^p_{i-1}d_i}{1+u^{ep}x_{i-1}^pb_i}.\]

    \item If $i\in \sigma(I)\cap I^c$, then a similar computation shows that
    \[x_i=F_i(x_{i-1})=\frac{1}{c_i}(u^{e(p-1)}x_{i-1}^p+a_i').\]

    \item If $i\in \sigma(I^c)\subset I$,  we have
    \[x_i=F_{i}(x_{i-1})=\frac{x_{i-1}^pc_i}{1+a_ix_{i-1}^p}.\]
\end{itemize}
By iteration, we have $x_i=F_i\circ F_{i-1}\circ\cdots\circ F_{i-g+1}(x_i)$ for all $i\in \Z/g\Z$.
Since all the functions $F_i$ are contracting for the $u$-adic topology on $\fS_1$, there exists a unique solution for every $x_i$. This proves the existence of $\fL$, whence the special subgroup $H$ of type $I$.

\end{proof}
\begin{rem} Note that the condition in \ref{prop-gen-split}(b) is stronger than the converse of \ref{prop-gen-split}(a)(1): we made the extra assumption that $\sigma(I^c)\subset I$. I don't know whether statement \ref{prop-gen-split}(b) still holds without this assumption.
\end{rem}

An interesting special case of the Proposition is the following

\begin{cor}\label{cor-split-CM} Assume $g$ is even. Let $G$ be a truncated Barsotti-Tate group of level $1$ with RM by $\Z_{p^g}$ over $\cO_K$ with $w_i(G)>0$ for $i\in \Z/g\Z$. Put $I_{+}\subset \Z/g\Z$ be the subset consisting of elements $i\equiv 0\mod 2$, and  $I_{-}=I_{+}^c$. Then $G$ admits a unique special subgroup $H_{+}$ $($resp. $H_{-})$ of type $I_{+}$ $($resp. $I_{-})$ if and only if $w_{i}(G)=1$ for all $i\in I_{-}$ $($resp. for all $i\in I_{+})$. In particular, $G$ admits both special subgroups $H_{+}$ and $H_{-}$ of type $I_{+}$ and $I_{-}$ if and only if $w_{i}(G)=1$ for all $i\in \Z/g\Z$.
\end{cor}

\begin{rem}
In appendix B, we will give a family version of Cor. \ref{cor-split-CM} over the  deformation space of a superspecial $p$-divisible group with RM by $\Z_{p^g}$ (cf. Prop. \ref{prop-window} and Remark \ref{rem-window}).
\end{rem}
Now we focus on the case $g=2$, and we identify $\Z/2\Z\simeq \{1,2\}$. The following Proposition refines the preceding corollary in the case $g=2$.

\begin{prop}\label{prop-split-kisin}
Let $G$ be truncated Barsotti-Tate group of level $1$ over $\cO_K$ with RM by $\Z_{p^{2}}$. Assume that $w_1=w_1(G)$ and $w_2=w_2(G)$ are both $>0$. Let $i\in \Z/2\Z$.

\emph{(a)} We have $w_i=1$ if and only if there exists a unique special subgroup $H\subset G$  of type $\{i+1\}$, i.e. we have $\deg_{i}(H)=0$ and $\deg_{i+1}(H)=1$.

\emph{(b)} If $w_i=1$ and $\frac{p}{p+1}<w_{i+1}\leq 1$, then there exists a unique $\Z_{p^2}$-cyclic subgroup $H'\subset G$ disjoint from the $H$ in \emph{(a)} such that $\deg_i(H')=1-p(1-w_{i+1})$ and $\deg_{i+1}(H')=1-w_{i+1}$.
\end{prop}
\begin{proof} We may assume that $i=1$ to simplify the notation. Statement (a) is a special case of the preceding corollary. If  $w_1=w_2=1$, then (b) follows also from the preceding proposition. In the sequels, we assume that $w_1=1$ and $\frac{p}{p+1}<w_2<1$. Let $\fM=\fM_1\oplus \fM_2$ be the Breuil-Kisin module attached to $G$. Let  $(\delta_j,\epsilon_j)_{j=1,2}$ be an adapted basis of $\fM$ so that $(1\otimes\varphi)\varphi^*(\fM_{j-1})$ is free over $\fS_1$ generated by $\delta_j$ and $u^e\epsilon_{j}$. Let $\fL\subset \fM$ be the Breuil-Kisin module corresponding to the quotient $G/H$ given in  (a). By Lemma \ref{lem-deg-kisin}, we have $\deg_1(\fL)=1-\deg_1(H)=1$ and $\deg_2(\fL)=0$. As in the proof of the preceding proposition, we may assume $\fL_1=\fS_1\epsilon_1$ and $\fL_2=\fS_2\delta_2$. Then  we have
\[\varphi(\delta_1,\epsilon_1)=(\delta_2,\epsilon_2)\begin{bmatrix}a_2 &1\\
u^ec_2 &0\end{bmatrix}\quad\quad \varphi(\delta_2,\epsilon_2)=(\delta_1,\epsilon_1)\begin{bmatrix}0 &1\\ u^ec_1 & 0\end{bmatrix},\]
with $v_u(a_2)=ew_2$ and  $c_1, c_2\in \fS_1^\times$. To prove the existence of $H'$, it suffices to construct a Breuil-Kisin submodule $\fK=\fK_1\oplus \fK_2$ of $\fM$ such that
$\deg_1(\fK)=p(1-w_2)$ and $\deg_2(\fK)=w_2$.

We assume that $\fK_1$ and $\fK_2$ are respectively generated over $\fS_1$ by
\[\xi_1=\delta_1 +u^{e(1-p(1-w_2))}y_1\epsilon_1\quad \text{and}\quad \xi_2=\delta_2+u^{e(1-w_2)}y_2\epsilon_2,\]
and we will prove that there exist $y_1,y_2\in \fS_1$ such that $\fK$ becomes the required  Breuil-Kisin submodule of $\fM$. We have
\[\begin{cases}
\varphi(\xi_1)&=(a_2+u^{ep(1-p(1-w_2))}y_1^p)\delta_2+u^ec_2 \epsilon_2,\\
\varphi(\xi_2)&=u^{ep(1-w_2)}y_2^p\delta_1+u^ec_1\epsilon_1.
\end{cases}\]
In order for $\varphi(\xi_1)\in \fK_2$ and $\varphi(\xi_2)\in \fK_1$, we should have
\begin{eqnarray}
&(a_2+u^{ep(1-p(1-w_2))}y_1^p)y_2=u^{ew_2}c_2\label{form-1}\\
&y_1y_2^p=c_1.\label{form-2}
\end{eqnarray}
As $v_u(a_2)=ew_2$, there exists $\hat{a}_2\in \fS_1^{\times}$ such that $a_2=u^{ew_2}\hat{a}_2$. It follows from the equality \eqref{form-1}  that
\begin{equation}\label{form-3}
y_2=\frac{c_2}{\hat{a_2}+u^{e(p^2-1)(w_2-\frac{p}{p+1})}y_1^p}.
\end{equation}
In view of \eqref{form-2}, we get
\[y_1=\frac{c_1}{c_2^p}(\hat{a}_2^p+u^{ep(p^2-1)(w_2-\frac{p}{p+1})}y_1^{p^2}).\]
Since $w_2>\frac{p}{p+1}$ by assumption, it's easy to see that the above equation admits a unique solution for $y_1\in \fS_1^{\times}$, and so $y_2\in \fS_1^{\times}$ is uniquely determined by \eqref{form-3}. With these solutions for $y_1$ and $y_2$, we see that
\[\begin{cases}
\varphi(\xi_1)&=u^{ew_2}(\hat{a}_2+u^{e(p^2-1)(w_2-\frac{p}{p+1})}y_1^p)\xi_2\\
\varphi(\xi_2)&=u^{ep(1-w_2)}y_2^p\xi_1.
\end{cases}\]
Therefore, we have $\deg_1(\fK)=p(1-w_2)$ and $\deg_2(\fK)=w_2$.

For the uniqueness of $H'$, we assume that $\tilde{H}'$ is a $\Z_{p^2}$-cyclic subgroup scheme of $G$ with $\deg_1(\tilde{H}')=1-p(1-w_2)$ and $\deg_2(\tilde{H}')=1-w_2$. Let $\tilde{\fJ}=\fM/\tilde{\fK}$ be the quotient module of $\fM$ corresponding to $\tilde{H}'$. By the definition of  degree, we have
\[\deg(\tilde{\fJ}_1/(1\otimes\varphi)\varphi^*(\tilde{\fJ}_2))=\deg\bigl(\fM_1/(\tilde{\fK}_1+(1\otimes\varphi)\varphi^*(\fM_2))\bigr)=1-p(1-w_2)>0.\]
Since $\fM_1/(1\otimes\varphi)\varphi^*(\fM_2)$ is generated by the image of $\epsilon_1$ and $\tilde{\fJ}_1$ is free of rank $1$ over $\fS_1$, we see by Nakayama that $\tilde{\fJ}_1=\fS_1\overline{\epsilon}_1$, where $\overline{\epsilon}_1$ denotes the image of $\epsilon_1$ in $\tilde{\fJ}_1$. In the same way, we see that $\tilde{\fJ}_2=\fS_1\overline{\epsilon}_2$. Consider the image $\overline{\delta}_1$ of $\delta_1$ in $\tilde{\fJ}_1$. As $\delta_1\in \varphi^*(\fM_2)$, so  we have
$$\overline{\delta}_1\in \varphi^*(\tilde{\fJ}_2)=u^{e(1-p(1-w_2))}\fS_1\overline{\epsilon}_1,$$
i.e. there exists $\tilde{y}_1\in \fS_1$ such that $\tilde{\xi}_1=\delta_1+u^{e(1-p(1-w_2))}\tilde{y}_1\epsilon_1\in \tilde{\fK}_1$. In the same way, there exists $\tilde{y}_2\in \fS_2$ such that $\tilde{\xi}_2=\delta_2+u^{e(1-w_2)}\tilde{y}_2\in \tilde{\fK}_2$. Then it follows that $\tilde{\xi}_1$ and $\tilde{\xi}_2$ generate respectively $\fK_1$ and $\fK_2$, and $\tilde{y}_1,\tilde{y}_2$ satisfy the same equations as $y_1$ and $y_2$; hence we must have $\tilde{y}_1=y_1$ and $\tilde{y}_2=y_2$ as these equations admit a unique solution. This implies $\tilde{\fK}=\fK$ and hence $\tilde{H}'=H'$.

\end{proof}

\begin{prop}\label{prop-subgp-sg} Let $G$ be a truncated Barsotti-Tate group of level $1$ over $\cO_K$ with RM by $\Z_{p^2}$. Assume that $w_1=w_1(G)>0$ and $w_2=w_2(G)=0$.

 \emph{(a)} If there exists a $\Z_{p^2}$-cyclic closed subgroup scheme $H\subset G$ with $\deg_1(H)=0$ and $ \deg_2(H)\geq \frac{1}{p}$, then we have $w_1=1$.

\emph{(b)} Conversely, if $w_1=1$, then any $\Z_{p^2}$-cyclic subgroup $H\subset G$ has $\deg_1(H)=0$ and $ 1/p\leq \deg_2(H)\leq 1$. More precisely, we have the following two possibilities:

 \begin{enumerate}
 \item There is exactly one such subgroup $H$ with $\frac{p+1}{p^2+1}<\deg_2(H)\leq 1$, and all the other $p^2$ cyclic subgroups $H'\subset G$ satisfy  $\frac{1}{p}\leq \deg_2(H')<\frac{p+1}{p^2+1}$ and
      $$\deg_2(H')=\frac{1}{p^2}(1+p-\deg_2(H)).$$

\item All the $(p^2+1)$ $\Z_{p^2}$-cyclic subgroups $H\subset G$ has $\deg_2(H)=\frac{p+1}{p^2+1}$.

    \end{enumerate}
\end{prop}
\begin{proof} The proof is similar to the preceding Lemma. Note that $w_1>0$ implies that $G$ has no multiplicatitve part. Up to replacing $K$ by a finite extension, we may assume that $p$ divides $e$ and all the $\Z_{p^2}$-cyclic closed subgroup schemes of $G$ are defined over $\cO_K$. Let $\fM=\fM_1\oplus\fM_2$ be the Breuil-Kisin module of $G$. Choose an adapted basis $(\delta_i,\epsilon_i)$ of $\fM_i$, we have
\[\varphi(\delta_1,\epsilon_1)=(\delta_2,\epsilon_2)\begin{bmatrix}a_2 &b_2\\u^e c_2 & u^ed_2\end{bmatrix},\quad
\varphi(\delta_2,\epsilon_2)=(\delta_1,\epsilon_1)\begin{bmatrix}a_1 &b_1\\
u^ec_1&u^ed_1\end{bmatrix}.\]
We have $v_u(a_2)=ew_2=0$ and $v_u(a_1)>0$ by assumption, hence $b_1,c_1\in \fS_1^{\times}$. Up to replacing $\delta_2$ by $\varphi(\delta_1)$ and $\delta_1$ by $\varphi(\epsilon_2)$, we may assume
\[\varphi(\delta_1,\epsilon_1)=(\delta_2,\epsilon_2)\begin{bmatrix}1 & b_2\\ 0 &u^e d_2\end{bmatrix},\quad
\varphi(\delta_2,\epsilon_2)=(\delta_1,\epsilon_1)\begin{bmatrix}a_1 & 1\\ u^e c_1 & 0\end{bmatrix}
\]
with $d_2,c_1\in \fS_1^{\times}$.
Let $H\subset G$ be a $\Z_{p^2}$-cyclic subgroup.  Let $\fN=\fN_1\oplus\fN_2$ be the quotient of $\fM$ corresponding to $H$,
and $\fL=\fL_1\oplus\fL_2$ be the kernel corresponding to $G/H$.

 (a) First, we suppose that  $\deg_1(H)=0$ and $\deg_2(H)\geq \frac{1}{p}$.  According to Lemma \ref{part-hodge}, we have to prove that $v_u(a_1)\geq e$. We have $\deg_1(\fN)=\deg_1(H)=0$ and $\deg_2(\fN)=\deg_2(H)\geq 1/p$ by \ref{lem-deg-kisin}. In particular, the $\fS_1$-module \[\fN_2/(1\otimes\varphi)\varphi^*(\fN_1)=\fM_2/((1\otimes \varphi)\varphi^*(\fM_1)+\fL_2)\]
 is non-zero. As $\fM_2/(1\otimes\varphi)\varphi^*(\fM_1)=(\fS_1/u^e\fS_1)\overline{\epsilon_2}$, we see by Nakayama that the image of $\epsilon_2$ generates $\fN_2$. There exists thus  $y_2\in \fS_1$ such that $\eta_2=\delta_2+y_2\epsilon_2$ is a basis of $\fL_2$ over $\fS_1$. An easy computation shows that $v_u(y_2)=e\deg_2(\fN)\geq e/p$.  Similarly, from $\deg_1(\fN)=0$, we deduce that the image of $\varphi(\epsilon_2)=\delta_1$ generates $\fN_1$. Therefore, there is $y_1\in \fS_1$ such that $\eta_1=\epsilon_1+y_1\delta_1$ forms a basis of $\fL_1$. We have
 \beq\label{equ-c-1}\begin{cases}
\varphi(\eta_1)&=(y_1^p+b_2)\delta_2+u^ed_2\epsilon_2,\\
\varphi(\eta_2)&=(a_1+y_2^p)\delta_1+u^ec_1\epsilon_1.
\end{cases}
\eeq
Since $\varphi(\eta_2)\in \fL_1=\fS_1\eta_1$, we see that
\beq\label{equ-c-2}a_1+p_2^p=y_1u^ec_1.\eeq
Since $v_u(y_2)\geq e/p$, we have $v_u(a_1)\geq e$.

(b) Now we suppose $w_1(G)=1$. By \ref{part-hodge}, we can write $a_1=u^ea_1'$ with $a_1'\in \fS_1$.  We claim first that $\fL_1$ contains no element of the form $\delta_1+x_1\epsilon_1$ with $x_1\in u\fS_1$. Indeed, if this were not the case, we would have
\[\varphi^2(\delta_1+x_1\epsilon_1)=u^e[(1+x_1^{p^2}b_2^p)a_1'+u^{e(p-1)}d_2^p]\delta_1+u^e(1+x_1^{p^2}b_2^p)c_1\epsilon_1\in \fL_1.\]
Since $\fL_1$ is free of rank $1$ over $\fS_1$, we should have
\[(1+x_1^{p^2}b_2^p)c_1=x_1[(1+x_1^{p^2}b_2^p)a_1'+u^{e(p-1)}d_2^p].\]
This is impossible, since the left hand side is a unit in $\fS_1$ but $x_1\in u\fS_1$. This proves the claim.

To prove $\deg_1(H)=0$, we  suppose conversely that $\deg_1(H)>0$. Then
\[\fN_1/(1\otimes \varphi)\varphi^*(\fN_2)=\fM_1/((1\otimes\varphi)\varphi^*(\fM_2)+\fL_1)\neq 0.\]
Since $\fM_1/\varphi^*(\fM_2)$ is generated by the image of $\epsilon_1$, we see  that the image of $\epsilon_1$ is a basis of $\fN_1$. As $\delta_1=\varphi(\epsilon_2)\in \varphi^*(\fM_2)$, we see that there exists $x_1\in \fS_1$ such that $\eta_1=\delta_1+x_1\epsilon_1\in \fL_1$. But $\eta_1$ must be a basis of $\fL_1$, since $\fL_1$ is a direct summand of $\fL_1$. It follows from $\deg_1(\fN)=\deg_1(H)>0$ that $v_u(x_1)>0$.  By the claim above,  this is absurd.

Next we show that $\deg_2(H)>0$. Otherwise, we would have $\deg_2(\fN)=\deg_2(H)=0$, i.e.
\[\fN_2/(1\otimes\varphi)\varphi^*(\fN_1)=\fM_2/((1\otimes\varphi)\varphi^*(\fM_1)+\fL_2)=0.\]
Since $(1\otimes\varphi)\varphi^*(\fM_1)=\fS_1\delta_2\oplus \fS_1E(u)\epsilon_2$, we see easily that $\fL_2$ contains an element of the form $\epsilon_2+x_2\delta_2$ with $x_2\in \fS_1$. Then we have
\[\varphi(\epsilon_2+x_2\delta_2)=(1+u^ex_2^pa_1')(\delta_1+\frac{x_2^pu^ec_1}{1+u^ex_2^pa_1'}\epsilon_1)\in \fL_1.\]
This contradicts with the claim mentioned above.

In summary, we have proved that $\deg_1(H)=\deg_1(\fN)=0$ and $\deg_2(H)=\deg_2(\fN)>0$. The same arguments as in part (a) show
that there exist  $y_1,y_2\in \fS_1$ such that $\eta_1=\epsilon_1+y_1\delta_1$ and  $\eta_2=\delta_2+y_2\epsilon_2$ form a basis of $\fL_1$ and $\fL_2$. Thus the same equations \eqref{equ-c-1} and \eqref{equ-c-2} are satisfied.
It follows that $v_u(y_2)\geq \frac{1}{p}\min\{v_u(a_1),e\}=\frac{e}{p}$, hence $\deg_2(H)\geq 1/p$. Recall that we have assumed $p|e$. Let's write $y_2=u^{e/p}z_2$. We get from \eqref{equ-c-2} that
\begin{equation*}
y_1=\frac{1}{c_1}(a_1'+z_2^{p}).
\end{equation*}
On the other hand, we deduce from  \eqref{equ-c-1} and $\varphi(\eta_1)\in \fL_2$ that $y_2(y_1^p+b_2)=u^ed_2$. Substituting $y_1$ and $y_2$,  we get finally
\[z_2^{p^2+1}+(b_2c_1^p+(a'_1)^{p})z_2=c_1^pu^{e(1-\frac{1}{p})}d_2.\]
It's clear that the $p^2+1$ roots of $z_2$ correspond to the $(p^2+1)$ $\Z_{p^2}$-cyclic closed subgroup schemes of $G$, and by our hypothesis on $K$, all the roots of $z_2$ are  in $\fS_1$. If $H$ is the closed subgroup corresponding to a root $z_2$, then we have
\begin{align*}\deg_2(H)&=\deg_2(\fN)=1-\deg_2(\fL)\\
&=\min\{1,\frac{1}{e}v_u(y_2)\}=\min\{1,\frac{1}{p}+\frac{1}{e}v_u(z_2)\}.
\end{align*} We put $\alpha=v_u(b_2c_1^p+(a'_1)^p)/e$ and recall that $c_1,d_2\in \fS_1^{\times}$. By considering the Newton polygon of the equation of $z_2$, we deduce that
\begin{itemize}
\item if $\alpha<\frac{p(p-1)}{p^2+1}$, there is exactly one root  $z_2$ with $v_u(z_2)=e(1-\frac{1}{p}-\alpha)$ and all the other $p^2$ roots of $z_2$ with $v_u(z_2)=\frac{e}{p^2}\alpha$;

\item if $\alpha\geq \frac{p(p-1)}{p^2+1}$, all the $p^2+1$ roots of $z_2$ satisfy $v_u(z_2)=e\frac{p-1}{p(p^2+1)}$.
\end{itemize}
Now the proposition follows easily from the correspondence between the $\Z_{p^2}$-cyclic subgroups of $G$ and the roots of $z_2$.
\end{proof}

\section{Goren-Kassaei's Stratification and Canonical Subgroups}

We retake the notation in \ref{notation}.

\subsection{} We recall briefly the Dieudonn\'e theory for HBAV. Let $R$ be a $\kappa$-algebra,  $A$ be a HBAV over $R$ equipped with a prime-to-$p$ polarization. In \cite[Ch. 3]{BBM}, the authors defined the (contravariant) Dieudonn\'e crystal associated with the finite and locally free group scheme $A[p]$ over $R$. We denote by $\bD(A[p])$  the evaluation of this crystal  on the trivial divided power immersion $\Spec(R)\hra \Spec(R)$. This is a locally free $(R\otimes \cO_F)$-module of rank 2 equipped with two natural morphisms of $(R\otimes \cO_F)$-modules
\[F: \bD(A[p])^{(p)}\ra \bD(A[p])\quad\text{and}\quad V:\bD(A[p])\ra \bD(A[p])^{(p)},\]
called respectively the Frobenius and the Verschiebung, where $(\_)^{(p)}$ denotes the base change by the absolute Frobenius endomorhism $F_R:a\mapsto a^{p}$ on $R$.
 We have a decomposition
\[\bD(A[p])=\bigoplus_{\beta\in \bB}\bD(A[p])_{\beta},\]
 where each $\bD(A[p])_{\beta}$ is a locally free $R$-module of rank 2 and $\cO_F$ acts on it via $\chi_{\beta}$.  For each $\beta\in \bB$, we have an exact sequence of $R$-modules
 \[\xymatrix{\cdots\ar[r]^-{V}&\bD(A[p])_{\beta}^{(p)}\ar[r]^{F} &\bD(A[p])_{\sigma\circ\beta}\ar[r]^{V} &\bD(A[p])_{\beta}^{(p)}\ar[r]^-{F}&\cdots.
 }\]
 Let $\omega_{A/R}$ be the modules of invariant differentials of $A$ relative to $R$, and $\Lie(A)$ be the Lie algebra of $A$.  They are both locally free $(R\otimes \cO_F)$-modules of rank 1, and  we have similar decompositions:
 \[\omega_{A/R}=\bigoplus_{\beta\in\bB} \omega_{A/R,\beta}\quad \text{and}\quad \Lie(A)=
 \bigoplus_{\beta\in\bB}\Lie(A)_{\beta}.\]
 For each $\beta\in \bB$, we have the $\beta$-component of the Hodge filtration
 \begin{equation}\label{beta-hodge}
 0\ra \omega_{A/R, \beta}\ra \bD(A[p])_{\beta}\ra \Lie(A)_{\beta}\ra 0.
 \end{equation}
 Here, the quotient is canonically $\Lie(A^\vee)_{\beta}$, but we have identified $\Lie(A)_{\beta}$ with $\Lie(A^\vee)_{\beta}$ using the polarization on $A$.
 Since the Frobenius induces the zero map on differential forms, we have  a commutative diagram
 \begin{equation}\label{diag-Ha}
 \xymatrix{
 \bD(A[p])_{\sigma^{-1}\circ\beta}^{(p)}\ar[rr]^{F}\ar[d]&&\bD(A[p])_{\beta}\ar[d]\\
 \Lie(A)^{(p)}_{\sigma^{-1}\circ\beta}\ar[rr]^{\mathrm{HW_{\beta}}}\ar[rru]^{\iota_{\beta}}&&\Lie(A)_{\beta},
 }\end{equation}
where $\mathrm{HW_{\beta}}$ is the $\beta$-component of the usual Hasse-Witt map. Note that $\iota_{\beta}$ is injective, since it is the case if $R$ is a perfect field and the diagram commutes with any base change. The morphism $\mathrm{HW}_{\beta}$ is just the dual map of the partial Hasse invariant $h_{\beta}$ defined in \eqref{part-Ha}. This fact will be used later to compute  partial Hasse  invariants.

\subsection{} We recall Goren-Oort's stratification of $X_{\kappa}$ defined in \cite{GO}. Let $A$ be a HBAV over a field containing $\kappa$. We put
\[\tau(A)=\{\beta\in \bB\;|\; h_{\beta}(A)=0\},\]
where $h_{\beta}(A)$ is the partial Hasse invariant of $A$ \eqref{part-Ha}.  For any subset $\tau\subset \bB$, let $Z_{\tau}$ be the closed subsubset of $X_{\kappa}$ where $h_{\beta}$ vanishes for any $\beta\in \tau$, \ie, we have
\[Z_{\tau}=\{x\in X_{\kappa}\;|\; \tau\subset \tau(A_x)\},\]
where $A_x$ is the fiber of the universal HBAV at $x$. It's clear that $Z_{\tau'}\subseteq Z_{\tau}$ for any subsets $\tau'\supseteq \tau$.  We put
$$W_{\tau}=Z_{\tau}\backslash \bigcup_{\tau'\supsetneq \tau}Z_{\tau'}=\{x\in X_{\kappa}\;|\; \tau(A_{x})=\tau\}.$$
Goren and Oort showed that $\{W_\tau\}_{\tau\subseteq \bB}$ form a stratification of $X$, and each stratum $W_{\tau}$ is smooth and equidimensional of dimension $g-|\tau|$, where $|\tau|$ denotes the cardinality of $\tau$. We note that $W_{\emptyset}=X_{\kappa}^\ord$, and $W_{\bB}$ is the set of superspecial points of $X_{\kappa}$, \ie, the points where the $p$-divisible group of the corresponding HBAV is isomorphic to a product of $g$ copies of the $p$-divisible group of a supersingular elliptic curves.

\subsection{}\label{strat-Y} Goren and Kassaei defined a similar stratification on $Y_{\kappa}$ in \cite{GK}. If $S$ is a subset of $\bB$, we denote by $\sigma^{-1}(S)$ the subset of $\bB$ formed by $\sigma^{-1}\circ \beta$ for $\beta\in S$,  by  $\sigma(S)$ the subset formed by $\sigma\circ\beta$ for $\beta\in S$, and by $S^{c}$ the complement $\bB\backslash S$. We say a pair $(\varphi,\eta)$ of subsets of $\bB$ is \emph{ admissible} if $\eta\supset \sigma^{-1}(\varphi^c)$ or equivalently $\varphi\supset \sigma(\eta^c)$. For an admissible pair $(\varphi,\eta)$, we have  decompositions
 \[\eta=\sigma^{-1}(\varphi^c)\coprod I \quad\text{and} \quad \varphi=\sigma(\eta^c)\coprod \sigma(I) \]
 with $I=\eta\cap \sigma^{-1}(\varphi)$.

 Admissible pairs of subsets of $\bB$ arise naturally from the points of $Y_{\kappa}$. Let $k$ be a perfect field containing $\kappa$,  $(A, H)$ be a $k$-point of $Y_{\kappa}$,  $f:A\ra B=A/H$ be the natural isogeny, and
\[\bD(f):\bD(B[p])\ra \bD(A[p])\]
be the morphism induced
 on contravariant Dieudonn\'e modules.  For a $k$ vector space $M$, we identify $M$ with $M^{(p)}=M\otimes_{\sigma} k$ by $x\mapsto x\otimes 1$; so  we regard the Frobenius $F$ on the Dieudonn\'e modules as $\sigma$-linear maps, and the Verschiebung as $\sigma^{-1}$-linear. We have a commutative diagram of exact sequences
 \[\xymatrix{
 \cdots\ar[r]^-{F}&\bD(B[p])_{\sigma\circ\beta}\ar[r]^-{V}\ar[d]^{\bD(f)_{\sigma\circ\beta}}&\bD(B[p])_{\beta}
 \ar[d]^{\bD(f)_{\beta}}\ar[r]^-{F}&\bD(B[p])_{\sigma\circ\beta}\ar[d]^{\bD(f)_{\sigma\circ\beta}}\ar[r]^-{V}&\cdots\\
 \cdots\ar[r]^-{F}&\bD(A[p])_{\sigma\circ\beta}\ar[r]^-{V}&\bD(A[p])_{\beta}\ar[r]^-{F}&\bD(A[p])_{\sigma\circ\beta}
 \ar[r]^-{V}&\cdots.
 }\]
 Since $H=\Ker(f)$ is a $(\cO_F/p)$-cyclic subgroup scheme, $\im(\bD(f)_{\beta})\subseteq \bD(A[p])_{\beta}$ is a one-dimensional vector space over $k$. Note that there are two special $k$-lines in $\bD(A[p])_{\beta}$, namely $(\Ker F)_{\beta}=(\im V)_{\beta}$ and $(\im F)_{\beta}=(\Ker V)_{\beta}$, where $(\_)_{\beta}$ means  the $\beta$-component.  We put
\begin{align*}
\varphi(A,H)&=\{\beta\in \bB\;|\; \im(\bD(f))_{\beta}=(\Ker V)_{\beta}=(\im F)_{\beta}\}\\
\eta(A,H)&=\{\beta\in \bB\;|\; \im(\bD(f))_{\beta}=(\Ker F)_{\beta}=(\im V)_{\beta}\}.
\end{align*}
The pair  $(\varphi(A,H),\eta(A,H))$ is then admissible \cite[2.3.3]{GK}. If we denote by $f^t:B\ra A$ the unique isogeny with $f\circ f^t=p\cdot 1_{B}$ and $f^t\circ f=p\cdot 1_A$, then we have \cite[2.3.2]{GK}
 \begin{align}
 \varphi(A,H)&=\{\beta\in \bB\;|\; \Lie(f)_{\sigma^{-1}\circ \beta}=0\},\label{defn-varphi}\\
 \eta(A,H)&=\{\beta\in \bB\;|\; \Lie(f^t)_{\beta}=0\},\nonumber\\
 I(A,H)&=\eta(A,H)\cap \sigma^{-1}(\varphi(A,H))=\{\beta\in \bB\;|\; \Lie(f)_{\beta}=\Lie(f^t)_{\beta}=0\}.\nonumber
 \end{align}
 We call the elements of $I(A,H)$ the \emph{critical indexes}.
 If $x=(A_x,H_x)$ is an arbitrary point of $Y_{\kappa}$, we define respectively $\varphi(x)$ and $\eta(x)$ as $\varphi(A_x\otimes k, H_{x}\otimes k)$ where $k$ is a perfect extension of the residue field $\kappa(x)$.  It's easy to see that the definition is independent of the choice of $k$.

 For an admissible pair $(\varphi, \eta)$, we put
\[
Z_{\varphi,\eta}=\{x\in Y_{\kappa}\;|\; \varphi(x)\supseteq \varphi, \eta(x)\supseteq \eta\}.
\]
 The subset $Z_{\varphi, \eta}$ is closed in $Y_{\kappa}$ \cite[2.5.1]{GK}. Let  $(\varphi',\eta')$ and $(\varphi,\eta)$ be two admissible pairs. We write $(\varphi',\eta')\supseteq (\varphi, \eta)$ if  $\varphi'\supseteq \varphi$ and $\eta'\supseteq \eta$.  We have $Z_{\varphi',\eta'}\subseteq Z_{\varphi, \eta}$ if $(\varphi',\eta')\supseteq (\varphi, \eta)$. We put
\[W_{\varphi,\eta}=Z_{\varphi,\eta}\backslash \bigcup_{(\varphi',\eta')\supsetneq (\varphi,\eta)}Z_{\varphi',\eta'}.\]
Goren and Kasseai show that \cite[2.5.2]{GK}:
\begin{itemize}
\item each $W_{\varphi,\eta}$ is non-empty and its Zariski closure is $Z_{\varphi,\eta}$;

\item the collection $\{W_{\varphi,\eta}\}$ with $(\varphi,\eta)$ admissible forms a stratification of $Y_{\kappa}$;

\item each stratum $W_{\varphi,\eta}$ is smooth and equi-dimensional of dimension $2g-|\varphi|-|\eta|$.

\end{itemize}
Note that there are $3^g$ strata in this stratification of $Y_{\kappa}$.  The relation between the stratifications on $X_{\kappa}$ and $Y_{\kappa}$ recalled above are given as follows. For an admissible pair $(\varphi,\eta)$, we have \cite[2.6.16]{GK}
\[\pi(W_{\varphi,\eta})=\bigcup_{\substack{
(\varphi\cap\eta) \subset\tau\\
\tau\subset [(\varphi\backslash \eta)\cup (\eta\backslash \varphi)]^c}
}W_{\tau};\]
in particular, for any point $x=(A,H)\in W_{\varphi,\eta}$, we have
$(\varphi\cap\eta)\subset \tau(A)\subset[(\varphi\backslash \eta)\cup (\eta\backslash \varphi)]^c $.

\subsection{Local coordinates}\label{loc-par}
Let $P$ be a closed point of $X_{\kappa}$, $\widehat{\cO}_{X, P}$ be the local ring of $X$ at $P$ with the maximal ideal $\m_{P}$, and $\omegab_{\beta, P}$  be the pull-back of $\omegab_{\beta}$ to $\Spec(\widehat{\cO}_{X,P}/p)$. We choose a basis $e_{\beta}$ of $\omegab_{\beta, P}$ for each $\beta\in \bB$. Then the partial Hasse invariant $h_{\beta}: \omegab_{\beta, P}\ra \omegab_{\sigma^{-1}\circ \beta, P}^{(p)}$, where $\omegab_{\sigma^{-1}\circ \beta, P}^{(p)}$ denotes the base change by the absolute Frobenius,  is given by
\[h_{\beta}(e_{\beta})=\overline{t}_{\beta}e_{\sigma^{-1}\circ \beta}^{(p)}\]
 for some $\overline{t}_{\beta}\in \cO_{X_{\kappa},P}/p$.   Note that $\overline{t}_{\beta}\in \overline{\m}_{P}$ if and only if $\beta\in \tau(P)$, where $\overline{\m}_{P}$ denotes the maximal ideal of $\widehat{\cO}_{Y,P}/p$. By Kodaira-Spencer isomorphism, the elements  $\{\overline{t}_{\beta}:\beta\in \tau(P)\}$ form part of a system of regular parameters of the regular local ring $\cO_{X_{\kappa},P}/p$. Let $t_{\beta}\in \widehat{\cO}_{X,P} $ be a lift of $\overline{t}_{\beta}$. Then if $x\in W_{\bB}$, we have $\widehat{\cO}_{X_{\kappa},P}\simeq W(\kappa(P))[[\{t_{\beta}: \beta\in \bB\}]]$.

 Let $Q$ be a closed point of $Y_{\kappa}$, $\widehat{\cO}_{Y,Q}$ be the completion of the local ring of $Y$ at $Q$ with the maximal ideal $\m_Q$. We denote by $(\cA_Q,\cH_Q)$  the base change of the universal object on $Y$ to $\Spec(\widehat{\cO}_{Y,Q})$, and by $f: \cA_{Q}\ra \cB_{Q}=\cA_{Q}/\cH_{Q}$ the canonical quotient, and by $f^t:\cB_Q\ra \cA_Q$ the unique isogeny with $f^t\circ f=p\cdot 1_{\cA_Q}$ and $f\circ f^{t}=p\cdot 1_{\cB_Q}$. Let $\omegab_{\cA_{Q},\beta}$ and $\omegab_{\cB_Q,\beta}$ be respectively the $\beta$-components of the invariants differentials of $\cA_Q$ and $\cB_Q$. They are both free $\widehat{\cO}_{Y,Q}$-modules of rank $1$; we choose generators  $e_\beta\in \omegab_{\cA_Q,\beta} $ and $\epsilon_{\beta}\in \omegab_{\cB_Q,\beta}$. For each $\beta\in \bB$, there are elements  $x_{\beta}, y_{\beta}\in \widehat{\cO}_{Y,Q}$ such that
 \begin{equation}\label{defn-y_beta}
 f^{t*}(e_{\beta})=x_{\beta}\epsilon_{\beta}\quad \text{and}\quad f^*(\epsilon_\beta)=y_{\beta}e_{\beta}.
 \end{equation}
 We have  $x_{\beta}y_{\beta}=p$ because of $f^*\circ f^{t*}=p$, and $x_\beta, y_{\beta}\in \m_{Q}$ if and only if $\beta$ is critical, i.e. $\beta\in I(Q)=\sigma^{-1}(\varphi(Q))\cap \eta(Q)$.
 Actually, Stamm \cite{St} showed that we have an isomorphism
 \beq\label{equ-loc-ring}
 \widehat{\cO}_{Y,Q}\simeq W(\kappa(Q))[[\{x_\beta,y_{\beta}:\beta\in I(Q)\}, \{z_{\gamma}:\gamma\in \bB-I(Q)\}]]/(\{x_{\beta}y_{\beta}-p: \beta\in I(Q)\}).
 \eeq
The following proposition,  proved in \cite[2.8.1]{GK} and referred as ``Key Lemma'' there, will play an important role in the proof of our main results.  We give here another  proof using Dieudonn\'e theory.

\begin{prop}[Goren-Kassaei]\label{prop-key-lemma} Let $Q$ be a closed point of $Y$, $P=\pi(Q)$, $\beta\in \varphi(Q)\cap\eta(Q)\subset \tau(P)$, and $\pi^*:\widehat{\cO}_{X,P}\ra \widehat{\cO}_{Y,Q}$ be the natural induced morphism.

\emph{(a)} If $\sigma\circ\beta\in \varphi(Q)$ and $\sigma^{-1}\circ\beta\in\eta(Q)$, we have
\[\pi^*(t_{\beta})\equiv ux_{\beta}+vy^p_{\sigma^{-1}\circ\beta}\mod p, \]
where $u,v$ are some units in $\widehat{\cO}_{Y,Q}$.

\emph{(b)} If  $\sigma\circ\beta\in \varphi(Q)$ and $ \sigma^{-1}\circ\beta\notin \eta(Q)$, we have
\[\pi^*(t_{\beta})\equiv ux_{\beta} \mod p\] for some unit $u\in\widehat{\cO}_{Y,Q}$.

\emph{(c)} If $\sigma\circ\beta\notin \varphi(Q)$ and $\sigma^{-1}\circ\beta\in \eta(Q)$, we have
\[\pi^*(t_{\beta})\equiv vy^p_{\sigma^{-1}\circ\beta}\mod p\] for some unit $v\in \widehat{\cO}_{Y,Q}$.

\emph{(d)} If $\sigma\circ\beta\notin \varphi(Q)$ and $\sigma^{-1}\circ\beta \notin \eta(Q)$, we have $\pi^*(t_\beta)\equiv 0\mod p$.

\end{prop}

\begin{proof} It suffices to prove the corresponding equalities in $R=\widehat{\cO}_{Y,Q}/p=\widehat{\cO}_{Y_{\kappa}, Q}$. By abuse of notation, we still denote by  $t_{\beta}$, $x_{\beta}$ and $y_{\beta}$ their image in $R$. Let $(\cAb,\cHb)$ be the universal HBAV over $\Spec(R)$, $f:\cAb\ra \cBb=\cAb/\cHb$ and $f^t:\cBb\ra\cAb$ be the canonical isogenies,  $(\cAb_0,\cHb_0)$, $f_0:\cAb_0\ra \cBb_0$ and  $f^t_0:\cBb_0\ra \cAb_0$  be the corresponding fibers at the closed point. We have the Dieudonn\'e modules $\bD(\cAb[p])=\oplus_{\gamma\in\bB} \bD(\cAb[p])_{\gamma}$ and $\bD(\cBb[p])=\oplus_{\gamma\in\bB}\bD(\cBb[p])_{\gamma}$, and a commutative diagram of exact sequences
\begin{equation}\label{diag-DF}\xymatrix{
\cdots\ar[r]&\bD(\cBb[p])_{\sigma^{-1}\circ\gamma}\ar[rr]^{\bD(f)_{\sigma^{-1}\circ\gamma}}\ar[d]^F
&&\bD(\cAb[p])_{\sigma^{-1}\circ\gamma}\ar[rr]^{\bD(f^t)_{\sigma^{-1}\circ\gamma}}\ar[d]^{F}&&\bD(\cBb[p])_{\sigma^{-1}\circ\gamma}\ar[r]\ar[d]^F& \cdots\\
\cdots\ar[r]&\bD(\cBb[p])_{\gamma}\ar[rr]^{\bD(f)_{\gamma}}&& \bD(\cAb[p])_{\gamma}\ar[rr]^{\bD(f^t)_{\gamma}}&&\bD(\cBb[p])_{\gamma}\ar[r]&\cdots
}\end{equation}
for each $\gamma\in \bB$. Let $\{e_{\gamma},d_{\gamma}\}$ and $\{\epsilon_{\gamma},\delta_{\gamma}\}$ be respectively a basis of $\bD(\cAb[p])_{\gamma}$ and $\bD(\cBb[p])_{\gamma}$ such that the cotangent modules $\omega_{\cAb/R}$ and $\omega_{\cBb/R}$ are generated respectively by $e_\gamma$ and $\epsilon_{\gamma}$. So we have $F(e_\gamma)=0$ and $F(\epsilon_\gamma)=0$. By \eqref{defn-y_beta}, we  may assume
\begin{equation}\label{equ-x-beta}
\bD(f^t)(e_{\gamma})=x_{\gamma}\cdot \epsilon_{\gamma} \quad \text{and} \quad \bD(f)(\epsilon_{\gamma})=y_{\gamma}\cdot e_{\gamma}.
\end{equation}
 For an element $z$ in  $\bD(\cAb[p])$ (or $\bD(\cBb[p])$), we denote by $z_0$ its image in $\bD(\cAb_0[p])$ (or $\bD(\cBb_0[p])$). If $\gamma\in \eta(Q)$, then we have
 \[\Ker\bD(f^t_0)_{\gamma}=\im\bD(f_0)_{\gamma}=(\im V)_{\gamma}=\kappa(Q)e_{\gamma,0}.\]
 Up to modifying $\delta_{\gamma}$, we may assume $\bD(f_0)(\delta_{\gamma,0})=e_{\gamma, 0}$ and even $\bD(f)(\delta_{\gamma})=e_{\gamma}+U_{\gamma} d_{\gamma}$
 for some $U_{\gamma}\in \m_R$. If $\gamma\in \sigma^{-1}(\varphi(Q))$, it follows from \eqref{defn-varphi} that
 $$
\Ker\bD(f_0)_{\gamma}=\im\bD(f^t_0)_{\gamma}=(\im V)_{\gamma}=\kappa(Q)\epsilon_{\gamma,0}.
$$  Hence, we may assume  $\bD(f^t_0)(d_{\gamma, 0})=\epsilon_{\gamma,0}$ and therefore $\bD(f^t)(d_{\gamma})=\epsilon_{\gamma}+V_{\gamma} \delta_{\gamma}$ for some $V_\gamma\in \m_R$. If $\gamma\in I(Q)=\eta(Q)\cap \sigma^{-1}(\varphi(Q))$, \ie, $\gamma$ is critical, we deduce from $\bD(f)\circ\bD(f^t)=0$ that
\begin{align*}
\bD(f)(\epsilon_{\gamma})&=-V_{\gamma}\bD(f)(\delta_{\gamma})=-V_{\gamma}(e_{\gamma}+U_{\gamma}d_{\gamma})\\
\bD(f^t)(e_{\gamma})&=-U_{\gamma}\bD(f^t)(d_{\gamma})=-U_{\gamma}(\epsilon_\gamma+V_{\gamma}\delta_{\gamma}).
\end{align*}
In view of \eqref{equ-x-beta}, we have $U_\gamma=-x_{\gamma}$ and $V_{\gamma}=-y_{\gamma}$. In summary, if $\gamma$ is critical, we have
\begin{equation}\label{equ-D-delta}
\bD(f)(\delta_{\gamma})=e_{\gamma}-x_{\gamma}d_{\gamma}\quad \text{and}\quad\bD(f^t)(d_{\gamma})=\epsilon_{\gamma}-y_{\gamma}\delta_{\gamma}.
\end{equation}

  Let $\beta\in \bB$ be in the statement of the proposition. Assume that
  \begin{equation}\label{equ-F-d}
  \begin{cases}F(d_{\sigma^{-1}\circ\beta})&=-ue_{\beta}+t_{\beta}d_{\beta}\\
  F(\delta_{\sigma^{-1}\circ\beta})&=-v\epsilon_{\beta}+s_{\beta}\delta_{\beta}\end{cases}
  \end{equation}
  for some $u,v,t_{\beta},s_{\beta}$ in $R$.
   By the remark below \eqref{diag-Ha}, $t_{\beta},s_{\beta}$ compute the partial Hasse invariants of $\cAb$ and $\cBb$ respectively. Note that $t_{\beta}\in \m_R$ since $\beta\in \varphi(Q)\cap\eta(Q)\subset \tau(Q)$ by assumption. Hence $u$ has to be a unit in $R$, because $(\im F)_{\beta}$ is a direct summand of $\bD(\cAb[p])_{\beta}$. Similarly, at least one of $v$ and $s_{\beta}$ is invertible in $R$. We distinguish the four cases in the statement:

   \textbf{Case (a).} In this case, both $\beta$ and $\sigma^{-1}\circ \beta$ are critical, hence the formula  \eqref{equ-D-delta} applies to $\gamma=\beta,\sigma^{-1}\circ \beta$. It results from \eqref{equ-F-d} that
\[
\begin{cases}
\bD(f^t)(F(d_{\sigma^{-1}\circ \beta}))&=(-ux_{\beta}+t_{\beta})\epsilon_{\beta}-t_{\beta}y_{\beta}\delta_{\beta}\\
\bD(f)(F(\delta_{\sigma^{-1}\circ \beta}))&=(-vy_{\beta}+s_{\beta})e_{\beta}-s_{\beta}x_{\beta}d_{\beta}.
\end{cases}
\]
On the other hand, it follows from the commutative diagram \eqref{diag-DF} that
\begin{equation}\label{equ-com-F-D}
\begin{cases}
\bD(f^t)(F(d_{\sigma^{-1}\circ\beta}))&=F(\bD(f^t)(d_{\sigma^{-1}\circ \beta}))=-y^p_{\sigma^{-1}\circ\beta}(-v\epsilon_{\beta}+s_{\beta}\delta_{\beta})\\
\bD(f)(F(\delta_{\sigma^{-1}\circ\beta}))&=F(\bD(f)(\delta_{\sigma^{-1}\circ\beta}))=-x^{p}_{\sigma^{-1}\circ\beta}(-u e_{\beta}+t_{\beta}d_{\beta}).
\end{cases}
\end{equation}

Comparing the coefficients of $\epsilon_{\beta}$ and $e_{\beta}$, we get
\[\begin{cases}t_{\beta}&=ux_{\beta}+vy^p_{\sigma^{-1}\circ\beta}\\
s_{\beta}&=v y_{\beta}+u x^p_{\sigma^{-1}\circ\beta}.\end{cases}
\]
We see that $s_{\beta}\in \m_R$, and it follows that  $v$ is a unit in $R$ as remarked above. This completes the proof in case (a).

 \textbf{Case (b).} In this case,  $\beta$ is critical.  The fact  $ \sigma^{-1}\circ\beta\notin \eta(Q)$ implies that
\[\im \bD(f_0)_{\sigma^{-1}\circ \beta}\neq (\im V)_{\sigma^{-1}\circ \beta}=\kappa(Q)e_{\sigma^{-1}\circ \beta}.\]
Therefore, up to modifying $d_{\sigma^{-1}\circ \beta}$, we may assume $d_{\sigma^{-1}\circ \beta}\in \im\bD(f)=\Ker\bD(f^t)$.
 Since $F$ commutes with $\bD(f^t)$, we have
\[0=F(\bD(f^t)(d_{\sigma^{-1}\circ\beta}))=\bD(F(d_{\sigma^{-1}\circ\beta}))=\bD(f^t)(-ue_{\beta}+t_{\beta}d_{\beta}).\]
Now the equality $t_\beta=ux_{\beta}$ follows from \eqref{equ-D-delta} applied to $\gamma=\beta$.

 \textbf{Case (c).} In this case,  $\sigma^{-1}\circ \beta$ is critical. The assumption  $\sigma\circ\beta\notin \varphi(Q)$ implies that
\[\Ker\bD(f_0)_{\beta}\neq (\im V)_{\beta}=\kappa(Q)\epsilon_{\beta,0}.\]
 It follows that $\bD(f_0)(\epsilon_{\beta,0})=y_{\beta}e_{\beta,0}\neq 0$, i.e. $y_{\beta}$ is invertible and $x_{\beta}=py^{-1}_{\beta}=0$ in $R$. Therefore, $\Ker \bD(f^t)_{\beta}=\im \bD(f)_{\beta}=R e_{\beta}$. Up to modifying $\delta_{\beta}$, we may assume $\bD(f^t)(d_{\beta})=\delta_{\beta}$. As in case (a), it follows from \eqref{equ-F-d} that
 \[
 \begin{cases}
\bD(f^t)(F(d_{\sigma^{-1}\circ\beta}))&=\bD(f^t)(-ue_{\beta}+t_{\beta}d_{\beta})=t_{\beta}\delta_{\beta}\\
\bD(f)(F(\delta_{\sigma^{-1}\circ \beta}))&=\bD(f)(-v\epsilon_{\beta}+s_{\beta}\delta_{\beta})=-vy_{\beta} e_{\beta}.
 \end{cases}
 \]
Using the fact that $F$ commutes with $\bD(f)$ and $\bD(f^t)$ and that $\sigma^{-1}$ is critical, we get the same formula \eqref{equ-com-F-D} as in case (a). Comparing the coefficients, we obtain
\[
v=-u y_{\beta}^{-1}x^p_{\sigma^{-1}\circ\beta}\quad \text{and}\quad
t_{\beta}=-y^p_{\sigma^{-1}\circ \beta}s_{\beta}.
\]
To complete the proof in case (c), we note that $v\in \m_R$, hence $s_{\beta}$ is invertible in $R$.

\textbf{Case (d).} The same argument as in case (b) shows that $t_{\beta}=ux_{\beta}$ (we didn't use the fact that $\beta$ is critical to get this). Now the same argument as in case (c) shows that $x_\beta=0$.

\end{proof}

\subsection{}\label{sect-val-X} We recall the  valuations  on the rigid spaces $\fX_{\rig}$ and $\fY_{\rig}$ defined by Goren-Kassaei \cite[4.2]{GK}. Let $\C_p$ be the completion of an algebraic closure of $\Q_\kappa$, and $v_p$ be the valuation on $\C_p$ normalized by $v_p(p)=1$. We define
\[\nu(x)=\min \{v_p(x),1\}.\]
Let $K$ be a finite extension of $\Q_{\kappa}$,  $\cO_K$ be its ring of integers, and $P$ be a  $K$-valued rigid point of $\fX_{\rig}$, i.e. $P$ corresponds to a  polarized HBAV $A$ with $\Gamma_1(N)$-level structure over $\cO_K$. For any $\beta\in \bB$, let
$t_{\beta}$ be a local lift of the $\beta$-partial Hasse invariant around $P$. We define the \emph{$\beta$-th partial Hodge height} of $P$ (or of $A$) to be
\begin{equation}\label{equ-part-hodge}
w_{\beta}(P)=w_{\beta}(A)=\nu(t_{\beta}(P)).
\end{equation}
 It's easy to see that the definition does not depend on the lift  $t_{\beta}$. We have  $w_{\beta}(P)>0$ if and only if $\beta\in \tau(\Pb)$, where $\Pb\in Y_{\kappa}$ is the specialization of $P$. Therefore  $P\in \fX_{\rig}^{\ord}$ if and only if $w_{\beta}(P)=0$ for all $\beta\in \bB$. Let $\fp$ be a prime ideal of $\cO_{F}$ dividing $p$, and $\cO_{F_{\fp}}$ be the completion  of $\cO_F$ of its localization at $\fp$. Note that $\bB_{\fp}\subset \bB$ is identified with the set of embeddings of $\cO_{F_{\fp}}$ into $\cO_K$. Then the finite flat group scheme $A[\fp]$ is a truncated Barsotti-Tate group of level $1$ with RM by $\cO_{F_{\fp}}$ in the sense of \ref{defn-gp-RM}, and the $\beta$-th partial Hodge height of $A$ coincides with that of $A[\fp]$ defined in \ref{sect-hodge}.

  \subsection{Partial degrees}\label{sect-mult-deg} Let $K$ be  as above,  $G$ be a finite flat group scheme over $\cO_K$ equipped with an action of $\cO_F$, and $\omega_{G}$ be the invariant differential module of $G$. Similarly as \ref{defn-deg}, we define, for each $\beta\in \bB$, the \emph{$\beta$-degree} of $G$ to be \[\deg_{\beta}(G)=\deg(\omega_{G,\beta}),\]
  where $\omega_{G,\beta}$ is the direct summand of $\omega_G$ on which $\cO_F$ acts via $\chi_{\beta}$. The ``total'' degree of $G$ defined in \cite{Fa} is thus
  \[\deg(G)=\sum_{\beta\in \bB}\deg_{\beta}(G).\]
   Let  $Q=(A,H)$ be a $K$-valued rigid point of $\fY_{\rig}$.    We put $\nu_{\fY}(Q)=(\nu_{\beta}(Q))_{\beta\in \bB}$ with
   \begin{equation}\label{defn-val-Y}
   \nu_{\beta}(Q)=\deg_{\beta}(H).
   \end{equation}
    This definition is slightly different from that in \cite[4.2]{GK}, and their relationship is given by the following
 \begin{prop}\label{prop-val} Let $Q=(A,H)\in \fY_{\rig}$ be a rigid point defined over a finite extension $L/\Q_{\kappa}$, and $\Qb\in Y_{\kappa}$ be its specialization. Then our definition of $\nu_{\beta}(Q)$ is $1$ minus that of Goren-Kassaei, i.e., we have
 \[\nu_{\beta}(Q)=\begin{cases}
0 &\text{if }\beta\in \eta(\Qb)\backslash I(\Qb),\\
\nu(y_{\beta}(Q))& \text{if } \beta\in I(\Qb),\\
1 &\text{if } \beta\notin \eta(\Qb),\end{cases}\]
 where $y_{\beta}$ is the local parameter around $Q$ introduced in \eqref{defn-y_beta}.

 \end{prop}
\begin{proof} Let $f:A\ra B=A/H$ be the canonical isogeny, and $f^t: B\ra A$ be the isogeny with kernel $A[p]/H$. We have exact sequences of invariant differential modules:
 \[0\ra \omega_{B}\xra{f^*}\omega_{A}\lra\omega_{H}\ra 0,\quad\quad 0\ra \omega_{A}\xra{f^{t*}}\omega_{B}\lra \omega_{A[p]/H}\ra0. \]
 So by the definitions of $x_{\beta},y_{\beta}$, we have
 \begin{align*}\deg_{\beta}(A[p]/H)&=\deg(\omega_{B,\beta}/f^{t*}\omega_{A,\beta})=v_p(x_p(Q)),\\
 \deg_{\beta}(H)&=\deg(\omega_{A,\beta}/f^*\omega_{B,\beta})=v_p(y_p(Q)).
 \end{align*}
 Thus the case where $\beta\in I(\Qb)$ follows immediately. If $\beta\in \eta(\Qb)\backslash I(\Qb)$, we have  $\Lie(f_{\Qb})_{\beta}\neq0$, where $f_{\Qb}$ denotes the special fiber of $f$. It follows that $f^*_{\beta}:\omega_{B,\beta}\ra \omega_{A,\beta}$ is surjective, hence $\nu_{\beta}(Q)=\deg_{\beta}(H)=0.$ If $\beta \notin \eta(\Qb)$, then we have $\beta\in \eta^c(\Qb)\subset \sigma^{-1}(\varphi(\Qb))$, i.e., $\Lie(f^{t}_{\Qb})_{\beta}\neq 0$ by \eqref{defn-varphi}. This means $f^{t*}_{\beta}:\omega_{A,\beta}\ra \omega_{B,\beta}$ is surjective, hence we have $\deg_{\beta}(A[p]/H)=0$ and $$\nu_{\beta}(Q)=\deg_{\beta}(H)=\deg_{\beta}(A[p])-\deg(A[p]/H)=1.$$
\end{proof}

It follows from this Proposition that $Q\in \fY_{\rig}^{\ord}$ if and only if $\nu_{\beta}(Q)=1$ for all $\beta\in \bB$.
 Recall that for every prime ideal $\fp$ of $\cO_F$ above $p$, $\bB_{\fp}$ is the subset of $\bB=\Hom_{\Q}(F,\Q_{\kappa})$ which induces the valuation on $F$ corresponding to $\fp$. Following \cite[5.3]{GK}, we define an admissible open subset of $\fX_{\rig}$ or of $\fY_{\rig}$ by
 \begin{align}
 \cU_{\fp}&=\{P\in \fX_{\rig}\;|\; w_{\beta}(P)+p w_{\sigma^{-1}\circ\beta}(P)<p, \forall \beta\in \bB_{\fp}\},\nonumber\\
 \cV_{\fp}&=\{Q\in \fY_{\rig}\;|\; \nu_{\beta}(Q)+p\nu_{\sigma^{-1}\circ\beta}(Q)>1, \forall \beta\in \bB_{\fp}\},\label{defn-p-can}\\
 \cW_{\fp}&=\{Q\in \fY_{\rig}\;|\; \nu_{\beta}(Q)+p\nu_{\sigma^{-1}\circ\beta}(Q)<1, \forall \beta\in \bB_{\fp}\}.
\end{align}

\begin{thm}[Goren-Kassaei]\label{thm-GK-p}
\emph{(a)} For every prime ideal $\fp$ of $\cO_F$ above $p$, we have
$\pi^{-1}(\cU_{\fp})=\cV_{\fp}\cup\cW_{\fp}.$

\emph{(b)} Let $L$ be a finite extension of $\Q_{\kappa}$, and $P\in \cU_{\fp}(L)$ be rigid point corresponding to a HBAV $A$ over $\cO_L$. For every rigid point $Q=(A,H)$ of $\cV_{\fp}$ above $P$, the $\fp$-component $H[\fp]$ of $H$ is the canonical subgroup of $A[\fp]$ given by Theorem \ref{thm-can}.
\end{thm}

This theorem is essentially [\emph{loc. cit.} 7.1.3], and statement (b) is also a direct consequence of Theorem \ref{thm-can}. Let $Q=(A,H)$ be a rigid point of $\fY_{\rig}$, and  $H[\fp]$ be the subgroup killed by $\fp\subset \cO_F$, so that $H=\prod_{\fp|p}H[\fp]$. Following [\emph{loc. cit.} 5.4.1], we say $H$  (or $Q$) is
 \begin{itemize}
 \item  \emph{canonical at $\fp$}  if $Q\in \cV_{\fp}$;

 \item   \emph{anti-canonical at $\fp$} if $Q\in \cW_{\fp}$;

 \item    \emph{canonical} if it is canonical at   all primes  $\fp$ above $p$;

 \item  \emph{anti-canonical} if it is anti-canonical at $\fp$ at all primes $\fp$ above $p$;

 \item \emph{too singular at $\fp$} is it is neither canonical nor anti-canonical at $\fp$.
 \end{itemize}

We put
 \begin{align}
 \cU_{\can}&=\bigcap_{\fp|p}\cU_{\fp},\quad \quad \cV_{\can}=\bigcap_{\fp|p}\cV_{\fp},\nonumber\\
 \cW&=\bigcup_{\emptyset \neq S\subseteq \{\fp|p\}}\bigl[\bigcap_{\fp\in S}\cW_{\fp}\cap\bigcap_{\fp\notin S}\cV_{\fp} \bigr]\label{anti-can}
 \end{align}
   Then $\cU_{\can}$ and $\cV_{\can}$ are respectively strict neighborhoods of $\fX_{\rig}^{\ord}$ and $\fY_{\rig}^\ord$. The following theorem is a consequence of Theorem \ref{thm-GK-p}.
 \begin{thm}\cite[5.3.1, 5.3.7]{GK}\label{thm-GK}
 With the notation above, we have
 $\pi^{-1}(\cU_{\can})=\cV_{\can}\cup \cW$, and the restriction $\pi|_{\cV_{\can}}:\cV_{\can}\ra \cU_{\can}$ is an isomorphism, i.e., there exists a section $s^{\dagger}:\cU_{\can}\ra \cV_{\can}$ extending the section $s^{\circ}:\fX_{\rig}^{\ord}\ra \fY_{\rig}^{\ord}$ defined in \ref{sect-ord}.
 \end{thm}

 We call $\cV_{\can}$ (resp. $\cU_{\can}$) the \emph{canonical locus}  of $\fY_{\rig}$ (resp. $\fX_{\rig}$). The following Proposition describes the dynamics of Hecke correspondence over the canonical and anti-canonical locus.

 \begin{prop}[Goren-Kassaei]\label{prop-hecke-can} Let $\fp$ be a prime ideal of $\cO_F$ dividing $p$, $U_{\fp}$ be the set theoretic Hecke correspondence \eqref{defn-set-U_p} on $\fY_{\rig}$, and  $Q=(A,H)\in \fY_{\rig}$ be a rigid point.

\emph{(a)} Assume that $Q$ is canonical at $\fp$. Then for every $(\cO_F/\fp)$-cyclic subgroup $H'$ of $A[\fp]$ distinct from   $H[\fp]$, we have
$$\deg_{\beta}(H')=\frac{1}{p}(1-\nu_{\sigma\circ\beta}(Q))\quad \text{for }\beta\in \bB_{\fp};$$
or equivalently all $Q_1\in U_{\fp}(Q)$  are canonical at $\fp$, and  we have $\nu_{\beta}(Q_1)=\frac{p-1}{p}+\frac{1}{p}\nu_{\sigma\circ\beta}(Q)$ for $\beta\in \bB_{\fp}$.

\emph{(b)} Assume $Q$ is anti-canonical at $\fp$. Let $C\subset A[\fp]$ be its canonical subgroup. For any $(\cO_F/p)$-cyclic subgroup $H'\subset A[\fp]$ distinct from $H[\fp]$, we have
$$\deg_{\beta}(H')=\begin{cases}1-p\nu_{\sigma^{-1}\circ\beta}(Q) &\text{if } H'=C;\\
\nu_{\beta}(Q)&\text{if }H'\neq C\end{cases}$$ for all $\beta\in \bB_{\fp}$.
Equivalently, if $Q_1=(A/H', (H+H')/H')\in U_p(Q) $, we  have
\[\nu_\beta(Q_1)=\begin{cases}p\nu_{\sigma^{-1}\circ \beta}(Q)&\text{if $H'=C$;}\\
1-\nu_{\beta}(Q)&\text{if $H'\neq C$.} \end{cases}\]
In particular, if $H'\neq C$, $Q_1$ is canonical at $\fp$.

\end{prop}

This is essentially contained in \cite[5.4.3]{GK}. For the convenience of the reader, we reproduce its proof here.

\begin{proof} The equivalence between the statement on $H'$ and that on $Q_1\in U_{\fp}(Q)$ follows from the fact that
\[\nu_{\beta}(Q_1)=\deg_{\beta}(A[\fp]/H')=1-\deg_{\beta}(H'),\]
since $\deg_{\beta}(A[\fp])=1$ for all $\beta\in \bB_{\fp}$.

(a) Since $Q$ is canonical at $\fp$, we have $1-\nu_{\beta}(Q)<p\nu_{\sigma^{-1}\circ\beta}(Q)$ for all $\beta\in \bB_{\fp}$.
 It follows thus from Prop. \ref{prop-key-lemma} and \ref{prop-val} that
\[w_{\beta}(A)=\nu(x_{\beta}(Q))=1-\nu_{\beta}(Q),\]
where $x_{\beta}$ is the local parameters on $Y$ introduced in \ref{loc-par}.
The subgroup $H'$ must be anti-canonical at $\fp$  by Theorem \ref{thm-GK-p}, i.e., we have
$p\deg_{\sigma^{-1}\circ\beta}(H')<1-\deg_{\beta}(H')$
for $\beta\in \bB_{\fp}$.
It follows from \ref{prop-key-lemma} that
$$
w_{\beta}(A)=p\deg_{\sigma^{-1}\circ\beta}(H').
$$
Hence $\deg_{\sigma^{-1}\circ\beta}(H')=\frac{1}{p}(1-\nu_{\beta}(Q))$, i.e. $\deg_{\beta}(H')=\frac{1}{p}(1-\nu_{\sigma\circ\beta}(Q))$.

(b) We proceed in the same way as in (a). Since $Q$ is anti-canonical at $\fp$, we have $w_{\beta}(A)=p\nu_{\sigma^{-1}\circ\beta}(Q)$ for $\beta\in \bB_{\fp}$. If $H'=C$, we have $w_{\beta}(A)=1-\deg_{\beta}(H')$. Thus the equality
\[
\deg_{\beta}(H')=1-p\nu_{\sigma^{-1}\circ\beta}(Q)
\]
follows immediately.
If $H'\neq C$, then $H'$ is anti-canonical at $\fp$ by Theorem \ref{thm-GK-p}, so we have $w_{\beta}(A)=p\deg_{\sigma^{-1}\circ\beta}(H')$ for $\beta\in \bB_{\fp}$.  We deduce immediately that $\deg_{\beta}(H')=\nu_{\beta}(Q)$.
\end{proof}

Recall that for any weight $\vk\in \Z^{\bB}$ and any admissible open subset $U\subset \fY_{\rig}$,   we have defined  in \ref{norms} $|f|_U$ for  the space  $f\in H^0(U,\omegab^{\vk})$. Note that if $U$ is not quasi-compact, it's possible that $|f|_U=\infty$. We have the following basic estimation of norms under Hecke correspondence.

\begin{lemma}\label{lemma-basic-est}
Let $Q=(A,H)$ be a rigid point of $\fY_{\rig}$ defined over a finite extension $K$ of $\Q_{\kappa}$, and $\omega$ be a basis of the free $(\cO_K\otimes \cO_F)$-module $\omegab_{A/\cO_{K}}$. Let $\fp$ be a prime ideal of $\cO_F$ dividing $p$, $H'\subset A[\fp]$ be a $(\cO_F/\fp)$-cyclic closed group disjoint from $H$, and $\hat{\phi}:A/H'\ra A$ be the canonical isogeny with kernel $A[p]/H'$.  Let $V$ be an admissible open subset  containing the rigid point $Q'=(A/H',(H+H')/H')$. If $f$ is  a section of $\omegab^{\vk}$ over  $V$ such that $|f|_{V}$ is finite, we have
\[|f(Q')|=|f(A/H',(H+H')/H', p^{-1}\hat{\phi}^*\omega)|\leq p^{-\sum_{\beta\in \bB_{\fp}}k_{\beta}\deg_{\beta}(H')}\;|f|_{V}.\]
\end{lemma}
\begin{proof}
Let $\omega'$ be a basis of $\omegab_{A'/\cO_K}$ as $(\cO_K\otimes\cO_{F})$-module, where $A'=A/H'$. Then it's easy to see that $p^{-1}\hat{\phi}^*\omega=a\omega'$ for some
$a\in (K\otimes F)^{\times}$ with $|\chi_{\beta}(a)|=p^{\deg_{\beta}(H')}$ for all $\beta\in\bB$. Note that $\deg_{\beta}(H')=0$ if $\beta\notin \bB_{\fp}$. Therefore, we have
\begin{align*}
|f(A/H',(H+H')/H',p^{-1}\hat{\phi}^*\omega)|&=|(\prod_{\beta\in \bB}\chi^{
-k_{\beta}}_{\beta}(a))f(A/H',(H+H')/H',\omega')|\\
&=p^{-\sum_{\beta\in\bB_{\fp}}k_{\beta}\deg_{\beta}(H')}\;|f(A/H',(H+H')/H',\omega')|\\
&\leq p^{-\sum_{\beta\in \bB_{\fp}}k_{\beta}\deg_{\beta}(H')}\;|f|_{V}.
\end{align*}
\end{proof}

From the formula \ref{equ-U_p}, we deduce immediately that

\begin{cor}\label{cor-est-norm}
Let $f$ be a section of $\omegab^{\vk}$ defined over an admissible open subset $V$ of $\fY_{\rig}$ with $|f|_V$ finite, and $\fp$ be a prime ideal of $\cO_F$ above $p$ with $f_{\fp}=[\kappa(\fp):\F_p]$. Let  $Q=(A,H)$ be a rigid point of $\fY_{\rig}$ such that $U_{\fp}(Q)\subset V$. Assume that $k_{\beta_0}=\min_{\beta\in \bB_{\fp}}\{k_{\beta}\}$, and there exists $c>0$ such that $\sum_{\beta\in \bB_{\bB}}\deg_{\beta}(H')$ for any $(\cO_{F}/\fp)$-cyclic subgroups of $H'\subset A[\fp]$ different from $H$. Then we have
\begin{equation*}|U_{\fp}(f)(Q)|\leq p^{f_{\fp}-k_{\beta_0}c}|f|_{V}.
\end{equation*}
\end{cor}
\begin{proof}
By the the definition of $U_{\fp}(f)$ \eqref{equ-U_p} and the preceding Lemma,  we have
\begin{align*}
|U_{\fp}(f)(Q)|&\leq p^{f_{\fp}}\sup_{\substack{H'\subset A[p]\\ H'\cap H=0}}|f(A/H',(H+H')/H')|\\
&\leq p^{f_{\fp}}\sup_{\substack{H'\subset A[p]\\ H'\cap H=0}}
\{p^{-\sum_{\beta\in \bB_{fp}}k_{\beta}\deg_{\beta}(H')}\}|f|_{V}.
\end{align*}
The corollary follows from the fact that
$$\sum_{\beta\in \bB_{\fp}}k_{\beta}\deg_{\beta}(H')\geq k_{\beta_0}\sum_{\beta\in\bB_{\beta}}\deg_{\beta}(H')\geq ck_{\beta_0}.$$
\end{proof}
Now we prove the first result on analytic continuation of $p$-adic Hilbert modular forms.

\begin{prop}\label{prop-cont-can}
 Let $f$ be an overconvergent $p$-adic Hilbert modular form of weight $\vk\in \Z^{\bB}$. Assume that for all prime $\fp|p$, we have  $U_{\fp}(f)=a_{\fp}f$ with $a_{\fp}\in \C_p^{\times}$. Then $f$ extends uniquely to a section of $\omegab^{\vk}$ over $\cV_{\can}$ such that $U_{\fp}(f)=a_{\fp}f$ remains true for all $\fp|p$. Moreover,  $|f|_{\cV_{\can}}=\sup_{Q\in \cV_{\can}}{|f(Q)|}$ is finite.
\end{prop}
\begin{proof}
For any rational number $0<r<p$ and any prime ideal $\fp$ of $\cO_F$ dividing $p$, we put
\begin{align*}\cV_{\can}(\fp;r)&=\{Q\in \cV_{\can}\;|\; \nu_{\beta}(Q)+p\nu_{\sigma^{-1}\circ\beta}(Q)\geq p+1-r\; \text{for all }\beta\in \bB_{\fp} \}.\\
\cV_{\can}(r)&=\bigcap_{\fp|p}\cV_{\can}(r).
\end{align*}
Note that $\cV_{\can}(r)$ is a quasi-compact admissible open subset of $\cV_{\can}$, and  $\{\cV_{\can}(r)\}_{r\ra 0^{+}}$ form  a fundamental system of strict neighborhoods of $\fY_{\rig}^{\ord}$.
Using Prop. \ref{prop-hecke-can}(a), it's easy to check that $U_{\fp}(\cV_{\can}(\fp;r))\subset \cV_{\can}(\fp;r/p)$, and hence
\[(\prod_{\fp|p}U_{\fp})(\cV_{\can}(r))\subset \cV_{\can}(r/p).\]
 We may assume that $f$ is defined over some $\cV_{\can}(r_0)$. Let $n\geq 1$ be the minimal integer such that $p^nr_0>p$, then we have $(\prod_{\fp|p}U_{\fp})(\cV_{\can})\subset \cV_{\can}(r_0)$. Therefore, $(\prod_{\fp}\frac{1}{a^n_{\fp}}U^n_{\fp})f$ is a well-defined section of $\omegab^{\vk}$ over $\cV_{\can}$ extending $f$. It's clear that the (extended) form $f$ still satisfy the functional equations $U_{\fp}(f)=a_{\fp}f$. The finiteness of  $|f|_{\cV_{\can}}$ follows immediately from Corollary \ref{cor-est-norm}.

\end{proof}

The following useful Proposition is motivated by \cite[\S 7]{Pi}.

\begin{prop}\label{prop-exten}
Let $Z$ be any closed  scheme of $Y_{\kappa}$ with $\mathrm{codim}_{Y_{\kappa}}(Z)\geq 2$, and $]Y_{\kappa}-Z[$ be the rigid analytic tube in $\fY_{\rig}$ of the open subset $Y_{\kappa}-Z$. Then for any finite extension $L/\Q_{\kappa}$ and any weight $\vk\in \Z^{\bB}$, the natural restriction map
\[ H^0(\fY_{\rig,L},\omegab^{\vk})\ra H^0(]Y_{\kappa}-Z[_{L},\omegab^{\vk})\]
is an isomorphism. Moreover, for any $f\in H^0(\fY_{\rig,L},\omegab^{\vk})$, we have $|f|_{\fY_{\rig,L}}\leq |f|_{]Y_{\kappa}-Z[_L}$.
\end{prop}
\begin{proof}
We will just work in the case  $L=\Q_{\kappa}$ to simplify the notation, as the general case can be treated in the same way.
Let $Q$ be a closed point of $Y_{\kappa}$, and $\widehat{\cO}_{Y_{\kappa},Q}$ be the completion of the local ring of $Y_{\kappa}$ at $Q$. Then with the notation in \eqref{equ-loc-ring}, we have
\[\widehat{\cO}_{Y_{\kappa},Q}\simeq \kappa(Q)[[\{x_{\beta},y_{\beta}:\beta\in I(Q)\}, \{z_{\gamma}: \gamma\in \bB-I(Q)\}]]/(\{x_{\beta}y_{\beta}: \beta\in I(Q)\}).\]
We see that $\widehat{\cO}_{Y_{\kappa},Q}$  is Cohen-Macaulay, in particular it satisfy the condition $S_2$. Since $Y_{\kappa}$ is excellent by \cite[7.8.3(ii)]{ega}, it follows from [\emph{loc. cit.} 7.8.3(v)] that $\cO_{Y_{\kappa},Q}$ also Cohen-Macaulay. Now the Proposition follows directly from Corollary \ref{cor-ext-1} and Remark \ref{rem-ext}.
\end{proof}

The following proposition is an analogue of   \cite[2.4]{Pi} in the Hilbert case.

\begin{prop}\label{prop-pilloni}
 Let $\fp$ be a prime ideal of $\cO_F$ dividing $p$, $f_{\fp}=[\kappa(\fp):\F_{p}]$, $n\geq 1$ be an integer,  $U_{\fp}:\fY_{\rig}\ra \fY_{\rig}$ be the set theoretic Hecke correspondence \eqref{defn-set-U_p}. Let $L$ be a finite extension of $\Q_\kappa$,  $Q=(A,H)$ be an $L$-valued rigid point of $\fY_{\rig}$, and $Q_n=(A_n,H_n)\in U_{\fp}^n(Q)$ defined over a finite extension of $L$. Assume that the $\fp$-divisible group $A[\fp^{\infty}]$ is not ordinary. Then we have $\nu_{
\beta}(Q_n)=\nu_{\beta}(Q)$ for $\beta\notin \bB_{\fp}$, and
\begin{equation}\label{equ-QQ}
\sum_{i=0}^{f_{\fp}-1}p^i\nu_{\sigma^{-i}\circ\beta}(Q_n)\geq \sum_{i=0}^{f_{\fp}-1}p^i\nu_{\sigma^{-i}\circ\beta}(Q)
\end{equation}
for $\beta\in \bB_{\fp}$.
The equalities above hold for all $\beta\in \bB_{\fp}$ if and only if the following properties are verified:

\emph{(a)} The subgroup scheme $H[\fp]$ is a truncated Barsotti-Tate group of level $1$. Moreover,  there exists a subset $I_{\fp}\subsetneq \bB_{\fp}$ such that $\nu_{\beta}(Q)=1$ for $\beta\in I_{\fp}$ and $\nu_{\beta}(Q)=0$ for $\beta\in I_{\fp}^c=\bB_{\fp}-I_{\fp}$, $w_{\beta}(A)=0$ for $\beta\in (\sigma(I_{\fp})\cap I_{\fp})\cup (\sigma(I_{\fp}^c)\cap I_{\fp}^c) $, and $w_{\beta}(A)=1$ for $\beta\in (\sigma(I_{\fp})\cap I_{\fp}^c)\cup (\sigma(I_{\fp}^c)\cap I_{\fp})$.

 \emph{(b)} There exists  a truncated  Barsotti-Tate subgroup  $G_n\subset A[\fp^n]$ of level $n$, stable under the action of $\cO_F$, such that the natural maps
  $$G_n\times H[\fp]\ra A[\fp^n]$$
   is a closed embedding.

 \emph{(c)} We have $Q_n=(A/G_n, H)\in \fY_{\rig}$, where  we have considered $H$ as a subgroup via \emph{(b)}. In particular, $Q_n$ can be defined over $L$.

\emph{(d)}  Let $Q_n'\in U_{\fp}^n(Q)$ distinct from $Q_n=(A/G_n,H)$. If $p\geq 3$, we have
\[\nu_{\beta}(Q_n')\geq \frac{p-2}{p-1}+\frac{1}{p^{g-1}(p-1)}\quad \text{for any }\beta\in \bB_{\fp};\]
in particular, $Q_n'$ is canonical at $\fp$.
\end{prop}
\begin{proof}
We have a canonical decomposition of $p$-divisible groups
\[A[p^\infty]=\prod_{\fq|p}A[\fq^\infty].\]
Since the Hecke correspondence $U_{\fp}$ concerns only the $\fp$-component, the natural isogeny $A\ra A'$ induces an isomorphism of finite flat  group schemes
\[H[\fq]\xra{\sim}H'[\fq]\]
for $\fq\neq \fp$,  hence $\nu_{\beta}(Q')=\nu_{\beta}(Q)$ if $\beta\notin \bB_{\fp}$. There exists a finite extension $L'/L$, and a sequence $Q_0=Q,\, Q_1, \,\ldots,\, Q_n\in \fY_{\rig}$ defined over $L'$ such that $Q_{m}\in U_{\fp}(Q_{m-1})$ for $1\leq m\leq n$. Assume $Q_m=(A_m,H_m)$ and $A_{m}\simeq A_{m-1}/D_{m}$ for some closed subgroup $D_{m}\subset A_{m-1}[\fp]$ with $D_{m}$ distinct form $H_{m-1}$ for $1\leq m\leq n$ ($H_0=H$). We have a sequence of homomorphisms of groups schemes of $1$-dimensional $(\cO_F/\fp)$-vector spaces over $\cO_{L'}$:
\[H[\fp]\ra H_1[\fp]\ra H_2[\fp]\ra\cdots\ra H_n[\fp]\]
which are generically isomorphisms.
The first part of the proposition follows from Lemma \ref{lem-ray}.

 For the second part, the ``if'' direction is trivial. Assume now the equalities in \eqref{equ-QQ} hold for all $\beta\in \bB_{\fp}$.  Lemma \ref{lem-ray} implies that the natural morphism $H[\fp]\ra H_n[\fp]$  is an isomorphism,  hence so is $H[\fp]\ra H_m[\fp]$ for $1\leq m\leq n$. It follows that
 the exact sequence $0\ra D_m\ra A_{m-1}[\fp]\ra H_{m}[\fp]\ra 0$
 splits for $1\leq m\leq n$, i.e.
 \begin{equation}\label{equ-split}
 A_{m-1}[\fp]=H_{m}[\fp]\times D_m\simeq H[\fp]\times D_m.
 \end{equation} As direct summands of  a Barsotti-Tate group of level $1$, both $H[\fp]$ and $D_m$ are Barsotti-Tate groups of level $1$. Moreover, since $\omega_{A[\fp]}=\omega_{H[\fp]}\oplus\omega_{D_1}$ and $\omega_{A[\fp]}$ is a free $\cO_{L}\otimes (\cO_{F}/\fp)$-module, we have $\nu_{\beta}(Q)=\deg_{\beta}(H)\in \{0,1\}$ for $\beta\in \bB_{\fp}$. Therefore, there exists a subset $I_{\fp}\subset \bB_{\fp}$ such that $H[\fp]$  (resp. $D_1$) is a special subgroup of $A[\fp]$ of type $I_{\fp}$  (resp. of type $I_{\fp}^c$) in the sense of Prop. \ref{prop-gen-split}.
   The hypothesis that $A[\fp^{\infty}]$ is not ordinary implies that $I_{\fp}\neq \bB_{\fp}$. Condition (a) follows immediately from \ref{prop-gen-split}(a).

  To simplify notation, we denote simply by $H$ and $A$ their base change to $\cO_{L'}$. To prove (b) and (c), we construct inductively a sequence of closed subgroup schemes $G_m\subset A[\fp^m]$ for $1\leq m\leq n$ such that  $G_{m}\subset G_{m+1}$ and $A_m=A/G_m$. We put $G_1=D_1$, and assume that $m\geq 2$ and $G_{m-1}$ has been constructed.  We have a left exact sequence
 \[0\ra G_{m-1}(\overline{L})\ra A[\fp^{m-1}](\overline{L})\ra A_{m-1}[\fp^{m-1}](\overline{L}),\]
 where $\overline{L}$ is an algebraic closure of $L$.  We define $G_{m}$ to be the scheme theoretic closure in $A[\fp^{m-1}]$ of the inverse image of $D_{m}({\overline{L}})\subset A_{m-1}[\fp](\Lb)$. Since  the image of $A[\fp]$ in $A_{m}[\fp]$ is $H_m[\fp]\simeq H[\fp]$, we have $G_{m}({\overline{L}})\cap A[\fp](\overline{L})=G_1(\Lb)\simeq \cO_F/\fp$. It follows that $G_{m}(\Lb)\simeq \cO_F/\fp^{m}$ for any $1\leq m\leq n$.  Now we can show $G_n$ is a Barsotti-Tate group of level $n$, i.e.  the natural map $G_{m+1}/G_1\ra G_{m}$ induced by the multiplication by $p$ is an isomorphism for any $m$. It is clearly an isomorphism on the generic fibers, so we just need to prove that $\deg(G_{m+1})=\deg(G_1)+\deg(G_{m})$ by \cite[Cor. 3]{Fa}. This  results from the splitting  \eqref{equ-split}:
\begin{align*}
 \deg(G_{m+1})-\deg(G_m)&=\deg(A_m[\fp])-\deg(H)\\
 &=\deg(A[\fp])-\deg(H)\\
 &=\deg(G_1).
\end{align*}
In this same way, the natural map $G_n\times H[\fp]\ra A[\fp^n]$ is clearly a closed embedding over the generic fibers,  and it's an isomorphism over $\cO_{L'}$ because the degree of $G_n\times H[\fp]$ equals that of its image.
This proves (b) and (c) except for the rationality of $G_n$ over $L$.
Actually,  $D_{m}$ is the unique special subgroup of $A_{m-1}[\fp]$ of type $I^c_{\fp}\neq \emptyset$ by \ref{prop-gen-split}(a). By induction, all $A_m=A_{m-1}/D_m$  for $0\leq m\leq n$ are defined over $\cO_L$,  so is $G_n$.

 For condition (d), note that there exists a sequence of rigid points $Q_0'=Q, Q_1',\cdots,Q_n'$ such that $Q_{m}'\in U_{\fp}(Q_{m-1}')$ for $1\leq m\leq n$.   Let $r\leq n$ be the minimal integer such that $Q_r'\neq Q_r$. We have $Q_r'=(A_{r-1}/H', (H+H')/H')$ for some $(\cO_{F}/\fp)$-cyclic subgroup $H'\subset A_{r-1}[\fp]$ distinct from $H$ and $D_{r}$. Note that $H[\fp]$ and $D_{r}$ are respectively special subgroups of $A_{r-1}[\fp]$ of type $I_{\fp}$ and $I_{\fp}^c$. It follows from \ref{prop-gen-split}(a)(3)  that
  $$\nu_{\beta}(Q_r')=1-\deg_{\beta}(H')\geq \frac{p-2}{p-1}+\frac{1}{p^{g-1}(p-1)};$$  in particular, $\nu_{\beta}(Q_r')+p\nu_{\sigma^{-1}\circ \beta}(Q_r')>1$ for $\beta\in \bB_{\fp}$ if $p\geq 3$. By definition, $Q'_r$ is canonical at $\fp$. This finishes the proof if $r=n$. If $r<n$, Prop. \ref{prop-hecke-can}(a) implies that $$\nu_{\beta}(Q_m')=\frac{p-1}{p}+\frac{1}{p}\nu_{\sigma\circ\beta}(Q_{m-1}')>\frac{p-2}{p-1}+\frac{1}{p^{g-1}(p-1)}$$
 for any $r<m\leq n$ and $\beta\in \bB_{\fp}$.
\end{proof}

\begin{cor}
 Let $I_{\fp}\subset \bB_{\fp}$ be a subset, and $I_{\fp}^c$ be its complementary in $\bB_{\fp}$. Let $Q=(A,H)$ be rigid point of $\fY_{\rig}$ with  $\nu_{\beta}(Q)=1$ for $\beta\in I_{\fp}$ and $\nu_{\beta}(Q)=0$ for $\beta\in I_{\fp}^c$. Assume that $\sigma(I_{\fp})\subset I_{\fp}^c$. Then there exists a rigid point $Q_1\in U_{\fp}(Q)$ such that $\nu_{\beta}(Q_1)=\nu_{\beta}(Q)$ for all $\beta$, i.e. the equivalent conditions in the Proposition are satisfied for $n=1$, if and only if $w_{\beta}(A)=1$ for $\beta\in
\sigma(I_{\fp}^c)\cap I_{\fp}$.

 \end{cor}

\begin{proof} From \ref{prop-pilloni}(a), the condition that $w_{\beta}(A)=1$ for $\beta\in  \sigma(I_{\fp}^c)\cap I_{\fp}$ is clearly necessary.  We note  that $H[\fp]$ is a special subgroup of $A[\fp]$ of type $I_{\fp}$. It follows from Prop. \ref{prop-gen-split}(a) that $w_{\beta}(A)=0$ for $\beta\in \sigma(I_{\fp}^c)\cap I_{\fp}^c$, and  $w_{\beta}(A)=1$ for $\beta\in \sigma(I_{\fp})\cap I_{\fp}^c=\sigma(I_{\fp})$. If  $w_{\beta}(A)=1$ for $\beta\in \sigma(I_{\fp}^c)\cap I_{\fp}$, then \ref{prop-gen-split}(b) implies that  there exists a special subgroup $D_1\subset A[\fp]$ of type $I_{\fp}^c$. The point $Q_1=(A/D_1, (H+D_1)/D_1)$ satisfies the required property.
\end{proof}

 \section{Proof of Theorem \ref{thm-quad}:  Case of \ref{thm-main} when $g=2$.}

 In this section, we assume $g=2$, and $p$ is inert in $F$ so that $\cO_F/p\simeq \F_{p^2}$. We identify  $\bB=\mathrm{Emd}_{\Q}(F,\Q_{\kappa})$ with $\Z/2\Z=\{1,2\}$.

 \subsection{}\label{sect-strata} We start with the stratifications on $X_{\kappa}$ and $Y_{\kappa}$ described in the previous section. To simplify the notation, we drop the symbol ``$\{\_\}$'' in the subscript, when there is only one element in the set. For example, we denote $W_{\{1\}}$ and $W_{\{1\},\{1\}}$ simply by $W_1$ and $W_{1,1}$. There are 4 strata in $X_{\kappa}$: $W_{\emptyset}, W_{1}, W_{2}$ and $W_{\bB}$.  The $2$-dimensional stratum $W_{\emptyset}$ is the ordinary locus,  $W_{1}$ and $W_{2}$ are the loci where the fibers of the universal HBAV are supersingular and with $a$-number equal to $1$, and $W_{\bB}$ is the discrete set consisting  of superspecial points. There are 9 strata in the Goren-Kassaei stratification of $Y_{\kappa}$.  Here is a list of the strata  according to their dimensions:
\begin{itemize}
\item 2-dimensional: $W_{\bB,\emptyset}$, $W_{\emptyset,\bB}$, $W_{1,1}$, $W_{2,2}$.

    \item 1-dimensional: $W_{\bB,1}$, $W_{\bB, 2}$, $W_{1,\bB}, W_{\bB,2}$.

    \item 0-dimensional: $W_{\bB,\bB}$.
\end{itemize}
We have  $\pi(W_{\bB,\emptyset})=\pi(W_{\emptyset, \bB})=W_{\bB}$, $\pi(W_{\bB,i})=\pi(W_{i,\bB})=W_{i}$ for $i\in \Z/2\Z$, $\pi(W_{\bB,\bB})=W_{\bB}$, and finally $\pi(W_{i,i})=W_{i}\cup W_{\bB}$. In particular, the projection $\pi$ is not quasi-finite on the strata $W_{1,1}$ and $W_{2,2}$. For $i\in \Z/2\Z$, we put
\[
W_{i,i}^{sg}=W_{i,i}\cap \pi^{-1}(W_{i}),\quad\quad
W_{i,i}^{ss}=W_{i,i}\cap \pi^{-1}(W_{\bB}).
\]
We call
\begin{equation}\label{defn-sg-ss}
Y_{\kappa}^{sg}=W_{1,1}^{sg}\cup W_{2,2}^{sg}
\end{equation}
the supergeneral locus, and
\[
Y_{\kappa}^{ss}=\pi^{-1}(W_{\bB})=W_{1,1}^{ss}\cup W^{ss}_{2,2}\cup W_{\bB,\bB}.\nonumber
\]
 the superspecial locus.

\subsection{} Let $k$ be an algebraically closed field containing $\F_{p^2}$. For a finite group scheme $G$ over $k$, we denote by $\omega_{G}$ its module of invariant differentials. If $G$ is equipped with an action of $\cO_{F_p}\simeq \Z_{p^2}$, we have a decomposition $\omega_{G}=\omega_{G,1}\oplus \omega_{G,2}$, where $\cO_F$ acts on $\omega_{G,\beta_i}$ via $\chi_{\beta_i}$.

In \cite[4.1]{Pi}, Pilloni classified the commutative finite group schemes of order $p^2$ over $k$. There are only 4 isomorphism classes: $\alpha_p\times \alpha_p$, $\alpha_{p^2}$, $\alpha_{p^2}^\vee$ and $\alpha$. Here, $\alpha_{p^n}=\Ker (F^n: \G_a\ra \G_a)$ is the kernel of $n$-th iterated Frobenius of the additive group for $n=1,2$, $\alpha_{p^2}^\vee$ is the Cartier dual of $\alpha_{p^2}$, and $\alpha$ is the $p$-torsion of a supersingular elliptic curve over $k$. The groups $\alpha_p\times \alpha_p$ and $\alpha$ are isomorphic to their Cartier dual. These groups can be characterized by the dimension of their invariant differential modules and those of their duals:
\begin{align*}
\dim_k(\omega_{\alpha_p\times \alpha_p})&=2 \quad\quad\dim_k(\omega_{\alpha})=1\\
\dim_k(\omega_{\alpha_{p^2}})&=1 \quad\quad \dim_k(\omega_{\alpha^\vee_{p^2}})=2.
\end{align*}

\begin{lemma}\label{lem-fiber-sg}
Let $P$ be a $k$-valued point of the stratum $W_{1}\subseteq X_{\kappa}$, and $A$ be the corresponding HBAV. Then we have $\pi^{-1}(P)_{\red}\simeq \bP^1_k$ with two distinguished points corresponding to
 $$H=\Ker(F_A:A\ra A^{(p)})\quad \text{and}\quad H=\Ker(V_{A^{(p^{-1})}}:A\ra A^{(p^{-1})}).$$
 In the first distinguished case, we have $H\simeq \alpha_{p^2}^\vee$ and $(A,H)\in W_{\bB,1}$; in the second distinguished case, we have $(A,H)\in W_{1,\bB}$; otherwise, we have $H\simeq \alpha$ and $(A,H)\in W_{1,1}$. Moreover, we have $\dim_{k}(\omega_{H,2})=1$ in all the cases, and
 \[\dim_k(\omega_{H,1})=\begin{cases}0 & \text{if } H\simeq \alpha\;\text{or }\; \alpha_{p^2}\\
 1 &\text{if } H\simeq \alpha_{p^2}^\vee.\end{cases}
 \]
\end{lemma}
\begin{proof}
Up to isomorphisms, the contravariant Dieudonn\'e module of $A[p]$ can be described as  $\bD(A[p])=\bD(A[p])_{1}\oplus \bD(A[p])_{2}$ with $\bD(A[p])_{i}=k e_i\oplus k d_i$, and the Frobenius and Verschiebung are given by
\begin{align*}F(e_1,d_1)&=(e_2,d_2)\begin{bmatrix}0 &1\\ 0&t_2\end{bmatrix},\quad \quad
 F(e_2,d_2)=(e_1,d_1)\begin{bmatrix}
0&1\\0&0\end{bmatrix}, \\
V(e_1,d_1)&=(e_2,d_2)\begin{bmatrix}0&1 \\ 0 &0\end{bmatrix},\quad\quad
V(e_2,d_2)=(e_1,d_1)\begin{bmatrix} -t_2^{1/p} &1\\0&0\end{bmatrix},
\end{align*}
where $t_2\in k^\times$ is the value of the partial Hasse invariant $h_{2}$ at $A$ with respect to certain basis. Giving an $(\cO_F/p)$-cyclic subgroup $H\subset A[p]$ corresponds to giving a $k$-line $\bL_{i}=\bD(A[p]/H)_{\beta_i}\subset \bD(A[p])_{i}$ for each $i=1,2$ satisfying
\[F(\bL_{i})\subseteq \bL_{i+1}\quad \text{and}\quad V(\bL_{i})\subseteq \bL_{i+1}.\]
It's easy to see that the only choice for $\bL_{1}$ is $k e_1=(\Ker F)_{1}=(\Ker V)_{1}$, and the choices for $\bL_{2}$ are $k(ae_2+b(e_2+t_2 d_2))$ with $(a:b)\in \bP^1_k$. If $(a:b)=(1:0)$, we have $\bL_{2}=(\im V)_{2}$ and $H=\Ker(V_{A^{(p^{-1})}})\simeq \alpha_{p^2}$. If $(a:b)=(0:1)$, we have $\bL_{2}=(\im F)_{2}$ and $H=\Ker(F_A)\simeq \alpha_{p^2}^\vee$. Otherwise, $\bL_{2}$ is neither $(\im V)_{2}$ nor $(\im F)_{2}$, we have $H\simeq \alpha$. The ``moreover'' part of the Lemma follows from the fact that $\bD(H)_{i}=\bD(A[p])_{i}/\bL_{i}$.

\end{proof}

\begin{lemma}\label{lem-fiber-ss}
 Let $P$ be a $k$-point of $W_{\bB}\subset X_{\kappa}$, and $A$ be the corresponding HBAV.

\emph{(a)} The reduced fiber $\pi^{-1}(P)_{\red}$ consists of two copies of $\bP^1_k$, denoted respectively by $\bP^1_{k,1}$  and $\bP^1_{k,2}$, intersecting transversally at a single point $Q$.

\emph{(b)} We have $\pi^{-1}(P)_{\red}\cap W_{i,i}=\bP^1_{k,i}\backslash \{Q\}$ for $i\in\Z/2\Z$, and  $\pi^{-1}(P)_{\red}\cap W_{\bB,\bB}=\{Q\}$.

\emph{(c)} For a point $(A,H)\in \pi^{-1}(P)_{\red}$, we have $(A,H)\in \bP^{1}_{k,i}$ if and only if $H\simeq \alpha$ and $\omega_{H,i}\neq 0$, and we have $(A,H)=Q$, if and only if $H=\Ker(F_A)=\Ker(V_{A^{(1/p)}})\simeq \alpha_p\times \alpha_p$.
\end{lemma}

The proof of this lemma is quite similar to the previous one, and will be left to the reader as an exercise.

\subsection{}\label{sect-pic} Consider the valuation $\nu_{\fY}=(\nu_{1},\nu_{2}): \fY_{\rig}\ra [0,1]\times [0,1]$ defined in \ref{defn-val-Y}. Let $\spe: \fY_{\rig}\ra Y_{\kappa}$ be the specialization map. If $Z\subset Y_{\kappa}$ is a locally closed subset, we denote by $]Z[\;=\spe^{-1}(Z)$ the tube of $Z$ in $\fY_{\rig}$.  By Prop. \ref{prop-val},  we have, for $i\in\Z/2\Z$, that
\begin{align*}
&Q\in\;]W_{\bB,\emptyset}[\; \Leftrightarrow \nu_{\fY}(Q)=(1,1);\\
&Q\in \;]W_{\emptyset, \bB}[\;  \Leftrightarrow \nu_{\fY}(Q)=(0,0);\\
&Q\in \;]W_{\bB,i}[\;\Leftrightarrow 0<\nu_{i}(Q)< 1\quad \text{and}\quad \nu_{i+1}(Q)=1; \\
&Q\in \;]W_{i,i}[\; \Leftrightarrow \nu_{i}(Q)=0\quad \text{and}\quad \nu_{i+1}(Q)=1;\\
&Q\in \;]W_{i,\bB}[\; \Leftrightarrow \nu_{i}(Q)=0\quad \text{and}\quad 0<\nu_{i+1}(Q)<1;\\
&Q\in \;]W_{\bB,\bB}[\; \Leftrightarrow 0<\nu_{1}(Q),\nu_2(Q)<1.
\end{align*}
In summary,  the $4$ vertices of the square $[0,1]\times[0,1]$ correspond to the strata of dimension $2$, the $4$ edges correspond to the $1$-dimensional strata, and the interior corresponds to the unique stratum of dimension $0$. We have the following graph:

 \begin{center}
 \setlength{\unitlength}{2cm}
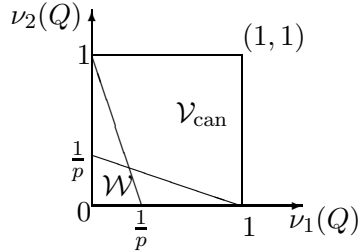
 \begin{picture}(1.7,1.8)(-0.2,-0.2)

 \put(0,0){\vector(1,0){1.4}}
 \put(0,0){\vector(0,1){1.3}}
 \put(-0.1,-0.1){$0$}
 \put(-0.3,1.26){\makebox(0,0){$\nu_{2}(Q)$}}
 \put(1.55,-0.1){\makebox(0,0){$\nu_{1}(Q)$}}
 \put(0,1){\line(1,0){1}}
 \put(1,0){\line(0,1){1}}
 \put(1,1.07){$(1,1)$}
 \put(1,-0.2){$1$}
 \put(-0.1,0.94){{$1$}}
 \put(0.3333,0){\line(-1,3){0.3333}}
 \put(0.27,-0.2){$\frac{1}{p}$}
 \put(0,0.3333){\line(3,-1){1}}
 \put(-0.15,0.26){$\frac{1}{p}$}
 \put(0.56,0.56){$\cV_{\can}$}
 \put(0.07,0.07){$\cW$}

 \end{picture}
\end{center}
Here, the two line segments with end points $\{(0,1),(1/p,0)\}$ and $\{(0,1/p),(1,0)\}$ are respectively $p\nu_{1}(Q)+\nu_{2}(Q)=1$ and $\nu_{1}(Q)+p\nu_{2}(Q)=1$. They intersect at the point $(\frac{1}{p+1},\frac{1}{p+1})$. The anti-canonical locus $\cW$ is strictly below these two lines, and the canonical  locus $\cV_{\can}$ is strictly above them. Note that $\cV_{\can}$ contains $\;]W_{\bB,\emptyset}[\;$ and $\;]W_{\bB,i}[\;$ for $i\in\Z/2\Z$.

 Let $U_p:\fY_{\rig}\ra \fY_{\rig}$ be the set theoretic Hecke correspondence \eqref{defn-set-U_p}. We want to understand the dynamics of $U_p$ on the too-singular locus.

\begin{prop}\label{prop-dyn-sg} Let $Q=(A,H)\in \fY_{\rig}$ be a rigid point,  $i\in \Z/2\Z$.

\emph{(a)}   If $Q\in \;]W_{i,i}^{sg}[\;$, then we have
\[\nu_{i}(Q_1)=1\quad \text{and}\quad  \nu_{i+1}(Q_1)=1-\frac{1}{p}\]
for  all $Q_1\in U_p(Q)$. In particular, the orbit $U_p(Q)$ is contained in the canonical locus $\cV_{\can}$.

\emph{(b)} If $Q\in \;]W_{i, \bB}[\;$ with $\frac{p+1}{p^2+1}\leq\nu_{i+1}(Q)<1$, we have $\nu_{\beta_{i}}(Q_1)=1$ and $\nu_{i+1}(Q_1)=1-\frac{1}{p}-\frac{1}{p^2}(1-\nu_{i+1}(Q))$ for all $Q_1\in U_p(Q)$. In particular,  the orbit $U_p(Q)$ is contained in $\cV_{\can}$.

\emph{(c)} If $Q\in \;]W_{i,\bB}[\;$ with $\frac{1}{p}\leq \nu_{{i+1}}(Q)<\frac{p+1}{p^2+1}$. Then there is a unique point $Q_1\in U_p(Q)$ with $\nu_i(Q_1)=1$ and $\nu_{i+1}(Q_1)=p^2(\nu_{i+1}(Q)-\frac{1}{p})$, and the all the other $(p^2-1)$ points $Q_1'\in U_p(Q)$ satisfy $\nu_i(Q_1')=1$ and $\nu_{i+1}(Q_1')=1-\nu_{i+1}(Q)$. In particular, if $\nu_{i+1}(Q)>\frac{1}{p}$, we have $U_p(Q)\subset \cV_{\can}$; and if $\nu_{i+1}(Q)=\frac{1}{p}$, then there is a unique point $Q_1\in U_p(Q)$ contained in  $\,]W_{i,i}[\,$, and the other points of $U_p(Q)$ are all in $\cV_{\can}$.

\end{prop}
\begin{proof}  If $Q_1=(A/H',A[p]/H')$ is the rigid point corresponding to a $(\cO_F/p)$-cyclic closed subgroup scheme $H'$ of $A[p]$, then we have
\begin{equation}\label{equ-hecke-deg}
\nu_{i}(Q_1)=\deg_i(A[p]/H')=1-\deg_{i}(H').
\end{equation}
Now the Proposition is a direct consequence of Prop. \ref{prop-subgp-sg}.

\end{proof}

To describe the dynamics of $U_p$ on $\,]W_{i,i}^{ss}[\,$ for $i\in \Z/2\Z$, we need the following

\begin{prop}\label{lem-subgp-ss}
 Let $K$ be a finite extension of $\Q_{\kappa}$, and $Q=(A,H)$ be a $K$-valued rigid point of  $\,]W_{i,i}^{ss}[\,$. Then we have $w_i(A)=1$, $0< w_{i+1}(A)\leq 1$, and all the $(\cO_{F}/p)$-cyclic closed subgroup schemes $H'$ of $A[p]$ different from $H$ satisfy
\begin{equation}\label{equ-deg}
\deg_i(H')+p\deg_{i+1}(H')=1.
\end{equation}
Moreover,

\emph{(a)} if $0<w_{i+1}\leq \frac{p}{p+1}$, all the $p^2$ $(\cO_F/p)$-cyclic subgroups $H'\subset A[p]$ different from $H$ satisfy
\[\deg_i(H')=\frac{w_{i+1}(A)}{p}\quad \text{and}\quad \deg_{i+1}(H')=\frac{1}{p}-\frac{w_{i+1}(A)}{p^2};\]

\emph{(b)} if $\frac{p}{p+1}<w_{i+1}(A)\leq 1$, there is a unique $(\cO_F/p)$-cyclic subgroup $H'_0$ different from $H$ with
\[\deg_i(H'_0)=1-p(1-w_{i+1}(A))\quad\text{and}\quad\deg_{i+1}(H'_0)=1-w_{i+1}(A),\]
and all the other $p^2-1$ subgroups $H'$ satisfy $\deg_1(H')=\deg_2(H')=\frac{1}{p+1}$.

\end{prop}

\begin{proof} The fact that $w_i(A)=1$ follows from Prop. \ref{prop-key-lemma}(d) or Prop. \ref{prop-split-kisin}(a), and that $w_{i+1}(A)>0$ follows from the definition of $\;]W_{i,i}^{ss}[\;$. Let $H'$ be an $(\cO_F/p)$-cyclic subgroup of $A[p]$ different from $H$, and put $Q'=(A,H')$. Lemma \ref{lem-fiber-ss} implies that $Q'$ lies  either in $\,]W_{i,i}^{ss}[\;$, or in $]W_{i+1,i+1}^{ss}[\,$, or in $\,]W_{\bB,\bB}[\,$. Since Prop. \ref{prop-split-kisin}(a) implies that $H$ is the unique $(\cO_F/p)$-cyclic subgroup of $A[p]$ in $\,]W_{i,i}^{ss}[\,$, the first case is excluded. If $Q'\in ]W_{i+1,i+1}^{ss}[$, we have $\deg_{i+1}(H')=0$ and $\deg_i(H')=1$; in particular, \eqref{equ-deg} holds for $H'$. If $Q'\in \,]W_{\bB,\bB}[\,$, then Prop. \ref{prop-key-lemma}(a) implies that there exists units $u,v\in\cO_K$ such that
\[1=w_i(A)=\min\{1,v_p(ux_{i}(Q')+vy_{i+1}^p(Q'))\},\]
where $x_i, y_{i+1}$ are the local parameters defined in \ref{loc-par}. Since $0<v_p(x_i(Q')),v_p(y_{i+1}(Q'))<1$, we see that $1-v_p(y_i(Q'))=v_p(x_i(Q'))=v_p(y_{i+1}^p(Q'))$, i.e. $v_p(y_{i}(Q'))+pv_p(y_{i+1}(Q'))=1$. We deduce \eqref{equ-deg}  from Prop. \ref{prop-val}. It remains to prove the ``Moreover'' part of the Proposition. Applying Prop. \ref{prop-key-lemma}(a) to the $(i+1)$-th partial Hasse invariant, we see that there exist units $r,s\in \cO_K$ with
\[w_{i+1}(A)=\min\{1,v_p(rx_{i+1}(Q')+sy^p_{i}(Q'))\}.\]
It follows from Prop. \ref{prop-val} and \eqref{equ-deg} that
\begin{align*}
v_p(x_{i+1}(Q'))&=1-\deg_{i+1}(H')=1-\frac{1}{p}+\frac{\deg_i(H')}{p},\\
v_p(y_{i}^p(Q'))&=p\deg_{i}(H').
\end{align*}
We have three cases:
\begin{itemize}
\item{} If $0<\deg_{i}(H')<\frac{1}{p+1}$, we have $v_p(x_{i+1}(Q'))>v_p(y^p_{i}(Q'))$. Therefore, we have
$$w_{i+1}(A)=v_p(y^p_{i}(Q'))=p\deg_i(H')<\frac{p}{p+1},$$
i.e. $\deg_i(H')=\frac{w_{i+1}(A)}{p}$ and $\deg_{i+1}(H')=\frac{1}{p}(1-\deg_i(H'))=\frac{1}{p}-\frac{1}{p^2}w_{i+1}(A).$

\item{} If $\deg_i(H')=\frac{1}{p+1}$, we have $v_p(x_{i+1}(Q'))=v_p(y_{i+1}^p(Q'))=\frac{p}{p+1}$, and consequently $w_{i+1}(A)\geq \frac{p}{p+1}$ and $\deg_{i+1}(H')=\frac{p}{p+1}$.

\item{} If $\deg_{i}(H')>\frac{1}{p+1}$, we have $v_p(x_{i+1}(Q'))<v_p(y^p_{i}(Q'))$. Hence, we have
\[w_{i+1}(A)=v_p(x_{i+1}(Q'))=1-\frac{1}{p}+\frac{\deg_i(H')}{p},\]
i.e. $\deg_i(H')=1-p(1-w_{i+1}(A))$ and $\deg_{i+1}(H')=1-w_{i+1}(A)$.
\end{itemize}
From this list, statement (a) is clear. For (b), we have seen that if $\frac{p}{p+1}<w_{i+1}(A)\leq 1$, the possible values for $\deg_{i}(H')$ are $\frac{1}{p+1}$ and $1-p(1-w_{i+1}(A))$. On the other hand, Prop. \ref{prop-split-kisin}(b) implies that there is exactly one  $H'_0$ with $\deg_i(H'_0)=1-p(1-w_{i+1}(A))$ (note that this is even true when $w_{i+1}(A)=1$). Hence statement (b) follows.

\end{proof}

\begin{cor}\label{prop-dyn-ss}
Let $Q=(A,H)$ be a rigid point of $\;]W^{ss}_{i,i}[\;$. Then all the points $Q_1\in U_p(Q)$ satisfy
\begin{equation}\label{equ-deg=1}
\nu_{i}(Q_1)+p\nu_{i+1}(Q_1)=p.
\end{equation}
More precisely, we have two cases:

\emph{(a)} If $0<w_{i+1}(A)\leq\frac{p}{p+1}$, then we have
$$\nu_i(Q_1)=1-\frac{w_{i+1}(A)}{p}\quad\text{and}\quad \nu_{i+1}(Q_1)=1-\frac{1}{p}+\frac{1}{p^2}w_{i+1}(A)$$
 for all $Q_1\in U_p(Q)$. In particular, we have $U_p(Q)\subset \cV_{\can}$.

\emph{(b)} If  $\frac{p}{p+1}<w_{i+1}(A)\leq 1$, there is a unique point $Q_1\in U_p(Q)$ with
$$\nu_i(Q_1)=p(1-w_{i+1}(A))\quad \text{and}\quad\nu_{i+1}(Q_1)=w_{i+1}(A),$$
 and all the other points $Q_1'\in U_p(Q)$ satisfy $\nu_1(Q_1')=\nu_2(Q_1')=\frac{p}{p+1}$. In particular, we have $U_p(Q)\subset \cV_{\can}$ execpt when $w_{i+1}(A)=1$; in the exceptional case, there is exactly one $Q_1\in U_p(Q)$ contained in $\;]W_{i,i}[\;$, and all the other $Q_1$'s are contained in $\cV_{\can}$.
\end{cor}

\begin{proof}
Let $H'$ be a $\Z_{p^2}$-cyclic subgroup of $A[p]$ disjoint from $H$ corresponding to  $Q_1=(A/H',A[p]/H')$. We have $\nu_i(Q_1)=\deg_i(A[p]/H')=1-\deg_{i}(H')$ by the definition \eqref{defn-val-Y}. The Corollary follows immediately from the Proposition.
\end{proof}

\subsection{} We can interpret the results above geometrically as follows.   For $i\in \Z/2\Z$ and any rational number $\varepsilon$ with $0<\varepsilon\leq 1$, we put
\[U_{i}(\varepsilon)= \{P\in \fX_{\rig}\;|\; w_{i}(P)=1, w_{i+1}(P)\geq \varepsilon\}.\]
Then every rigid point in  $U_i(\varepsilon)$ has necessarily superspecial reduction.
Similarly, we put
\[V_{i}(\varepsilon)=\{Q\in \;]W_{i,i}[\;|\; w_{i+1}(Q)\geq \varepsilon \}\]
They are respectively quasi-compact admissible open subsets of $\fX_{\rig}$
and $\fY_{\rig}$.  Prop. \ref{lem-subgp-ss} implies that the natural projection $\pi:\fY_{\rig}\ra \fX_{\rig}$ induces an isomorphism
\[\pi|_{V(\varepsilon)}:V_{i}(\varepsilon)\ra U_{i}(\varepsilon).\] Let $\Pb\in W_{\bB}$, i.e. a superspecial point of $X_{\kappa}$, and $\widehat{\cO}_{\fX,\Pb}$ be the completion of the local ring of $\fX$ at $\Pb$. We choose  local lifts $t_{\Pb,1},t_{\Pb,2}\in \hat{\cO}_{\fX,\Pb}$ of the partial Hasse invariants. Then we have  an isomorphism \eqref{loc-par}
\[\widehat{\cO}_{\fX,\Pb}\simeq W(\kappa(\Pb))[[t_{\Pb,1},t_{\Pb,2}]].\]
Let $\DD_{\Pb}$ be the  2-dimensional rigid open unit disk associated with the formal scheme $\Spf(\widehat{\cO}_{\fX,\Pb})$. Then we have
\[\DD_{\Pb,i}(\varepsilon)=U_i(\varepsilon)\cap \DD_{\Pb}=\{P\in \DD_{\Pb}\;|\; v_p(t_{\Pb,i}(P))\geq 1, v_p(t_{\Pb,i+1}(P))\geq \varepsilon\},\]
and  $U_{i}(\varepsilon)$ is a disjoint union of the closed polydisks $\DD_{\Pb,i}(\varepsilon)$ for all $\Pb\in W_{\bB}$. Composed with the isomorphism $\pi|_{V_i(\varepsilon)}$, we see that  $t_{\Pb,1},t_{\Pb,2}$ for each superspecial point $\Pb$ establish an isomorphism
\beq\label{iso-V_i}
V_i(\varepsilon)\xra{\sim} \coprod_{\Pb\in W_{\bB}}\DD_{\Pb}(\varepsilon).
\eeq

Now suppose $\frac{p}{p+1}<\varepsilon\leq 1$. Let $\pi_1:\cC(p)_{\rig}\ra \fY_{\rig}$ be the first projection of the Hecke correspondence \ref{sect-rig-U_p}. Then  $\pi_1^{-1}(V_{i}(\varepsilon))$ is  a disjoint union of two rigid analytic spaces
\[
\pi^{-1}_1(V_{i}(\varepsilon))=\cC(p)^{\circ}_{\rig}|_{V_{i}(\varepsilon)}\coprod \cC(p)_{\rig}^{s}|_{V_{i}(\varepsilon)},
\]
where $\cC(p)^{s}_{\rig}|_{V_{i}(\varepsilon)}$ corresponds to the unique $(\cO_F/p)$-cyclic subgroup $H'_0\subset A[p]$ disjoint from $H$ given by Prop. \ref{lem-subgp-ss}(b), and $\cC(p)^{\circ}_{\rig}|_{V_i(\varepsilon)}$ corresponds to the remaining $(\cO_F/p)$-cyclic subgroups. Correspondingly, we have a decomposition of set theoretic Hecke correspondences
$U_p=U_p^\circ\coprod U_p^s$ from  $V_{i,1}(\varepsilon)$ to  $\fY_{\rig}$,
given by
$$U^{\circ}_p(Q)=\pi_2((\pi^{\circ}_1)^{-1}(Q))\quad\text{and}\quad U_p^s(Q)=\pi_2((\pi_1^s)^{-1}(Q)).$$
 By Corollary \ref{prop-dyn-ss}, we have $U^{\circ}_p(Q)\subset \cV_{\can}$,   $U_p^{s}(Q)\subset \cV_{\can}$ for $Q\in V_i(\varepsilon)- V_i(1)$, and $U_p^{s}(Q)\subset ]W_{i,i}[$ for $Q\in V_i(1)$. We have  a diagram
\[\xymatrix{
&\cC(p)^{s}_{\rig}|_{V_{i}(\varepsilon)}\ar[dl]^{\sim}_{\pi_1^s}\ar[rd]^{\pi_2}\\
V_i(\epsilon)&& \fY_{\rig},
}\]
where $\pi^s_1$ is an isomorphism of rigid analytic spaces. Hence, the correspondence $U_p^s$ comes from a genuine  morphism of rigid analytic spaces:
\[\pi_{12}^{s}=\pi_2\circ(\pi_1^s)^{-1}: V_{i}(\varepsilon)\ra \cV\cup \,]W_{1,1}[\,\cup \,]W_{2,2}[\,.\]
Note that $\pi_{12}^s(V_i(1))\subset ]W_{i,i}[$, and $\pi_{12}^s(V_{i}(\epsilon)-V_1(1))\subset \cV_{\can}$ for $\frac{p}{p+1}<\epsilon< 1$. Let $Q=(A,H)$ be a  rigid point of $V_{i}(\varepsilon)$ over a finite extension $L/\Q_{\kappa}$, and $H'_0\subset A[p]$ be the unique $(\cO_F/p)$-cyclic subgroup distinct from $H$ given by Prop. \ref{lem-subgp-ss}(b). Let $f$ be a section of $\omegab^{\vk}$ defined over a neighborhood of $Q'=(A/H'_0,A[p]/H_0')$.  Then $U_p^s(f)$ is defined at $Q$, and we have
\beq\label{equ-U_p-s}
U_p^s(f)(A,H,\omega)=\frac{1}{p^2}f(A/H_0',A[p]/H_0', p^{-1}\hat{\phi}^*\omega),
\eeq
where  $\omega$ is a basis of  $\omega_{A/\cO_L}$ as a $(\cO_L\otimes \cO_F)$-module, and $\hat{\phi}:A/H_0'\ra A$
is the canonical isogeny with kernel $A[p]/H_0'$. Similarly, we can define a section $U_p^{\circ}(f)$ of $\omegab^{\vk}$ over $V_i(\varepsilon)$ whenever $f$ is defined over $\cV_{\can}$.

\subsection{}\label{sect-V_n} Fix a rational number $\varepsilon$ with $\frac{p}{p+1}<\varepsilon<1$. We slightly change our notation by  putting $V_{i,1}(\varepsilon)=V_{i}(\varepsilon)$, $V_{i,1}=V_{i}(1)$. For any  integer $n\geq 2$, we define inductively
\[
V_{i,n}(\varepsilon)=(\pi_{12}^s)^{-1}(V_{i, n-1}(\varepsilon)),\quad\text{and}\quad V_{i,n}=(\pi_{12}^s)^{-1}(V_{i, n-1}).\]

We have natural inclusions
\[V_{i,1}(\varepsilon)\supset V_{i,1}\supset V_{i,2}(\varepsilon)\supset V_{i,2}\supset \cdots \supset V_{i,n}(\varepsilon)\supset V_{i,n}\supset\cdots.\]
 By composing $\pi_{12}^s$ with itself $n$-times, we get a  morphism of rigid analytic spaces
\[(\pi_{12}^s)^{n}:V_{i,n}\ra \,]W_{i,i}[\,.\]
\begin{lemma}\label{lem-V_n}

 Let $K$ be a finite extension of $\Q_{\kappa}$, and  $Q=(A,H)$ be a $K$-valued rigid point of $]W_{i,i}[$.

\emph{(a)} For any integer $n\geq 1$, we have $Q\in V_{i,n}$ if and only if there exists a unique $(\cO_{F}/p^n)$-cyclic  truncated  Barsotti-Tate closed subgroup $G_n\subset A[p^n]$ of level $n$ such that the natural morphism
\[G_n\times H\ra A[p^n]\]
is a closed embedding. In that case, we have $(\pi_{12}^s)^n(Q)=(A/G_{n},H)$, i.e. the set $(U_p^s)^n(Q)$ consists of the unique rigid point $Q_n=(A/G_n,H)$, where we have identified $H$ with its image in $A/G_n$.

\emph{(b)} For any integer $n\geq 2$, we have $Q\in V_{i,n}(\varepsilon)$ if and only if the point $Q_{n-1}=(A/G_{n-1},H)\in V_{i,1}(\varepsilon)$.

\end{lemma}
\begin{proof} This lemma is  a consequence of Prop. \ref{prop-pilloni}.
\end{proof}

The following technical Lemma will play an important role in the sequel, and it relies largely on the results proven in Appendix B.

\begin{lemma}\label{lem-tech-ss}
Let $\overline{\kappa}$ be an algebraic closure of $\kappa$, $\Pb$ be a superspecial closed point of $X_{\kappa}$, $\DD_{\Pb}(1)\subset V_{i,1}$ be the corresponding closed polydisc by \eqref{iso-V_i}. Then there exist local parameters $t_{\Pb,1}, t_{\Pb,2}$ of $\DD_{\Pb}(1)$ defined over $W[\overline{\kappa}][1/p]$ such that we have, for any $n\geq 2$,
\begin{align*}\DD_{\Pb}(1)\cap V_{i,n}(\varepsilon)&=\{Q\in \DD_{\Pb}(1)\;|\;  v_p(t_{\Pb,i+1}(Q))\geq n-1+\varepsilon\}\\
\DD_{\Pb}(1)\cap V_{i,n}&=\{Q\in \DD_{\Pb}(1)\;|\;  v_p(t_{\Pb,i+1}(Q))\geq n\}.
\end{align*}
In particular, $V_{i,n}(\varepsilon)$ is a strict neighborhood of $V_{i,n}$.
\end{lemma}

\begin{proof}
We may assume $i=1\in \Z/2\Z$ to simplify the notation. We consider first the case $n=2$.
Let $\Ab$ be the HBAV corresponding to $\Pb$. We will consider $\Ab$ as a HBAV  over $\overline{\kappa}$. Let $G_{\Pb}=\Ab[p^\infty]$ be the associated $p$-divisible group. In the terminology of Appendix  \ref{p-div-RM}, $G_{\Pb}$ is a superspecial $p$-divisible group with RM by $\Z_{p^2}$. By Serre-Tate's theory on the deformations of abelian varieties, the completion of the local ring $\widehat{\cO}_{X_{\kappab},\Pb}\simeq W(\kappab)[[t_{\Pb,1},t_{\Pb,2}]]$ is canonically identified with the universal deformation ring $R^\univ$ of $G_{\Pb}$. We take local lifts of partial Hasse invariants $t_{\Pb,1}=T_1, t_{\Pb,2}=T_2$ in  $\widehat{\cO}_{X_{\kappab},\Pb}$ as in \ref{p-div-RM}. The subdisc $\DD_{\Pb}(1)$ is the rigid subspace $\DD(1,1)$ defined in \ref{rem-window} of the rigid generic fiber of the deformation space of $G_{\Pb}$. Let $K$ be a finite extension of $W(\kappab)[1/p]$ with ring of integers $\cO_K$, $Q=(A,H)$ be a $K$-valued rigid point of $\DD_{\Pb}(1)$. We denote by  $(A,H,H'_0)$  the unique point above $Q$ in $\cC(p)^{s}_{\rig}$, and by $Q_1=(A/H'_0,A[p]/H'_0)$ the unique rigid point in $U_p^s(Q)$. By Remark \ref{rem-window},  the subgroup $H$, $H'_0$ are respectively just the pull-back to $\cO_K$ via $Q$ of the subgroup $H^{1,1}_{+}, H_{-}^{1,1}$ obtained in \ref{prop-window}. Now Prop. \ref{prop-window}(b) implies that $pt_{\Pb,1}(Q)$ and $t_{\Pb,2}(Q)/p$ lifts the partial Hasse invriants of $(A/H'_0)\otimes_{\cO_K}\cO_{K}/p$; in particular,  we have $w_1(Q_1)=1$ and $w_2(Q_1)=\max\{1,v_p(t_{\Pb,2}(Q)/p)\} $. By definition, we have $Q\in V_{2,i}(\varepsilon)$ if and only if $Q_1\in V_{1,i}(\varepsilon)$, i.e. $v_p(t_{\Pb,2}(Q))\geq 1+\varepsilon$.  Similarly, we have $Q\in V_{2,i}$ if and only if $Q_1\in V_{1,i}$. This proves this Lemma for $n=2$. The general case  follows by an easy induction on $n$.

\end{proof}

\begin{lemma}\label{lem-est-V_n} Let $\vk=(k_1,k_2)\in \Z^2$ with $k_1\geq k_2$, and $f$ be a section of $\omega^{\vk}$ over $\cV_{\can}$ such that $|f|_{\cV_{\can}}$ is finite.

\emph{(a)} For any integer $n\geq 1$,  the form $g_n=\frac{1}{a_p^n}(U_p^s)^n(f)$ is a well-defined section of $\omegab^{\vk}$ over $V_{i,n}(\varepsilon)- V_{i,n}$. We have
\[|g_n|_{V_{i,n}(\varepsilon)- V_{i,n}}\leq p^{-n(k_2-v_p(a_p)-2)+(1-\varepsilon)(p-1)k_2}|f|_{\cV_{\can}}.\]

\emph{(b)} For any integer $n\geq 1$, the form $h_n=\frac{1}{a_p^n}(U_p^s)^{n-1}(U^{\circ}_p(f))$ is a well-defined section of $\omegab^{\vk}$ over $V_{i,n}(\varepsilon)$. We have
\[|h_n|_{V_{i,n}(\varepsilon)}\leq p^{-n(k_2-v_p(a_p)-2)+(p-1)k_2/(p+1)}|f|_{\cV_{\can}}. \]
\end{lemma}
\begin{proof}
 (a) By definition, we have
$$(\pi_{12}^{s})^n(V_{i,n}(\varepsilon)-V_{i,n})\subset \pi^s_{12}(V_{i,1}(\varepsilon)-V_{i,1})\subset \cV_{\can},$$ where the second inclusion used Corollary \ref{prop-dyn-ss}(b). Therefore, $g_n$ is well-defined over $V_{i,n}(\varepsilon)-V_{i,n}$. To finish the proof of (a), it suffices  to show that
  \begin{align*}|g_1|_{V_{i,1}(\varepsilon)-V_{i,1}}&\leq p^{-(k_2-2-v_p(a_p))+(1-\varepsilon)(p-1)k_2}|f|_{\cV_{\can}};\\
  |g_n|_{V_{i,n}(\varepsilon)-V_{i,n}}&\leq p^{-(k_2-2-v_p(a_p))}|g_{n-1}|_{V_{i,{n-1}}(\varepsilon)-V_{i,{n-1}}}\quad \text{for }n\geq 2.
  \end{align*}
  Let $Q=(A,H)$ be a rigid point of $V_{i,n}(\varepsilon)-V_{i,n}$ over a finite extension $L/\Q_{\kappa}$, $H_0'$ be the unique $(\cO_F/p)$-cyclic subgroup of $A[p]$ given by \ref{lem-subgp-ss}(b), and $\omega$ be a basis of $\omega_{A/\cO_L}$ as a $(\cO_{L}\otimes\cO_F)$-module. Consider first the case $n=1$. By \eqref{equ-U_p-s} and Lemma \ref{lemma-basic-est}, we have \begin{align*}|g_1(A,H,\omega)|&=|\frac{1}{p^2a_p}f(A/H_0',A[p]/H_0',p^{-1}\hat{\phi}^*\omega)|\\
  &\leq p^{2+v_p(a_p)-(k_1\deg_1(H_0')+k_2\deg_{2}(H_0'))}|f|_{\cV_{\can}}\\
  &\leq p^{2+v_p(a_p)-k_2(\deg_1(H_0')+\deg_2(H_0'))}|f|_{\cV_{\can}}.
  \end{align*}
  By Prop. \ref{lem-subgp-ss}(b), we have
  \[\deg_1(H_0')+\deg_2(H_0')\geq 1-(p-1)(1-\varepsilon).\]
  from which the desired estimation for $|g_1|_{V_{i,1}(\varepsilon)-V_{i,1}}$ follows. For $n\geq 2$ , we have similarly
  \[|g_n(A,H,\omega)|\leq p^{2+v_p(a_p)-k_2(\deg_1(H_0')+\deg_2(H_0'))}|g_{n-1}|_{V_{i,{n-1}}(\varepsilon)-V_{i,{n-1}}},\]
  and the estimation follows from the fact that $\deg_1(H_0')+\deg_2(H_0')=1$.

(b) Since  we have
\[U_p^{\circ}(U_p^s)^{n-1}(V_{i,n}(\varepsilon))\subset U_p^{\circ}(V_{i,1}(\varepsilon))\subset \cV_{\can},\]
we see that $h_n$ is well defined over $V_{i,{n}}(\varepsilon)$. To prove the estimation for $h_n$, it suffices to show that
\begin{align*}
|h_1|_{V_{i,1}(\varepsilon)}&\leq p^{-(k_2-v_p(a_p)-2)+(p-1)k_2/(p+1)}|f|_{\cV_{\can}}=p^{2+v_p(a_p)-2k_2/(p+1)}|f|_{\cV_{\can}},\\
|h_n|_{V_{i,n}(\varepsilon)}&\leq p^{-(k_2-v_p(a_p)-2)}|h_{n-1}|_{V_{i,n-1}(\varepsilon)}\quad \text{for }n\geq 2.
\end{align*}
The  estimation for $n\geq 2$ is exactly the same as in (a). For the case $n=1$,  we can  conclude in the same way  by using the fact that, if $Q=(A,H)\in V_{i,1}(\varepsilon)$, then  all the subgroups $H'\subset A[p]$ with $H'\neq H$ corresponding to $U_p^{\circ}$ have $\deg(H')=\deg_1(H')+\deg_2(H')=2/(p+1)$.
\end{proof}

Now we prove Theorem \ref{thm-quad}, i.e Theorem \ref{thm-main} in the case $g=2$ and $p$ inert in $F$. Since the scalar extension by a finite extension $L/\Q_{\kappa}$ is not essential, we may assume $L=\Q_{\kappa}$ in Theorem \ref{thm-quad}. We have the following

\begin{prop}\label{prop-cont-1}

 Let $f$ be an overconvergent $p$-adic Hilbert modular form of level $\Gamma_{00}(N)\cap\Gamma_0(p)$ and weight $\vk=(k_1,k_2)\in \Z^2$ with $k_1\geq k_2$, and $U_p(f)=a_pf$ for some $a_p\in \C_p^{\times}$. We put  $\cV=\,]W_{1,1}[\,\cup\, ]W_{2,2}[\,\cup \cV_{\can}$, and
 \[\cV_1=\coprod_{i\in \Z/2\Z} V_{i,1}=\{Q=(A,H)\in ]W_{1,1}\cup W_{2,2}[\;|\; w_1(A)=w_2(A)=1\}.\]

 \emph{(a)} The form $f$ extends uniquely to a section of $\omegab^{\vk}$ over $\cV-\cV_1$ such that $U_{p}(f)=a_pf$ remains true, $|f|_{\cV-\cV_1}$ is finite and
 \beq\label{equ-est-V_1}|f|_{\cV-\cV_1}\leq \max\{1,p^{2+v_{p}(a_p)-k_2/p}\}|f|_{\cV_{\can}}.\eeq

 \emph{(b)} Assume that  $v_p(a_p)<k_2-2$. Then $f$ extends further to a section of $\omegab^{\vk}$ over $\,]Y_{\kappa}-W_{\bB,\bB}[\,\cup\cV_{\can}$. Moreover, we have $U_{p}(f)=a_pf$ and
 \beq\label{est-cont-2}|f|_{]Y_{\kappa}-W_{\bB,\bB}[\cup\cV_{\can}}\leq p^{2+v_p(a_p)}|f|_{\cV_{\can}}.\eeq
\end{prop}

Before proving this Proposition, we note that Theorem \ref{thm-quad} follows immediately from  Prop. \ref{prop-exten} and \ref{prop-koecher}.

\begin{proof} (a) By Prop. \ref{prop-cont-can}, the form $f$ can be uniquely extended to a section of $\omegab^{\vk}$ over the canonical locus $\cV_{\can}$, and $|f|_{\cV_{\can}}$ is finite. We first define the candidate  extension of $f$ in a strict neighborhood of
$$\;]W_{1,1}\cup W_{2,2}[-\cV_1=\coprod_{i\in \Z/2\Z}(\,]W_{i,i}[\,-V_{i,1})$$ in $\cV-\cV_1$, and show that  it coincides  with the old $f$ over the overlap with $\cV_{\can}$.
 Let $0<\epsilon\leq 1/2$ be a rational number. We put for $i\in \Z/2\Z$
\[\,]W_{i,i}[^{\circ}_\epsilon=\{Q=(A,H)\in \fY_{\rig}\;|\; 0\leq \nu_i(Q)\leq \epsilon, \nu_{i+1}(Q)=1, 0\leq w_{i+1}(A)<1.\}\]
This is a strict neighborhood of $]W_{i,i}[-V_{i,1}$ in $\cV-\cV_1$. Then the admissible open subsets
  $\{\cV_{\can}, \,]W_{1,1}[^{\circ}_{\epsilon}, \,]W_{2,2}[^{\circ}_{\epsilon}\}$ form an admissible open covering of $\cV- \cV_1$. By Prop. \ref{prop-dyn-sg}(a) and \ref{prop-dyn-ss},  the image of $]W_{i,i}[^{\circ}_{\epsilon}$ under the Hecke correspondence $U_p$ is contained in $\cV_{\can}$. Hence, the form $F_i=\frac{1}{a_p}U_p(f)$ is well-defined on $]W_{i,i}[^{\circ}_{\epsilon}$, and we have
\[F_i|_{\,]W_{i,i}[^{\circ}_{\epsilon}\cap \cV_{\can}}=f|_{]W_{i,i}[^{\circ}_{\epsilon}\cap\cV_{\can}}.\]
In particular, the sections $\{f, F_1, F_2\}$ glue together to a section, still denoted by $f$,  of $\omegab^{\vk}$ over $\cV-\cV_{1}$. By our construction, it's clear that $U_p(f)=a_pf$ still holds. To prove \eqref{equ-est-V_1}, it suffices to show that for $i\in\Z/2\Z$
\[|f|_{]W_{i,i}[-V_{i,1}}=\sup_{Q\in ]W_{i,i}[-V_{i,1}}|\frac{1}{a_p}U_p(f)(Q)|\leq p^{2+v_p(a_p)-\frac{k_2}{p}}|f|_{\cV_{\can}}.\]
 By \ref{prop-subgp-sg} and \ref{lem-subgp-ss}, for any $Q=(A,H)\in\, ]W_{i,i}[-V_{i,1}$ and $(\cO_F/p)$-cyclic subgroup  $H'\subset A[p]$ with $H'\cap H=0$, we have
$\deg_1(H')+\deg_2(H')\geq k_2/p.$
 By our construction of $f$, the required estimation above follows from  Cor. \ref{cor-est-norm}.
This finishes the proof of (a).\\

 (b) The proof will be divided into 3 steps. At each step, we will always denote by $f$ the extension obtained in the previous step.

\textbf{Step 1. Extension to $\cV$.} We fix a rational number $\varepsilon$ with $\frac{p}{p+1}<\varepsilon \,<1$ as in \ref{sect-V_n}.  We have to show that, for $i\in \Z/2\Z$, there exists a section $F\in H^0(V_{i,1}(\varepsilon),\omegab^{\vk})$ such that $F|_{V_{i,1}(\varepsilon)-V_{i,1}}=f|_{V_{i,1}(\varepsilon)-V_{i,1}}$ and
\beq\label{est-norm-V_1}
|F|_{V_{i,1}(\varepsilon)}\leq p^{2+v_p(a_p)-\frac{2}{p+1}k_2}|f|_{\cV_{\can}}.
\eeq
We prove first that, for any integer $n\geq 2$, there exists a form $F_{n}\in H^0(V_{i,1}(\varepsilon)-V_{i,n},\omegab^{\vk})$ with $F_n|_{V_{i,1}(\varepsilon)-V_{i,1}}=f|_{V_{i,1}(\varepsilon)-V_{i,1}}$. We put $G_1=f|_{V_{i,1}(\varepsilon)-V_{i,1}}$ and
\[
G_n=\sum_{m=1}^{n-1}\frac{1}{a_p^m}(U_{p}^{s})^{m-1}U_p^{\circ}(f)+\frac{1}{a_p^{n-1}}(U_p^s)^{n-1}(f)
\]
for any integer $n\geq 2$. We note that $G_n$ is a well-defined section of $\omegab^{\vk}$ over $V_{i,n-1}(\varepsilon)-V_{i,n}$. Indeed, the first $n-1$ terms are even well defined over $V_{i,n-1}(\varepsilon)$ by \ref{lem-est-V_n}(b). To justify  the last term, we note that
\[(U_p^s)^{n-1}(V_{i,n-1}(\varepsilon)-V_{i,n})\subset U_p^s(V_{i,1}(\varepsilon)-V_{i,2})\subset \cV_{\can}\cup (]W_{i,i}[-V_{i,1})\subset \cV-\cV_1.\]
The same argument as in Lemma \ref{lem-est-V_n} shows that
\[|\frac{1}{a_p^{n-1}}(U_p^s)^{n-1}(f)|_{V_{i,n-1}(\varepsilon)-V_{i,n}}\leq p^{-(n-1)(k_2-2-v_p(a_p))}|f|_{\cV-\cV_1} < |f|_{\cV-\cV_1} ,\]
where the last step uses the assumption $k_2>2+v_p(a_p)$.
On the other hand,  \eqref{lem-est-V_n}(b) implies that the first $n-1$ terms in definition of $G_n$ are bounded by
\[\max_{1\leq m\leq n-1}\{p^{-m(k_2-2-v_p(a_p))+k_2(p-1)/(p+1)}\}|f|_{\cV_{\can}}=p^{2+v_p(a_p)-\frac{2}{p+1}k_2}|f|_{\cV_{\can}}.\]
Therefore,  for any $n\geq 1$, we have
\begin{align*}
|G_n|_{V_{i,n-1}(\varepsilon)-V_{i,n}}&\leq \max\{p^{2+v_p(a_p)-\frac{2}{p+1}k_2}|f|_{\cV_{\can}}, |f|_{\cV-\cV_1} \}\\ &\leq \max\{1, p^{2+v_p(a_p)-k_2/p}\}|f|_{\cV_{\can}},
\end{align*}
where we have used \eqref{equ-est-V_1} in the last step.
Using the functional equation $U_p(f)=a_pf$, it's easy to check that
$$G_n|_{V_{i,n-1}(\varepsilon)-V_{i,n-1}}=G_{n-1}|_{V_{i,n-1}(\varepsilon)-V_{i,n-1}}.$$
 Since  $\{V_{i,m-1}(\varepsilon)-V_{i,m}\}_{2\leq m\leq n}$ form an admissible open covering of $V_{i,1}(\varepsilon)-V_{i,n}$, we see that the forms $\{G_m\}_{2\leq m\leq n}$ glue together to a section $F_n$ of $\omegab^{\vk}$ over $V_{1,i}(\varepsilon)-V_{i,n}$ whose restriction to $V_{i,1}(\varepsilon)-V_{i,1}$ coincides with $f$. By the estimation above for $G_n$, we get
 \beq\label{equ-est-F_n}|F_n|_{V_{1,i}(\varepsilon)-V_{i,n}}\leq \max\{1, p^{2+v_p(a_p)-k_2/p}\}|f|_{\cV_{\can}}\eeq

 To obtain a form $F$ over $V_{1,i}(\varepsilon)$, we define
\[F_n'=\sum_{m=1}^n\frac{1}{a_p^m}(U_p^s)^{m-1}U_p^{\circ}(f)\]
over $V_{i,n}(\varepsilon)$. Then it follows from Lemma \ref{lem-est-V_n}(a) that
\begin{align*}|F_n-F'_n|_{V_{i,n}(\varepsilon)-V_{i,n}}&=|\frac{1}{a_{p}^n}(U_{p}^s)^{n}(f)|_{V_{i,n}(\varepsilon)-V_{i,n}}
\leq p^{-n(k_2-v_p(a_p)-2)+(1-\varepsilon)(p-1)k_2}|f|_{\cV_{\can}}.
\end{align*}
 Since $v_{p}(a_p)+2<k_2$, the estimation above tends to $0$ as $n$ tends to $\infty$. Note that  Lemma \ref{lem-est-V_n}(b) implies  that
$$ |F_n'|_{V_{i,n}(\varepsilon)}\leq p^{2+v_p(a_p)-\frac{2}{p+1}k_2}|f|_{\cV_\can}\leq \max\{1, p^{2+v_p(a_p)-\frac{1}{p}k_2}\}|f|_{\cV_{\can}}.$$
Applying the  gluing lemma \ref{lem-glue} to $U_0=V_{i,1}(\varepsilon)$, we get a unique section $F$ of $\omegab^{\vk}$ over $V_{i,1}(\varepsilon)$ whose restriction to $V_{i,1}(\varepsilon)-V_{i,n}$  coincides with $F_n$. Moreover, in view of the estimations for $|F_n|_{V_{i,1}-V_{i,n}}$ and $|F_n'|_{V_{i,n}(\varepsilon)}$, lemma \ref{lem-glue} also implies that  $|F|_{V_{i,1}(\varepsilon)}$ is bounded by $\max\{1, p^{2+v_p(a_p)-\frac{1}{p}k_2}|f|_{\cV_{\can}}\}$. This extends the form $f$ to $\cV$.
Combining \eqref{equ-est-V_1}, we obtain
 \begin{equation}\label{est-cont-1}
 |f|_{\cV}\leq\max\{1, p^{2+v_p(a_p)-\frac{1}{p}k_2}\}|f|_{\cV_{\can}} .
 \end{equation}

\textbf{Step 2: Extension to $\cV\cup]W_{1,\bB}[\cup ]W_{2,\bB}[$.}
For an interval $I\subset (0,1]$ and $i\in \Z/2\Z$, we put
\begin{align*}
]W_{i,\bB}[_I&=\{Q\in \,]W_{i,\bB}[\cup\, ]W_{i,i}^{sg}[\;|\; \nu_{i}(Q)=0, \nu_{i+1}(Q)\in I\},\\
]W_{\star, \bB}[_{I}&=]W_{1,\bB}[_I\cup ]W_{2,\bB}[_I.
\end{align*}
By Prop. \ref{prop-dyn-sg}, the image of  $]W_{\star,\bB}[_{(\frac{1}{p},1]}$ under the Hecke correspondence $U_p$ is contained in $\cV_{\can}$. The form $\frac{1}{a_p}U_p(f)$ is thus well defined over $]W_{\star,\bB}[_{(\frac{1}{p},1]}$, and it extends naturally $f$.  For any rigid point $Q=(A,H)$ in $]W_{\star,\bB}[_{(\frac{1}{p},1]}$, Prop. \ref{prop-subgp-sg} implies that
\[\deg(H')=\deg_1(H')+\deg_2(H')\geq \frac{1}{p}\]
for any  $(\cO_F/p)$-cyclic subgroup $H'\subset A[p]$ with $H\cap H'=0$. Therefore, it follows from Cor. \ref{cor-est-norm}  that
\begin{equation}\label{est-f-1}
|f|_{]W_{\star,\bB}[_{(\frac{1}{p},1]}}=\sup_{Q\in ]W_{\star,\bB}[_{(\frac{1}{p},1]}}|\frac{1}{a_p}(U_pf)(Q)|\leq p^{2+v_p(a_p)-k_2/p}|f|_{\cV_{\can}}.
\end{equation}
Next, we extend $f$ to the remaining part of $]W_{1,\bB}[\cup\,]W_{2,\bB}[$. Let  $\epsilon_n=\frac{p+1}{p^n(p^2+1)}$ for any integer $n\geq 0$. The quasi-compact open subsets $\{\,]W_{\star,\bB}[_{[\epsilon_n,\epsilon_{n-1}]}\,\}_{n\geq 1}$ form an admissible open covering of $]W_{\star,\bB}[_{(0,\epsilon_0]}$. Note that $\epsilon_0>\frac{1}{p}$, and that $]W_{\star,\bB}[_{[\epsilon_n,\epsilon_{n-1}]}$ for $n\geq 2$ is contained in the anti-canonical locus \eqref{anti-can}
\[\cW=\{Q\in \fY_{\rig}\,|\, \nu_{i}(Q)+p\nu_{i-1}(Q)>1 \quad \forall i\in \Z/2\Z.\}\]
 By Prop. \ref{prop-dyn-sg} and \ref{prop-hecke-can}, the $U_p$ sends $]W_{\star,\bB}[_{[\epsilon_1,\epsilon_0]}$ into
 $$\cV\cup\,]W_{\star,\bB}[_{[\epsilon_0,1]}= \cV_{\can}\cup \,]W_{1,1}[\,\cup \,]W_{2,2}[\cup ]W_{\star,\bB}[_{[\epsilon_0,1]},$$ where the form $f$ has been defined, and it sends $]W_{\star,\bB}[_{[\epsilon_{n+1},\epsilon_{n}]}$ into
 $]W_{\star,\bB}[_{[\epsilon_{n},\epsilon_{n-1}]}$ for $n\geq 1.$
 Therefore, we can define inductively a form $f_n$ on $]W_{\star,\bB}[_{[\epsilon_n,\epsilon_{n-1}]}$ by putting $f_1=\frac{1}{a_p}U_p(f)$ and $f_{n+1}=\frac{1}{a_p}U_p(f_{n})$ for $n\geq 1$. It is easy to see that $$f_1|_{]W_{\star,\bB}[_{(\frac{1}{p},\epsilon_0]}}=f|_{]W_{\star,\bB}[_{(\frac{1}{p},\epsilon_0]}},$$
 and   the forms $\{f_n\}_{\geq 1}$ coincide with each other over the overlaps of their definition domains. This proves that $f$ extends to a section of $\omegab^{\vk}$ over $\cV\cup\,]W_{\star,\bB}[_{(0,1]}$. By \ref{prop-hecke-can}(b) and \eqref{prop-subgp-sg}, for any rigid point $Q=(A,H)\in ]W_{\star,\bB}[_{[\epsilon_1,\epsilon_0]}$, every $(\cO_F/p)$-cyclic subgroup $H'\subset A[p]$ with $H'\cap H=0$ has degree
 \[\deg(H')=\deg_1(H')+\deg_2(H')\geq \epsilon_1.\]
 It results from Cor. \ref{cor-est-norm} that
 \begin{align}\label{equ-est-eps-1}
 |f|_{]W_{\star,\bB}[_{[\epsilon_1,\epsilon_0]}}&=\sup_{Q\in ]W_{\star,\bB}[_{[\epsilon_1,\epsilon_0]}}|\frac{1}{a_p}(U_pf)(Q)|\\
 &\leq p^{2+v_p(a_p)-\epsilon_1 k_2}|f|_{\cV\cup\,]W_{\star,\bB}[_{[\epsilon_0,1]}}\nonumber\\
 &\leq p^{2+v_p(a_p)-\epsilon_1k_2}\max\{1,p^{2+v_p(a_p)-\frac{1}{p}k_2}\}|f|_{\cV_{\can}},\nonumber
 \end{align}
 where we have used \eqref{est-cont-1} and \eqref{est-f-1} in  the last inequality.

 \textbf{Step 3.} It remains to extend $f$ to $]W_{\emptyset,\bB}[$. We denote by $\cC(p)_{\rig}|_{\cW}$ the inverse image of the anti-canonical locus $\cW$ by $\pi_1:\cC(p)_{\rig}\ra \fY_{\rig}$. Prop. \ref{prop-hecke-can}(b) implies that we have a decomposition of rigid analytic spaces
 \[\cC(p)_{\rig}|_{\cW}=\cC(p)_{\rig}^c\coprod \cC(p)_{\rig}^{a}.\]
Here, for every rigid point $Q=(A,H)\in \cW$, $\pi_1^{-1}(Q)\cap \cC(p)^c_{\rig}$ consists of the single point  $(A,H,C)$, where $C\subset A[p]$ is the canonical subgroup; and $\pi^{-1}_{1}(Q)\cap\cC(p)_{\rig}^a$ corresponds to the other $p^2-1$ anti-canonical subgroups. Correspondingly, we have a decomposition of Hecke correspondences $U_p=U_p^c+U_p^a$. For $\epsilon_n=\frac{p+1}{p^n(p^2+1)}$ as above, we put
\[]W_{\emptyset,\bB}[_{\epsilon_n}=\,]W_{\emptyset,\bB}[\,\cup\,]W_{\star,\bB}[_{(0,\epsilon_n]}.\]
They are strict neighborhoods of $]W_{\emptyset,\bB}[$ in $]Y_{\kappa}-W_{\bB,\bB}[$. By Prop. \ref{prop-hecke-can},  we have
$$U_p^c(]W_{\emptyset,\bB}[_{\epsilon_n})\subset ]W_{\emptyset,\bB}[_{\epsilon_{n-1}}\quad \text{for }n\geq 1,$$ and $U_p^a(]W_{\emptyset,\bB}[_{\epsilon_n})\subset \cV_{\can}$. We define  a section of $\omegab^{\vk}$ on $]W_{\emptyset,\bB}[_{\epsilon_n}$ by
\begin{align*}
g_n&=\sum_{m=1}^n \frac{1}{a_p^m}(U_p^c)^{m-1}U_p^{a}(f).
\end{align*}
By Prop. \ref{prop-hecke-can}, for any rigid point $Q=(A,H)\in ]W_{\emptyset, \bB}[_{\epsilon_n}$, the canonical subgroup $C\subset A[p]$ has degree
\[\deg(C)=\deg_1(C)+\deg_2(C)\geq 2-p\epsilon_n>1.\]
 We deduce from \ref{lemma-basic-est} that, for $m>1$ and  $Q\in ]W_{\emptyset, \bB}[_{\epsilon_n}$,
  \begin{align*}|\frac{1}{a_p^m}(U_p^c)^{m-1}(U_p^a)(f)(Q)|
  &\leq p^{2+v_p(a_p)-(k_1\deg_1(C)+k_2\deg_{2}(C))}|\frac{1}{a_p^{m-1}}(U_p^c)^{m-2}(U_p^a)(f)|_{]W_{\emptyset,\bB}[_{\epsilon_{n-1}}}\\
  &\leq p^{2+v_p(a_p)-k_2(2-p\epsilon_n)}|\frac{1}{a_p^{m-1}}(U_p^c)^{m-2}(U_p^a)(f)|_{]W_{\emptyset,\bB}[_{\epsilon_{n-1}}}.
  \end{align*}
 Since $2+v_p(a_p)-k_2(2-p\epsilon_n)<2+v_p(a_p)-k_2<0$, we get
  \begin{align}\label{est-3-c}
  |\frac{1}{a_p^m}(U_p^c)^{m-1}(U_p^a)(f)|_{]W_{\emptyset, \bB}[_{\epsilon_n}}&\leq p^{2+v_p(a_p)-k_2(2-p\epsilon_n)}|\frac{1}{a_p^{m-1}}(U_p^c)^{m-2}(U_p^a)(f)|_{]W_{\emptyset,\bB}[_{\epsilon_{n-1}}}\\
  &<|\frac{1}{a_p^{m-1}}(U_p^c)^{m-2}(U_p^a)(f)|_{]W_{\emptyset,\bB}[_{\epsilon_{n-1}}}.\nonumber
  \end{align}
 Moreover, it follows trivially from \ref{lemma-basic-est} that
$|\frac{1}{a_p}U_p^a(f)|_{]W_{\emptyset, \bB}[_{\epsilon_{n}}}\leq p^{2+v_p(a_p)}|f|_{\cV_{\can}}.$
 Hence,  we get
\begin{align}\label{est-g_n}
|g_n|_{]W_{\emptyset,\bB}[_{\epsilon_n}}&\leq \max_{1\leq m\leq n} |\frac{1}{a_p^m}(U_p^c)^{m-1}U_p^a(f)|_{]W_{\emptyset, \bB}[_{\epsilon_{n}}}= |\frac{1}{a_p}U^a_p(f)|_{]W_{\emptyset, \bB}[_{\epsilon_{n}}}\\
&\leq p^{2+v_p(a_p)}|f|_{\cV_{\can}}.\nonumber
\end{align}
Using the functional equation $f=\frac{1}{a_p}U_p(f)$, we see that $g_n-f=\frac{1}{a_p^n}(U_p^c)^n(f)$. A similar argument as in \eqref{est-3-c} shows that for any $n\geq 1$
\begin{align*}
|f-g_n|_{]W_{\star,\bB}[_{[\epsilon_{n+1},\epsilon_n]}}&=|\frac{1}{a_p^n}(U_p^c)^n(f)|_{]W_{\star,\bB}[_{[\epsilon_{n+1},\epsilon_n]}}\\
&\leq p^{2+v_p(a_p)-k_2(2-p\epsilon_n)}|\frac{1}{a_p^{n-1}}(U_p^c)^{n-1}(f)|_{]W_{\star,\bB}[_{[\epsilon_{n},\epsilon_{n-1}]}}\\
&\leq p^{n(2+v_p(a_p))-k_2(2n-p\sum_{r=1}^n\epsilon_r)}|f|_{]W_{\star,\bB}[_{[\epsilon_1,\epsilon_0]}}.
\end{align*}
As $2+v_p(a_p)<k_2$, we check easily for any $n\geq 1$ that
\begin{align*}n(2+v_p(a_p))-k_2(2n-p\sum_{r=1}^n\epsilon_r)
<-k_2n+k_2\frac{(p+1)p}{(p^2+1)(p-1)}(1-\frac{1}{p^n})<0.
\end{align*}
Therefore, we get $|f-g_n|_{]W_{\star,\bB}[_{[\epsilon_{n+1},\epsilon_n]}}\leq |f|_{]W_{\star,\bB}[_{[\epsilon_1,\epsilon_0]}}$ and
\[|f-g_n|_{]W_{\star,\bB}[_{[\epsilon_{n+1},\epsilon_n]}}\ra 0\quad\text{ as } n\ra \infty.\]
 From \eqref{equ-est-eps-1}, it follows  that $$|f-g_n|_{]W_{\star,\bB}[_{[\epsilon_{n+1},\epsilon_n]}}<p^{2+v_p(a_p)}|f|_{\cV_{\can}}.$$
 In view of \eqref{est-g_n}, we deduce that
 $|f|_{]W_{\star,\bB}[_{[\epsilon_{n+1},\epsilon_n]}}\leq p^{2+v_p(a_p)}|f|_{\cV_{\can}}$ for every $n\geq 1$. Combining with  the estimations \eqref{est-cont-1}, \eqref{est-f-1} and \eqref{equ-est-eps-1}, we see that
 $$|f|_{\cV\cup ]W_{\star,\bB}[_{(0,1]}}\leq p^{2+v_p(a_p)}|f|_{\cV_{\can}}.$$
Now the assumptions of Lemma \ref{lem-glue} are satisfied for $X_K=\fY_{\rig}$ with $U_0=]Y_{\kappa}-W_{\bB,\bB}[$, $U_n=]W_{\emptyset,\bB}[_{\epsilon_n}$ for $n\geq 1$, $V_n=U_{n+1}$, $F_n=f|_{]Y_{\kappa}-W_{\bB,\bB}[-V_n}$, $F_n'=g_n$ and $C=p^{2+v_p(a_p)}|f|_{\cV_{\can}}$.  We deduce that there exists a unique section $f$ of $\omegab^{\vk}$ over $]Y_{\kappa}-W_{\bB,\bB}[$, hence also over $]Y_{\kappa}-W_{\bB,\bB}[\cup \cV_{\can}$, which extends $f$ and is bounded by $p^{2+v_p(a_p)}|f|_{\cV_{\can}}$. This completes the proof.

\end{proof}

\section{The general case of Theorem \ref{thm-main}}
In this section, we indicate how to generalize the arguments in the preceding section to prove Theorem \ref{thm-main} in the general case.

\subsection{}We denote by $\Sigma$ the set of all prime ideals of $\cO_F$ above $p$, and we assume $[\kappa(\fp):\F_{p}]\leq 2$ for all $\fp\in \Sigma$. Let $\Sigma_1$ be the subset of $\Sigma$ consisting of primes of degree $1$, and $\Sigma_2\subset \Sigma$ be the subset of primes of degree 2. Since Theorem \ref{thm-main} in the case $\Sigma=\Sigma_1$ was treated by Sasaki \cite{Sa}, we always suppose that $\Sigma_2\neq \emptyset$. We have a partition
$$\bB=(\coprod_{\fp\in \Sigma_1}\bB_{\fp})\coprod(\coprod_{\fq\in \Sigma_2}\bB_{\fq}).$$
For $\fp\in \Sigma_1$, we denote by $\beta_{\fp}$ the unique element of $\bB_{\fp}$, and $\nu_{\fp}$ the corresponding valuation of $\fY_{\rig}$ \eqref{defn-val-Y}; for $\fq\in \Sigma_2$, we denote by $\bB_{\fq}=\{\beta_{\fq,1},\beta_{\fq,2}\}$, and $\nu_{\fq,1},\nu_{\fq,2}$ be the corresponding  two valuations.
Let  $\bI=[0,1]$ and $\bJ=[0,1]^2$. We  have  a valuation map
 $$\nu_{\fY}=((\nu_{\fp})_{\fp\in \Sigma_1},(\nu_{\fq,1},\nu_{\fq,2})_{\fq\in \Sigma_2}):\fY_{\rig}\ra \bI^{\Sigma_1}\times \bJ^{\Sigma_2}.$$
 In general, if  $\Omega$ is a subset of $\bI^{\Sigma_1}\times\bJ^{\Sigma_2}$, we put
 $$\fY_{\rig}(\Omega)=\{Q\in \fY_{\rig}\;|\;\nu_{\fY}(Q)\in\Omega\}.$$

Let  $\{S_{\fp}:\fp\in \Sigma\}$ be a collection of non-empty subsets of either $\bI$ or $\bJ$. If all the  $S_\fp$ are closed and connected in the usual real topology,  the corresponding rigid subspace $\fY_{\rig}(\prod_{\fp\in \Sigma}S_\fp)$ is a quasi-compact admissible open subset of $\fY_{\rig}$. For example, if $\one=(1,1,\cdots,1)\in \bI^{\Sigma_1}\times \bJ^{\Sigma_2}$, then $\fY_{\rig}(\one)$ is the ordinary locus $\fY_{\rig}^{\ord}$. Hence, if $\prod_{\fp\in \Sigma} S_\fp$ is a neighborhood of $\one$ in $\bI^{\Sigma_1}\times \bJ^{\Sigma_2}$, then $\fY_{\rig}(\prod_{\fp\in \Sigma}S_{\fp})$ is a strict neighborhood of the ordinary locus $\fY_{\rig}^{\ord}$ in $\fY_{\rig}$.

\begin{lemma}\label{lem-tube}
 Let $\bI=[0,1]$, and $\bJc$ be the boundary of $[0,1]^2$, i.e. the union of its four closed edges.  Then $\fY_{\rig}(\bI^{\Sigma_1}\times \bJc^{\Sigma_2})$ is the tube over an open subset of $Y_{\kappa}$ whose complement has codimension $2$.
 \end{lemma}
 \begin{proof}
 We consider  Goren-Kassaei's stratification $\{W_{\varphi,\eta}\}$ on $Y_{\kappa}$ defined in \eqref{strat-Y}, where $(\varphi,\eta)$ runs through all the admissible pairs of subsets of $\bB$. For each $\fp\in \Sigma$, we put $\varphi_{\fp}=\varphi\cap\bB_{\fp}$ and $\eta_{\fp}=\eta\cap \bB_{\fp}$. Then we have $\varphi=\coprod_{\fp\in\Sigma}\varphi_{\fp}$ and $\eta=\coprod_{\fp\in \Sigma}\eta_{\fp}$. Similarly, for each prime $\fp$, the pair  $(\varphi_{\fp},\eta_{\fp})$ of subsets of $\bB_{\fp}$ can be called admissible in the sense that $\varphi_{\fp}\supset \sigma(\eta_{\fp}^c)$, where $\eta_{\fp}^c$ denotes the complementary subset $\bB_{\fp}- \eta_{\fp}$, and $\sigma:\bB_{\fp}\ra \bB_{\fp}$ is the action of the Frobenius. In particular, we have  $$|\varphi_{\fp}|+|\eta_{\fp}|\geq |\bB_{\fp}|=[\kappa(\fp):\F_p].$$ Now, let $Z$ be the union of all the strata $W_{\varphi,\eta}$ such that there exists at least one $\fp\in \Sigma_2$ with $(\varphi_{\fp},\eta_{\fp})=(\bB_{\fp},\bB_{\fp})$. By \cite[2.5.2]{GK} (or cf. \ref{strat-Y}), $Z$ is closed in $Y_{\kappa}$, and  each stratum $W_{\varphi,\eta}$ in $Z$ has dimension
 $$2g-|\varphi|-|\eta|=2g-\sum_{\fp\in \Sigma}(|\varphi_{\fp}|+|\eta_{\fp}|)\leq g-2,$$
 i.e. $Z$ has codimension $2$ in $Y_{\kappa}$. Now one checks easily that $\fY_{\rig}(\bI^{\Sigma_1}\times \bJc^{\Sigma_2})$ identifies with the tube over $Y_{\kappa}-Z$. This proves the Lemma.
 \end{proof}

\subsection{} We fix a prime $\fp\in \Sigma_1$ of degree $1$.  We describe the Hecke correspondence  $U_{\fp}$  on $\fY_{\rig}$. Let $L$ be a finite extension of $\Q_{\kappa}$, and $P$ be a $L$-valued rigid point of $\fX_{\rig}$ corresponding to a HBAV over $\cO_L$. We denote by $w_{\fp}(A)$ the  partial Hodge height of $A$ \eqref{equ-part-hodge} corresponding to the unique element of $\bB_{\fp}$, and $A[\fp^{\infty}]$ be the $\fp$-component of $A[p^{\infty}]$. This is a $p$-divisible group of  dimension $1$ and height $2$, since $\fp$ has degree $1$. Thus for any  $(\cO_F/\fp)$-cyclic isotropic subgroup $H$ of $A[\fp]$, its $\fp$ component $H[\fp]$ is just a subgroup of order $p$ in the $p$-divisible groups $A[\fp^{\infty}]$. The possibilities for such subgroups are well analyzed  by Katz and Lubin \cite{Ka72}, and widely used in the work of \cite{Bu}, \cite{Ks} and \cite{Sa}. We summarize the results in our language as follows.

\begin{lemma}[Lubin-Katz]\label{lem-dim-1}
Let $\fp\in \Sigma_1$ as above,  $Q=(A,H)$ be an $L$-valued rigid point of $\fY_{\rig}$. Then we have the following possibilities:

\underline{}\emph{(a)} Assume $\frac{1}{p+1}<\nu_{i}(Q)\leq 1$, i.e. $Q$ is canonical at $\fp$ in the sense of Theorem \ref{thm-GK-p}. Then we have $w_{\fp}(A)=1-\nu_{\fp}(Q)< \frac{p}{p+1}$, and all the  subgroups $H'\subset A[\fp]$ of order $p$ different from $H[\fp]$  has $\deg(H')=\frac{1}{p}(1-\nu_{\fp}(Q))$. Or equivalently, all the points $Q_1\in U_{\fp}(Q)$ satisfy $1-\nu_{\fp}(Q_1)=\frac{1}{p}(1-\nu_{\fp}(Q))$, i.e.
\[\nu_{\fp}(Q_1)=\frac{p-1}{p}+\frac{1}{p}\nu_{\fp}(Q)>\frac{p}{p+1}.\]

\emph{(b)} Assume $\nu_{\fp}(Q)=\frac{1}{p+1}$, i.e. $Q$ is too-singular at $\fp$. Then we have $w_{\fp}(A)\geq \frac{p}{p+1}$, and all the subgroups $H'\subset A[\fp]$ of order $p$ different from $H$  has $\deg(H')=\frac{1}{p+1}$, or equivalently, we have
 $\nu_{i}(Q_1)=\frac{p}{p+1}$ for all $Q_1\in U_{\fp}(Q)$.

 \emph{(c)} If $0\leq \nu_{\fp}(Q)<\frac{1}{p+1}$, i.e. $Q$ is anti-canonical at $\fp$ in the sense of Theorem \ref{thm-GK-p}. Then we have $w_{\fp}(A)=p\nu_{\fp}(Q)<\frac{p}{p+1}$. There exists a unique subgroup $C\subset A[\fp]$ of order $p$ different from $H$ with  $\deg(C)=1-p\nu_{\fp}(Q)$,  and the other subgroups $H'\subset A[\fp]$ has $\deg(H')=\nu_{\fp}(Q)$. Equivalently, there exists a unique point $Q_1\in U_{\fp}(Q)$ with $\nu_{\fp}(Q_1)=p\nu_{\fp}(Q)<\frac{p}{p+1}$, and all the other $p-1$ points $Q_1'\in U_{\fp}(Q)$ satisfy $\nu_{\fp}(Q_1')=1-\nu_{\fp}(Q)> \frac{p}{p+1}$.
\end{lemma}
\begin{proof} Indeed, all these results are direct consequences of \cite[3.10.7]{Ka72}. Statement (a) and (c) are also special cases of Prop. \ref{prop-hecke-can}.
\end{proof}

We have the following  proposition on the analytic continuation of a $U_{\fp}$-eigenform form, due to essentially Kassaei \cite{Ks}.

\begin{prop}\label{prop-dim-1}
 Let $\fp\in \Sigma_1$ as above, and  $\prod_{\fq\in \Sigma}S_{\fq}\subset \bI^{\Sigma_1}\times \bJ^{\Sigma_2}$ be  a closed and connected neighborhood of $\one$.  Let $\vk\in \Z^{\bB}$, $k_{\fp}\in \Z$ be its $\fp$-th component, and $f$ be a section of $\omegab^{\vk }$ over  $\fY_{\rig}(\prod_{\fq\in \Sigma} S_{\fq})$.
 Assume that $f$ is a $U_{\fp}$-eigenform of weight $\vk$ with eigenvalue $a_{\fp}\neq 0$.

\emph{(a)} The form $f$ extends uniquely to an eigenform of $U_{\fp}$ with eigenvalue $a_{\fp}$ over $\fY_{\rig}(\bI_{>0}\times\prod_{\fq\neq \fp}S_{\fq})$, where $\bI_{>0}=(0,1]$ is considered as a subset of the $\fp$-th copy of $\bI$.

 \emph{(b)} If $v_p(a_{\fp})<k_{\fp}-1$, then $f$ extends uniquely to an eigenform of $U_{\fp}$ with eigenvalue $a_{\fp}$ over $\fY_{\rig}(\bI\times \prod_{\fq\neq \fp}S_{\fq})$.
\end{prop}
\begin{proof} The arguments are  the same as in \cite[4.1]{Ks} and \cite{Sa}. We include a proof here for completeness.

(a) For any integer $n\geq0$, we put $\epsilon_n=\frac{1}{p^{n-1}(p+1)}$.  By assumption, $S_{\fp}$ is a neighborhood of $1\in \bI=[0,1]$. We may assume thus $[1-\epsilon_N]\subset S_{\fp}$ for some integer $N>0$. By Lemma \ref{lem-dim-1}(a) and (b),  the image of $\fY_{\rig}([1-\epsilon_{n},1]\times \prod_{\fq\neq\fp}S_{\fq})$ under the correspondence $U_{\fp}$ is contained in $\fY_{\rig}([1-\epsilon_{n+1},1]\times \prod_{\fq\neq\fp}S_{\fq})$. Therefore, the form $\frac{1}{a_{\fp}^N}U_{\fp}^N$  is well defined over $\fY_{\rig}([1-\epsilon_0,1]\times \prod_{\fq\neq \fp}S_{\fq})$. Note that $1-\epsilon_0=\epsilon_1=\frac{1}{p+1}$. This gives the unique extension of $f$ over $\fY_{\rig}([\epsilon_1,1]\times \prod_{\fq\neq \fp}S_{\fq})$. By Lemma \ref{lem-dim-1}(c), we have
\[U_{\fp}\bigl(\fY_{\rig}([\epsilon_{n+1},1]\times \prod_{\fq\neq \fp}S_{\fq})\bigr)\subset  \fY_{\rig}([\epsilon_{n},1]\times\prod_{\fq\neq \fp }S_{\fq}).\]
Using the functional equation $f=\frac{1}{a_{\fp}}U_{\fp}(f)$, we can extend inductively $f$ to $\fY_{\rig}([\epsilon_n,1]\times \prod_{\fq\neq\fp}S_{\fq})$ for any $n\geq 1$.
As the quasi-compact admissible open subsets of $\{\fY_{\rig}([\epsilon_n,1]\times \prod_{\fq\neq\fp}S_{\fq}):n\geq 1\}$ form an admissible covering of $\fY_{\rig}(\bI_{>0}\times \prod_{\fq\neq\fp}S_{\fq})$, this finishes the proof of (a).

(b) We proceed in the same way as in Step 3 of Prop. \ref{prop-cont-1}. We put $W=\fY_{\rig}([0,\frac{1}{p+1})\times \prod_{\fq\neq\fp}S_{\fq})$. This is the analogue of the anti-canonical locus in our situation. Denote by $\cC(\fp)_{\rig}|_{W}$ the inverse image  of $W$ via the first projection of Hecke correspondence $\pi_1:\cC(\fp)_{\rig}\ra \fY_{\rig}$ \eqref{defn-set-U_p}. By Lemma \ref{lem-dim-1}(c), we have a decomposition
\[\cC(\fp)_{\rig}|_{W}=\cC(\fp)_{\rig}^{c}\coprod \cC(\fp)^{a}_{\rig},\]
where $\cC(\fp)_{\rig}^{c}$ corresponds to the distinguished canonical subgroup $C\subset A[p]$ above a point $Q=(A,H)\in W$ via $\pi_1$, and $\cC(\fp)^{a}_{\rig}$ corresponds to the remaining subgroups. Consequently, we have a similar decomposition for the $U_{\fp}$-operator $U_{\fp}=U_{\fp}^{c}+U_{\fp}^a$. By Lemma \ref{lem-dim-1}(c), we have
 \begin{gather*}
U_{\fp}^{c}(\fY_{\rig}([0,\epsilon_n]\times \prod_{\fq\neq \fp}S_{\fq}))\subset \fY_{\rig}([0,\epsilon_{n-1}]\times \prod_{\fq\neq\fp}S_{\fq})\quad \text{for any } n\geq 2,\\
U_{\fp}^a(\fY_{\rig}([0,\epsilon_n]\times \prod_{\fq\neq\fp}S_{\fq}))\subset \fY_{\rig}([\epsilon_0,1]\times \prod_{\fq\neq\fp}S_{\fq}).
 \end{gather*}
Hence, the section
\[g_n=\sum_{m=1}^{n}\frac{1}{a_{\fp}^m}(U_{\fp}^c)^{m-1}U_{\fp}^a(f)\]
is well defined over $\fY_{\rig}([0,\epsilon_{n+1}]\times \prod_{\fq\neq\fp}S_{\fq})$ for $n\geq 2$. By \ref{lem-dim-1}(c), for any rigid point $Q=(A,H)\in \fY_{\rig}([0,\epsilon_{n+1}]\times \prod_{\fq\neq\fp}S_{\fq})$, the canonical subgroup $C\subset A[\fp]$ has degree
\[\deg(C)=1-p\nu_{\fp}(Q)\geq 1-p\epsilon_{n+1}=1-\epsilon_n.\]
Using Lemma \ref{lemma-basic-est}, we deduce that for $2\leq m\leq n$
\begin{align*}
 &\quad\bigl|\frac{1}{a_{\fp}^m}(U_{\fp}^c)^{m-1} U_{\fp}^a(f)\bigr|_{\fY_{\rig}([0,\epsilon_{n+1}]\times \prod_{\fq\neq\fp}S_{\fq})}\\
&\leq p^{1+v_p(a_{\fp})-k_{\fp}(1-\epsilon_n)} |\frac{1}{a_{\fp}^{m-1}}(U_{\fp}^c)^{m-2} U_{\fp}^a(f)|_{\fY_{\rig}([0,\epsilon_{n}]\times \prod_{\fq\neq\fp}S_{\fq})}\\
&\leq p^{(m-1)(1+v_p(a_{\fp})-k_{\fp})+k_{\fp}\sum_{i=n-m+2}^n\epsilon_i}\;|\frac{1}{a_{\fp}}U_{\fp}^a(f)|
_{\fY_{\rig}([0,\epsilon_{n-m+2}]\times \prod_{\fq\neq\fp}S_{\fq})}
\end{align*}
Moreover, it follows trivially from \ref{lemma-basic-est} that
\[|\frac{1}{a_{\fp}}U_{\fp}^a(f)|_{\fY_{\rig}([0,\epsilon_{n-m+2}]\times \prod_{\fq\neq\fp}S_{\fq})}\leq p^{1+v_p(a_p)}|f|_{\fY_{\rig}([\epsilon_0,1]\times \prod_{\fq\neq\fp}S_{\fq})}.\]
Combining the these two estimations above and using  $\sum_{i=n-m+2}^n\epsilon_i<\sum_{i=2}^{\infty}\epsilon_i=\frac{1}{p^2-1}$, we get finally
\begin{align*} &\quad\bigl|\frac{1}{a_{\fp}^m}(U_{\fp}^c)^{m-1} U_{\fp}^a(f)\bigr|_{\fY_{\rig}([0,\epsilon_{n+1}]\times \prod_{\fq\neq\fp}S_{\fq})}< p^{m(1+v_{p}(a_{\fp})-k_{\fp})+k_{\fp}(1+\frac{1}{p^2-1})}|f|_{\fY_{\rig}([\epsilon_0,1]\times \prod_{\fq\neq\fp}S_{\fq})}
 \end{align*}
 for all  $1\leq m\leq n$. Thanks to the assumption $k_{\fp}>1+v_p(a_{\fp})$, the estimation above tends to $0$ as $m\ra \infty$, so there exists a constant $M>0$ independent of $n$ such that \begin{equation}\label{equ-est-6-g_n}|g_n|_{\fY_{\rig}([0,\epsilon_{n+1}]\times \prod_{\fq\neq\fp}S_{\fq})}\leq \max_{1\leq m\leq n}|\frac{1}{a_{\fp}^m}(U_{\fp}^c)^{m-1} U_{\fp}^a(f)|_{\fY_{\rig}([0,\epsilon_{n+1}]\times \prod_{\fq\neq\fp}S_{\fq})}<M.
  \end{equation}
   On the other hand, via the functional equation $U_{\fp}(f)=a_{\fp}f$, it's easy to see that
\[f-g_n=\frac{1}{a_{\fp}^n}(U_{\fp}^c)^n(f).\]
By a similar argument  with Lemma \ref{lemma-basic-est}, we get
\begin{align*}
&\quad |f-g_n|_{\fY_{\rig}([\epsilon_{n+2},\epsilon_{n+1}]\times \prod_{\fq\neq\fp}S_{\fq})}\\
&\leq p^{1+v_p(a_{\fp})-k_{\fp}(1-\epsilon_n)}|\frac{1}{a_{\fp}^{n-1}}(U_{\fp}^c)^{n-1}(f)|_{\fY_{\rig}([\epsilon_{n+1},\epsilon_{n}]\times \prod_{\fq\neq\fp}S_{\fq})} \\
&\leq p^{n(1+v_p(a_{\fp})-k_{\fp})+k_{\fp}\sum_{i=1}^n\epsilon_i}\;|f|_{\fY_{\rig}([\epsilon_2,\epsilon_1]\times \prod_{\fq\neq\fp} S_{\fq})}.\end{align*}
As $\sum_{i=1}^n\epsilon_i<\frac{p}{p^2-1}$, we see that $|f-g_n|_{\fY_{\rig}([\epsilon_{n+2},\epsilon_{n+1}]\times \prod_{\fq\neq\fp}S_{\fq})} $ tends to $0$ when $n\ra \infty$. By \eqref{equ-est-6-g_n},  up to modifying the constant $M$, we may assume that $|f|_{\fY_{\rig}([\epsilon_{n+2},\epsilon_{n+1}]\times \prod_{\fq\neq\fp}S_{\fq})}<M$, and even  $|f|_{\fY_{\rig}(\bI_{>0}\times \prod_{\fq\neq\fp}S_{\fq})}<M$. Now applying Lemma \ref{lem-glue} to $U_0=\fY_{\rig}(\bI\times \prod_{\fq\neq\fp}S_{\fq})$, $U_n=\fY_{\rig}([0,\epsilon_{n+1}]\times \prod_{\fq\neq\fp}S_{\fq})$, $V_n=U_{n+1}$, $F_n=f|_{U_0-V_n}$, $F_n'=g_n$, we get a $U_{\fp}$-eigenform  $f$ defined over $\fY_{\rig}(\bI\times \prod_{\fq\neq\fp}S_{\fq})$ extending $f$.
\end{proof}

\subsection{}  Let $\fp\in \Sigma_{2}$, and $\bB_{\fp}=\{\beta_{\fp,1},\beta_{\fp,2}\}$.  Here, the subscripts $1,2$ should be considered as elements in $\Z/2\Z$, so that if $i$ is one of them then $i+1$ denote the other. We indicate how to generalize the results of Section 5 to our situation. We start with generalizing
the $9$ strata \eqref{sect-strata} of $Y_{\kappa}$.
As in the proof of Lemma \ref{lem-tube}, we say a pair $(\varphi_{\fp},\eta_{\fp})$ of subsets of $\bB_{\fp}$ is admissible if $\varphi_{\fp}\supset \sigma(\bB_{\fp}-\eta_{\fp})$.  If $(\varphi,\eta)$ is an  admissible pair of subsets of $\bB$, then $(\varphi\cap\bB_{\fp},\eta\cap\bB_{\fp})$ is an admissible pair of subsets of $\bB_{\fp}$, and we denote it by $(\varphi,\eta)_{\fp}$. For an  admissible pair $(\varphi_{\fp},\eta_{\fp})$ of subsets of $\bB_{\fp}$, we put
\[W_{\varphi_{\fp},\eta_{\fp}}=\bigcup_{(\varphi,\eta)_{\fp}=(\varphi_{\fp},\eta_{\fp})}
W_{\varphi,\eta},\]
where $W_{\varphi,\eta}$ is the usual Goren-Kassaei stratum defined in \ref{strat-Y}. We have a similar list of strata in $Y_{\kappa}$ as in \ref{sect-strata}.
\begin{itemize}
\item{} $0$-codimensional: $W_{\bB_{\fp},\emptyset}, W_{\emptyset, \bB_{\fp}}, W_{\beta_{\fp,1},\beta_{\fp,1}}, W_{\beta_{\fp,2},\beta_{\fp,2}}.$

\item{} $1$-codimensional: $W_{\bB_{\fp},\beta_{\fp,1}}, W_{\bB_{\fp},\beta_{\fp,2}}, W_{\beta_{\fp,1},\bB_{\fp}}, W_{\beta_{\fp,2},\bB_{\fp}}$.

\item{} $2$-codimensional: $W_{\bB_{\fp},\bB_{\fp}}$.
\end{itemize}

The graph in \ref{sect-pic} generalizes to our case, so that under the valuations $(\nu_{\fp,1},\nu_{\fp,2})$ the loci of codimension $0$,  $1$ and $2$ correspond respectively to the four vertices, the four edges and the interior of the square $\bJ=[0,1]^2$. For $i\in \Z/2\Z$, we can define similarly the supergeneral locus $W_{\beta_{\fp,i},\beta_{\fp,i}}^{sg}$ of $W_{\beta_{\fp,i},\beta_{\fp,i}}$ to be the region where the $a$-number of the universal $\fp$-divisible group $\cA[\fp^{\infty}]$ equals to one, and the superspecial locus $W_{\beta_{\fp, i},\beta_{\fp, i}}^{ss}$ to be the area where the universal $\fp$-divisisible group $\cA[\fp^{\infty}]$ is superspecial. Now, Propositions \ref{prop-dyn-sg}, \ref{lem-subgp-ss} and \ref{prop-dyn-ss} generalize word by word to the tubes $]W_{\beta_{\fp,i},\beta_{\fp,i}}^{sg}[$ and $]W_{\beta_{\fp,i},\beta_{\fp,i}}^{ss}[$, because all the proofs have only used the properties on the $p$-divisible group $A[\fp^{\infty}]$.

To state the generalization of Prop \ref{prop-cont-1}, we introduce
\begin{equation*}
S^{\can}=\{(x_1,x_2)\in [0,1]^2\;|\; x_1+px_2>1\quad\text{and }x_2+px_1>1\}
\end{equation*}
So the locus  canonical at $\fp$  in \eqref{defn-p-can} is also denoted by
$\cV_{\fp}= \fY_{\rig}(\bI^{\Sigma_1}\times\bJ^{ \Sigma_2-\{\fp\}}\times S^{\can}).$

\begin{prop}\label{prop-dim-2}
Let  $\prod_{\fq\in \Sigma}S_{\fq} \subset \bI^{\Sigma_1}\times \bJ^{\Sigma_2}$ be  a closed and connected neighborhood of $\one$.
 Fix a prime ideal $\fp\in \Sigma_{2}$ as above. Let $\vk\in \Z^{\bB}$, and $f$ be  a $U_{\fp}$-eigenform defined over  $\fY_{\rig}(\prod_{\fq\in \Sigma}S_{\fq})$ with eigenvalue $a_{\fp}$. Assume that $v_p(a_{\fp})<\min\{k_{\fp,1},k_{\fp,2}\}-2$. Then  $f$ extends uniquely to a $U_{\fp}$-eigenform over $\fY_{\rig}(\prod_{\fq\neq\fp}S_{\fq}\times (S_{\fp}\cup S^{\can}\cup\bJc))$.

\end{prop}
\begin{proof} First, the same argument as in Prop. \ref{prop-cont-can} applied to the $U_{\fp}$-operator shows that $f$ extends to $\fY_{\rig}(\prod_{\fq\neq\fp}S_{\fq}\times (S_{\fp}\cup S^{\can}))$. Since Prop. \ref{prop-dyn-sg} and \ref{prop-dyn-ss} generalize naturally to our case, we just need to copy the proof of Prop. \ref{prop-cont-1} word by word. The only place that needs more justification is Step 2 of \ref{prop-cont-1}, where we used Lemma \ref{lem-tech-ss}. We can define similarly the quasi-compact open subsets  $V_{i,n}$ and $V_{i,n}(\varepsilon)$  as in \ref{sect-V_n}. In general, their global structures maybe be complicated. However, what we really need from \ref{lem-tech-ss} is the fact that  $V_{i, n}(\varepsilon)$ is a strict neighborhood of $V_{i,n}$. This is certainly true in the general case, because the universal $\fp$-divisible group $\cA[\fp^{\infty}]$   is exactly the same as in Lemma \ref{lem-tech-ss}, and the generalized  $V_{n}$ and $V_{n,\epsilon}$ are cut out by certain special local lifts of the partial Hasse invariants at $\fp$  in the same way.
\end{proof}

\subsection{Proof of Thm. \ref{thm-main}} Let $f$ be an overconvergent eigenform as in the statement of \ref{thm-main}. We choose a closed and connected neighborhood $\prod_{\fp\in \Sigma}S_{\fp}\subset \bI^{\Sigma_1}\times \bJ^{\Sigma_2}$ of $\one $ where $f$ is defined. Applying successively Prop. \ref{prop-dim-1} and \ref{prop-dim-2} for all  $\fp\in \Sigma$, we get an eigenform $f$ defined over $\fY_{\rig}(\bI^{\Sigma_1}\times \bJc^{\Sigma_2})$. Now Theorem \ref{thm-main} follows immediately from Lemma \ref{lem-tube}, Prop. \ref{prop-exten} and Prop. \ref{prop-koecher}.

\bigskip
\setcounter{section}{0}
\setcounter{subsection}{0}
\renewcommand{\thesection}{\Alph{section}}
\renewcommand{\thesubsection}{\Alph{section}.\arabic{subsection}}
\section*{Appendix A. Some results on the extension of sections in rigid geometry}
\addtocounter{section}{1}

The purpose of this appendix to summarize some general results on the extension or gluing of sections in rigid geometry. The basic ideas of the proofs  are already contained in the work of  Kassaei \cite{Ks} and Pilloni \cite{Pi}.

\subsection{}
Let $\cO_K$ be a complete discrete valuation ring with uniformizer $\pi$, $k=\cO_K/(\pi)$, and $K$ be the fraction field of $\cO_K$. We say a topological $\cO_K$-algebra $A$ is \emph{admissible} if it is topologically of finite type and flat over $\cO_K$, i.e. $A$ is a $\pi$-torsion free quotient algebra of the usual Tate algebra of convergent power series in several variables over $\cO_K$. An \emph{admissible formal scheme} over $\cO_K$ is a quasi-compact and separated formal scheme over $\Spf(\cO_K)$ which is locally the formal spectrum of an admissible $\cO_K$-algebra. An affine formal scheme $\Spf(A)$ is admissible over $\cO_K$ if and only if $A$ is an admissible topological $\cO_K$-algebra.

Let $X$ be an admissible formal scheme over $\cO_K$. According to Raynaud, we can associate to $X$ a quasi-compact rigid analytic space $X_K$ over $K$, called the rigid generic fiber of $X$. Similarly, to each coherent sheaf $\cF$ over $X$, we can attach a coherent rigid analytic sheaf $\cF_{K}$ over $X_K$. We denote by $X_0$ the special fiber of $X$, and we have a natural specialization map $\spe: X_K\ra X_0$ of topological spaces (cf. \cite{Be} Chap. 1). If $U_0$ is a locally closed subset of $X_0$, we put $]U_0[\,=\spe^{-1}(U_0)$, and call it the tube over $U_0$ in $X_{K}$. If $U\subset X$ is the open formal subscheme with special fiber $U_0$, then we have $]U_0[=U_K$.

\begin{lemma}[cf. \cite{Pi}, Lemme 7.1]\label{lem-ext-aff} Let $X=\Spf(A)$ be an admissible affine formal scheme over $\cO_K$, $U_0\subset X_0$ be an open subset of its specail fiber, $\cF$ be a coherent $\cO_X$-module flat over $\cO_K$, and $\cF_0=\cF/\pi\cF$. If the restriction map $H^0(X_0,\cF_0)\ra H^0(U_0,\cF_0)$ is an isomorphism, then the restriction map
\[H^0(X_K,\cF_{K})\ra H^0(]U_0[,\cF_{K})\]
is also an isomorphism.
\end{lemma}
\begin{proof}
Let $U$ be the formal open subscheme of $X$ corresponding to $U_0$. The the exact sequence $0\ra \cF\xra{\times \pi}\cF\ra \cF_0\ra 0 $ of sheaves on $X$ induces a commutative diagram
\[\xymatrix{H^0(X,\cF)\ar[d]^{\iota}\ar@{->>}[rr]&&H^0(X_0,\cF_0)\ar[d]^{\sim}\\
H^0(U,\cF)\ar@{->>}[r]&H^0(U,\cF)/\pi H^0(U,\cF)\ar@{^(->}[r] &H^0(U_0,\cF_0),}\]
where the horizontal maps are natural reductions, and the vertical arrows are restrictions.  Since the right vertical arrow is an isomorphism and $H^0(U,\cF)$ has no $\pi$-torsion, an easy diagram chase shows that $\iota$ is injective. On the other hand, since $H^0(U,\cF)/\pi H^0(U,\cF)$ is a finite over $A/\pi A$, Nakayama lemma implies that $H^0(U,\cF)$ is a finite $A$-module, and the image of $\iota$ generates $H^0(U,\cF)$. Hence, the restriction map $\iota$ is an isomorphism. The Lemma follows immediately by inverting $\pi$.

\end{proof}

\begin{lemma}\label{lem-dense}
Let $X$ be an admissible formal scheme over $\cO_K$, $U_0\subset X_0$ be an open dense subscheme of the special fiber, and $\cF$ be a coherent $\cO_X$-module flat over $\cO_K$. Then the restriction map
\[H^0(X_K,\cF_K)\ra H^0(]U_0[,\cF_K)\]
is injective.
\end{lemma}
\begin{proof}
Let $U\subset X$ be the open formal subscheme corresponding to $U_0$. Since $X$ is quasi-compact and separated, we have a natural isomorphism $H^0(X_K,\cF_K)\simeq H^0(X,\cF)\otimes_{\cO_K}K$; and similar statement holds with $X$ replaced by $U$. Hence, it suffices to prove that the restriction map $H^0(X,\cF)\ra H^0(U,\cF)$ is injective. Let $X=\bigcup_{i\in I}V_i$ be a finite open covering of $X$ by affine admissible formal schemes over $\cO_K$. Since $H^0(X,\cF)\ra \prod_{i\in I}H^0(V_i,\cF)$ is injective, we are reduced to proving that $H^0(V_i,\cF)\ra H^0(U\cap V_i,\cF)$ is injective for all $i\in I$. The density of $U_0$ implies that
$$H^0(V_{i},\cF)\otimes_{\cO_K}k =H^0(V_{i,0},\cF/\pi\cF)\ra H^0(U_0\cap V_{i,0}, \cF/\pi\cF)$$ is injective. Let $f\in H^0(V_{i},\cF)$ such that $f|_{U\cap V_i}=0$. Then there exists $f_1\in H^0(V_{i},\cF)$ with $f=\pi f_1$. Since $\cF$ has no $\pi$-torsion, we deduce that $f_1|_{U\cap V_i}=0$. Repeating this process, we see that $f=0$, because $H^0(V_i,\cF)$ is $\pi$-adically separated.
\end{proof}
In general, if $X_0$ is a locally noetherian scheme and $\cF$ is a coherent sheaf  on $X_0$, we denote by $\dep_{X_0}(\cF_x)$ the depth of $\cF_x$ as an $\cO_{X_0,x}$-module for any $x\in X_0$.

\begin{prop}\label{prop-ext-1}
Let $X$ be an admissible  formal scheme over $\cO_K$, $Y_0\subset X_0$ be a closed subscheme, $\cF$ be a coherent $\cO_X$-module on $X$ flat over $\cO_K$, $\cF_0=\cF/\pi\cF$. Assume that for any $x\in Y_0$, we have $\dep_{X_0}(\cF_{0,x})\geq 2$. Then the restriction map
\[\iota: H^0(X_K,\cF_K)\ra H^0(]X_0-Y_0[, \cF_K)\]
is an isomorphism.
\end{prop}

\begin{proof}
First, we have  $\dim(\cO_{X_0,x})\geq \dep_{X_0}(\cF_{0,x})\geq 2$ for all $x\in Y_0$ by \cite[Thm. 17.2]{Ma}, i.e. $Y_0$ has at least codimension $2$ in $X_0$. In particular, $X_0-Y_0$ is dense in $X_0$.  The morphism $\iota$ is therefore injective by Lemma \ref{lem-dense}. Let $X=\bigcup_{i\in I}U_i$ be a finite covering of $X$ by affine open admissible formal subschemes. By \cite[II 3.5]{Gr}, the natural map $H^0(U_{i,0},\cF_{0})\ra H^0(U_{i,0}- Y_0,\cF_0)$ is an isomorphism. It results  from Lemma \ref{lem-ext-aff} that
$$H^0(U_{i,K},\cF_K)\ra H^0(]U_{i,0}-Y_0[,\cF_K)$$
 is an isomorphism for all $i\in I$.  Let $f\in H^0(]X_0-Y_0[, \cF_K)$, and $f_i$ be its restriction to $]U_{i,0}-Y_0[$. By the previous discussion, $f_i$ extends naturally to $U_{i,K}$. We have to show that the $f_i|_{U_{i,K}\cap U_{j,K}}=f_{j}|_{U_{i,K}\cap U_{j,K}}$ for $i,j\in I$. As  $U_{i,0}\cap U_{j,0}-Y_0$ is dense in $U_{i,0}\cap U_{j,0}$ and $f_i|_{]U_{i,0}\cap U_{j,0}-Y_0[}=f_{j}|_{]U_{i,0}\cap U_{j,0}-Y_0[}$, we  conclude by Lemma \ref{lem-dense}.

\end{proof}

Recall that we say a noetherian local ring $A$ satisfy the condition $S_2$, if $\dep(A)\geq \min\{2,\dim(A)\}$. For instance, all the normal rings or Cohen-Macaulay noetherian local rings satisfy the condition $S_2$.

\begin{cor}\label{cor-ext-1}
Let $X$ be an admissible formal scheme over $\cO_K$, and $\cF$ be an $\cO_X$-module locally free of finite type. Let $Y_0\subset X_0$ be a closed subscheme of codimension $\geq 2$ such that  for any $x\in Y_0$, the local ring $\cO_{X_0,x}$ satisfy the condition $S_2$. Then the restriction map
\[H^0(X_K,\cF_K)\ra H^0(]X_0-Y_0[,\cF_K)\]
is an isomorphism.
\end{cor}
\begin{proof}
The condition $S_2$ implies that $\dep_{X_0}(\cF_{0,x})=\dep_{X_0}(\cO_{X_0,x})\geq 2$ for $x\in Y_0$.  The Corollary follows immediately form the Proposition.
\end{proof}

\subsection{Norms} Let  $|\cdot|: K\ra \R_{\geq 0}$ be a non-archimedean absolute value on $K$ with valuation ring $\cO_K$.  We fix an algebraic closure $\Kb$ of $K$. Note that $|\cdot|$ extends uniquely to $\Kb$. If $L$ is a finite extension of $K$, we always consider $L$ as a subfield of $\Kb$ by choosing a $K$-embedding, and we denote by $\cO_L$ the ring of integers in $L$. The following treatment on norms follows \cite[5.3]{Pi}.

Let $X$ be an admissible formal scheme over $\cO_K$, and $\cF$ be an $\cO_X$-module locally free of finite type. Let $x$ be a rigid point of $X_K$ with values in a finite extension $L/K$. Then $x$ comes from a morphism $\tilde{x}:\Spf(\cO_L)\ra X$ of formal schemes over $\cO_K$. Let $(e_i)_{i\in I}$ be a basis for the finite free $\cO_L$-module $\tilde{x}^*(\cF)$. We have a natural identification $x^*(\cF_K)=\tilde{x}^*(\cF)\otimes_{\cO_K}K$. For $f=\sum_{i\in I}a_i\, e_i\in x^*(\cF_K)$ with $a_i\in L$, we put $|f|=\max_{i\in I}\{|a_i|\}$. It's easy to check that this definition does not depend on the choice of the basis $(e_i)_{i\in I}$.
Let $U\subset X_K$ be an admissible open subset, and $f\in H^0(U,\cF_K)$. We put $|f|_{U}=\sup_{x\in U(\Kb)} |x^*(f)|$. Note that it's possible that $|f|_{U}=+\infty$  if $U$ is not quasi-compact. If $U$ is quasi-compact and reduced, it follows  from the maximality principle that $|f|<+\infty$, and we have $|f|=0$ if and only if $f=0$. Therefore, in this case, $H^0(U,\cF_K)$ is a Banach space over $K$.

\begin{rem}\label{rem-ext}
In  Corollary \ref{cor-ext-1}, if $f$ is  a section of $\cF_{K}$ over $]X_0-Y_0[$ with $|f|_{]X_0-Y_0[}\leq C$ for some  $C\geq 0$, then it's easy to see from the proof that its natural extension to $X_K$ still satisfy $|f|_{X_K}\leq C$.
\end{rem}

This following lemma is a variant of \cite[Lemma 2.3]{Ks}, and the proof is of also quite similar.

\begin{lemma}\label{lem-glue} Let $X$ be an  admissible formal scheme over $\cO_K$ such that its rigid generic fiber $X_{K}$ is smooth, and $\cF$ be an $\cO_{X}$-module locally free of finite type. Let $U_0$ be a quasi-compact  admissible open subset of $X_K$, and
\begin{align*}
&U_0\supset U_1\supset U_2\supset\cdots\supset U_n\supset \cdots\\
 &U_0\supset V_1\supset V_2\supset\cdots\supset V_n\supset \cdots
 \end{align*}
be two infinite sequences of quasi-compact admissible open subsets such that $U_n$ is a strict neighborhood of $V_n$. Assume that  no connected component of $U_0$ is contained in  $\bigcap_{n\geq 1}V_n$. Let $F_n\in H^0(U_0-V_n,\cF_{\rig})$ and $F_n'\in H^0(U_n,\cF_{\rig})$ satisfying the following condition.
\begin{itemize}
\item[(a)]{} $F_n|_{U_0-V_{n-1}}=F_{n-1}$ for $n\geq 2$;

\item[(b)]{}
$|F_n-F_n'|_{U_n-V_n}$ tends to $0$ as $ n\ra \infty$;

\item[(c)]{} there exists a constant $C>0$  such that  $|F_n|_{U_0-V_n}, |F'_n|_{U_n}\leq C$ for all $n\geq 1$;

\end{itemize}
 Then there exists a unique section $F\in H^0(U_0,\cF_{K})$ such that $F|_{U_0-V_n}=F_n$ and $|F|_{U_0}\leq C$.
\end{lemma}

\begin{proof} By assumption, each connected component of $U_0$ has a non-empty intersection with $U_0-V_n$ for $n$ sufficiently large. The uniqueness of $F$ with $F|_{U_0-V_n}=F_n$ follows from the principle of analytic continuation in rigid geometry \cite[0.1.13]{Be}.  We note also that the estimation $|F|_{U_0}\leq C$ is trivial. It remains to prove the existence of $F$. This is a local problem for the Grothendieck topology on $X_K$.  By  choosing a finite open affinoid covering of $X_{K}$ over which $\cF_{K}$ is trivial, we reduce to the case where $U_0$ is  affinoid and $\cF_K=\cO_{X_K}$. Let $\tilde{\cO}_{X_K}$ be the subsheaf of $\cO_{X_K}$ consisting of sections with norm at most $1$. Up to renormalization, we may assume that $F_n, F_n'$ are sections of $\tilde{\cO}_{X_K}$ and
$$|F_n-F_n'|_{U_n-V_n}\leq {|\pi|^n}.$$
Consider the quotient sheaf $\tilde{\cO}_{X_K}/\pi^n\tilde{\cO}_{X_K}$. Then $F_n, F_n'$ glue together to a section $\overline{G_n}\in H^0(X_K, \tilde{\cO}_{X_K}/\pi^n\tilde{\cO}_{X_K})$ for each $n\geq 1$. Note that the exact sequences
$$0\ra \tilde{\cO}_{X_K}\xra{\times \pi^n}\tilde{\cO}_{X_K}\ra \tilde{\cO}_{X_i}/\pi^n\tilde{\cO}_{X_K}\ra 0$$
 induce an exact sequence of cohomology groups
\[0\ra H^0(U_0,\tilde{\cO}_{X_K})\ra \varprojlim_n H^0(U_0,\tilde{\cO}_{X_K}/\pi^n\tilde{\cO}_{X_K})\ra H^1(U_0, \tilde{\cO}_{X_K}).\]
By a result of Bartenwerfer \cite[Thm. 2]{Ba} (which uses the smoothness of $X_K$), there exists $c\in \cO_K$ with $c\neq 0$ such that $cH^1(U_0, \tilde{\cO}_{X_K})=0$. Therefore, the sections $\{c\overline{G}_n; n\geq 1\}$ come from a certain $G\in H^0(U_0,\tilde{\cO}_{X_K})$. Then we can take $F=\frac{G}{c}$.
\end{proof}

\setcounter{subsection}{0}
\section*{Appendix B. Zink's theory on Dieudonn\'e windows and canonical local coordinates at superspecial points}
\addtocounter{section}{1}
In this appendix, we prove some results needed in   Lemma \ref{lem-tech-ss}. We fix a prime number $p>0$.

\subsection{} Recall first the definition of windows in \cite{zink} and the extension to $p$-adic complete rings in \cite{K09}. Let $R$ be a $p$-adically complete and separated ring. A \emph{frame for $R$}, denoted by $(S,J,\varphi)$ or simply by $S$, is a surjective ring homomorphism $S\ra R$ with kernel $J$, where
\begin{enumerate}
\item $S$ is a $p$-adic ring flat over $\Z_p$, equipped with an endomorphism $\varphi$ lifting the Frobenius on $S/pS$.

\item $J$ is an ideal equipped with divided powers compatible the natural divided power structure on $pS$.

\end{enumerate}

For  a frame $(S,J,\varphi)$ for $R$, a \emph{Dieudonn\'e $S$-window} over $R$ is a finitely generated projective $S$-module $\cM$ together with the following data:
\begin{enumerate}
\item a submodule $\Fil^1\cM$ containing $J\cM$ such that $\cM/\Fil^1\cM$ is a projective $R$-module;

\item a $\varphi$-linear map $\varphi: \cM\ra \cM$ such that $\varphi(\Fil^1\cM)\subset p\cM$ and $\cM$ is generated  over $S$ by $\varphi(\cM)$ and $\varphi/p(\Fil^1\cM)$ .

\end{enumerate}

It's easy to see from (2) that $1\otimes \varphi$ is injective and $p\cM\subset (1\otimes \varphi)(\varphi^*\cM)$. So there exists a unique morphism $\psi: \cM\ra \varphi^*(\cM)$ of $S$-modules such that the composite
$\cM\xra{\psi}\varphi^*(\cM)\xra{1\otimes \varphi}\cM$
is the multiplication by $p$. We say a Dieudonn\'e $S$-window $(\cM,\Fil^1\cM,\varphi)$ is an \emph{$S$-window} if the image of
$$\varphi^{n-1*}(\psi)\circ\cdots\circ\varphi^*(\psi)\circ\psi:\cM\ra \varphi^{n*}(\cM)$$
is contained in $(p,J)\varphi^{n*}\cM$ for sufficiently large $n$.

\subsection{} Let $R$ be a $p$-adic complete and separated ring. A divided power surjection over $R$ is a surjective ring homomorphism $S'\ra R'$ with kernel $J'$, where  $R'$ is  a $R$-algebra and $J'$ is equipped with a divided power structure and consists of elements topologically nilpotent in the $p$-adic topology.  Let $G$ be a $p$-divisible group over $R$, and put $G_0=G\otimes_R(R/p)$. By \cite{BBM}, we can associate contravariantly with $G$ a crystal $\bD(G)$ over the big crystalline site of $R$. That is, to  each divided power surjection $S'\ra R'$ over $R$, we associate a finite locally free $S'$-module $\bD(G)(S'\ra R')$, such that  for a morphism of divided power surjections over $R$
\[\xymatrix{S'\ar[r]\ar[d]&S''\ar[d]\\
R'\ar[r]&R'',}\]
we have $\bD(G)(S''\ra R'')=S''\otimes_{S'}\bD(G)(S'\ra R')$.

Now let $(S,J,\varphi)$ be a frame for $R$. We  put $\cM(G)=\bD(G)(S\ra R)$. By \cite{BBM}, we have a canonical isomorphism of $S$-modules $\cM(G)=\bD(G_0)(S\ra R/p)$.  Since the crystal $\bD(G)$ commute with base change, we deduce from the  morphism of divided power surjections
\[\xymatrix{
S\ar[rr]^{\varphi}\ar[d]&&S\ar[d]\\
R/p\ar[rr]^{\mathrm{Frob}_{R/p}}&&R/p
}
\]
a canonical isomorphism
\[\bD(G_0^{(p)})(S\ra R/p)\simeq S\otimes_{\varphi,S}\bD(G_0)(S\ra R/p)=\varphi^*\cM(G).
\]
Therefore, the Frobenius homomorphism $F_{G_0}:G_0\ra G_0^{(p)}$ and the Verschiebung $V_{G_0}:G^{(p)}\ra G_0$ induce respectively homomorphisms of $S$-modules
\[1\otimes \varphi:\varphi^*\cM(G)\ra \cM(G)\quad\text{and}\quad \psi: \cM(G)\ra \varphi^*\cM(G).\]
We have $(1\otimes\varphi)\circ \psi=p$ since $V_{G_0}\circ F_{G_0}=p_{G_0}$. Note that $R\otimes_{S}\cM(G)=\bD(G)(R\xra{\id_R} R)$ by the base change property of the crystal $\bD(G)$. Let $\omega_{G}$ denote the module of invariant differentials of $G$ relative to $R$, and $\Lie(G^\vee)$ be the Lie algebra of the dual of $G$. By \cite{BBM}, we have a Hodge filtration
\[0\ra \omega_{G}\ra R\otimes_{S} \cM(G)\ra \Lie(G^\vee)\ra 0.\]
We define $\Fil^1\cM(G)$ to be the inverse image of $\omega_{G}$ in $\cM(G)$. Then we claim that $\varphi(\Fil^1\cM(G))\subset p\cM(G)$ and $\cM(G)$ is generated by  $\varphi(\cM(G))$ and $\varphi/p(\Fil^1\cM(G))$. Indeed, if $R$ is an algebraically closed field $k$ of characteristic $p$ and $S=W(k)$, this is well known in the classical Dieudonn\'e theory. Since the formation of the crystal $\cM(G)$ commutes with arbitrary base change, the general case of claim follows from this special case. Therefore, we get a Dieudonn\'e $S$-window $(\cM(G),\Fil^1\cM(G),\varphi)$. Since $\psi$ is induced by $V_{G_0}$, it's easy to see that $(\cM(G),\Fil^1\cM(G),\varphi)$ is an $S$-window, i.e. the extra nilpotent condition on  $\psi$ is verified,  if and only if $G$ has no multiplicative part.

\begin{thm}[\cite{zink}, Thm. 4]\label{thm-zink} Assume $R$ is excellent. Then the  contravariant functor
\[G\mapsto (\cM(G),\Fil^1\cM(G),\varphi)\]
constructed above
induces an anti-equivalence between the category of $p$-divisible groups over $R$ without multiplicative part and the category of $S$-windows over $R$.
\end{thm}

We point out that, since we have used contravariant Dieudonn\'e theory, the $p$-divisible group corresponding to an $S$-window $(\cM,\Fil^1\cM,\varphi)$ in our sense is the dual of the $p$-divisible group associated with $(\cM,\Fil^1\cM,\varphi)$ in the sense of Zink.

\subsection{} Let $k$ be an algebraically closed field of characteristic $p$, and $W=W(k)$. Let $g\geq 2$ be an integer, $\F_{p^g}$ be the finite field with $p^g$ elements, $\Z_{p^g}=W(\F_{p^g})$. We identify the set of embeddings of $\Z_{p^g}$ into $W(k)$ with $\Z/g\Z$. Let $R$ be a $W$-algebra. We say a $p$-divisible group $G$ over $R$ has \emph{formal real multiplication (or simply RM)} by $\Z_{p^g}$ if $G$ has dimension $g$ and height $2g$, and is equipped with action of $\Z_{p^g}$ such that $\Lie(G)$ is a  locally free $R\otimes_{\Z_p}\Z_{p^g}$-module of rank $1$.

Let $(S,J,\varphi)$ be a frame for $R$, and $G$ be a $p$-divisible group with RM by $\Z_{p^g}$ over $R$. The action of $\Z_{p^g}$ on $G$ induces a natural action of $\Z_{p^g}$ on the Dieudonn\'e $S$-window $(\cM(G),\Fil^1\cM(G),\varphi)$ such that $\cM(G)$ becomes a locally free $S\otimes_{\Z_p}\Z_{p^g}$-module of rank 2. Hence, we have canonical decompositions
$$\cM(G)=\bigoplus_{i\in \Z/g\Z}\cM(G)_i\quad \text{and}\quad\Fil^1\cM(G)=\bigoplus_{i\in \Z/g\Z}\Fil^1\cM(G)_i,$$ and $\varphi(\cM(G)_{i-1})\subset \cM(G)_i$. Note that we have canonical isomorphisms of free $R\otimes_{\Z_p}\Z_{p^g}$-modules of rank 1:
\begin{gather}
\omega_{G}=\bigoplus_{i\in \Z/g\Z}\omega_{G,i}\xra{\sim} \Fil^1\cM(G)/J\cM(G)=\bigoplus_{i\in \Z/g\Z}\Fil^1\cM(G)_i/J\cM(G)_i,\label{equ-omega-win}\\
\Lie(G^\vee)=\bigoplus_{i\in \Z/g\Z}\Lie(G^\vee)_i\xra{\sim}\cM(G)/\Fil^1\cM(G)=\bigoplus_{i\in \Z/g\Z}\cM(G)_{i}/\Fil^1\cM(G)_i.\label{equ-Lie-win}
\end{gather}
  Assume now $R$ has characteristic $p$. Then for each $i\in\Z/g\Z$, we have  a commutative diagram of $\varphi$-semi-linear maps
  \beq\label{diag-ha}\xymatrix{\cM(G)_{i-1}\ar[r]^{\varphi}\ar[d]&\cM(G)_{i}\ar[d]\\
  \Lie(G^\vee)_{i-1}\ar[r]^{\HW_{G,i}}&\Lie(G^\vee)_i,}
  \eeq
  where the vertical arrows are natural quotient maps and $\HW_{G,i}$ is the $i$-th component of the usual Hasse-Witt map on $\Lie(G^\vee)$. If $f_i$ is a basis of  $\Lie(G^\vee)_i$ over $R$ for  $i\in \Z/g\Z$, there exists $t_i\in R$ such that $\HW_{G,i}(f_{i-1})=t_{i}f_{i}$.  We call $t_i$ the \emph{$i$-th partial Hasse invariant} of $G$ (for the basis $(f_i)_{1\leq i\leq g}$).

\subsection{}\label{p-div-RM} Let $G_0$ be  a superspecial $p$-divisible group with RM by $\Z_{p^g}$ over $k$. That is, $G_0$ is isomorphic to the $p$-divisible group of  a product of supersingular elliptic curves of $k$. Then by \cite[5.5.4]{GO}, such a $G_0$ is unique up to isomorphism, and  the (contravariant) Dieudonn\'e module $M(G_0)=\oplus_{i\in \Z/g\Z}M(G_0)_i$ of $G_0$ can be explicitly described as follows: Each $M(G_0)_i$ is a free $W$-module of rank 2 with basis $e_i,f_i$, and the Frobenius is given by
\[\varphi(e_{i-1},f_{i-1})=(e_i,f_i)\begin{bmatrix}0& 1\\p&0\end{bmatrix}.\]
In particular, in the Hodge filtration $0\ra \omega_{G_0}\ra M(G_0)\otimes_{W}k\ra \Lie(G_0^\vee)\ra 0$, $\omega_{G_0}$ is generated by the image of $(e_i)_{i\in \Z/g\Z}$, and $\Lie(G_0^\vee)$ is generated by the image of $(f_i)_{i\in \Z/g\Z}$.

 By \cite[2.3.4]{GO}, the formal scheme which classifies that the deformations of $G_0$ as $p$-divisible groups with RM by $\Z_{p^g}$ is given by $\Spf(R^\univ)$ with $R^\univ=W[[T_1,\cdots, T_g]]$. We equip $R^\univ$ with an endomorphism $\varphi$ which acts on $W$ via Frobenius and sends $T_i$ to $T_i^p$. Then $(R^\univ, 0, \varphi)$ becomes a frame of $R^\univ$ itself. Let $(\cM^{\univ}, \Fil^1\cM^\univ, \varphi)$ be the  Dieudonn\'e $R^\univ$-window of the universal deformation $G^\univ$ over $R^\univ$. By [\emph{loc. cit.}] and the relation between  displays and Dieudonn\'e windows, the universal $R^\univ$-window has the following description: In the canonical decomposition $$\cM^\univ=\bigoplus_{i\in \Z/g\Z}\cM^\univ_i,$$
each $\cM^\univ_i$ is a free $R^\univ$-module of rank 2 with basis $\ee_i,\ff_i$, and  we have $\Fil^1\cM^\univ_i= R^\univ\cdot\ee_i$. The Frobenius map on $\cM^{\univ}$ is given by
 \beq\label{univ-phi}\varphi(\ee_{i-1},\ff_{i-1})=(\ee_i,\ff_i)\begin{bmatrix}0& 1\\ p &T_i\end{bmatrix},\eeq
and the morphism $\psi: \cM^{\univ}\ra \varphi^*(\cM^\univ)$ is thus given by
 \beq\label{univ-psi}\psi(\ee_i,\ff_i)=(\varphi^*\ee_{i-1},\varphi^*\ff_{i-1})\begin{bmatrix}-T_i& 1\\p &0\end{bmatrix}.\eeq
 Note that the image of $\ff_i$ in $\Lie(G^\univ)^\vee_i$ forms a basis. Therefore by \eqref{diag-ha}, the $i$-th partial Hasse invariant of $G^{\univ}\otimes_{R^{\univ}}(R^{\univ}/p)$ is just the image of $T_i$ in $R^{\univ}/p$.

 \begin{rem}\label{rem-CM} Let $\tilde{G}_0$ be the base change of $G^{\univ}$ to $W$ via the map $T_i\mapsto 0$. The $p$-divisible group $\tilde{G}_0$ has many interesting properties, and can be considered as the canonical lifting of $G_0$ to $W$. For example, if $g$ is odd, then $\tilde{G}_0$ has an action by $\Z_{p^{2g}}$ extending the RM by $\Z_{p^{g}}$. If $g$ is even, then  $\tilde{G}_{0}$ has a decomposition $\tilde{G}_0=H_{+}\times_{\Spec(W)} H_{-}$, where $H_{+}$ and $H_{-}$ are $p$-divisible groups of dimension $g/2$ and height $g$. Moreover, there are natural actions of $\Z_{p^g}$ on $H_{+}$ and $H_{-}$ satisfying the properties of ``formal complex multiplication'' by $\Z_{p^g}$.
  \end{rem}

\subsection{} Assume that $g$ is even. For any integer $m,n\geq 1$, let  $R_{m,n}=W\{t_1,\cdots, t_g\}$ be the $p$-adic completion of the polynomial ring $W[t_1,\cdots, t_g]$, $\iota_{m,n}: R^\univ\ra R_{m,n}$ be the homomorphism of $W$-algebras given by
 \[T_i\mapsto\begin{cases}
 p^mt_i& \text{for $i$ odd,}\\
 p^nt_i& \text{for $i$ even.}
 \end{cases}\]
    Similarly to $R^\univ$, we equip  $R_{m,n}$ with the endomorphism $\varphi$ that acts as Frobenius on $W$ and $\varphi(t_i)=t_i^p$. This makes $(R_{m,n},0,\varphi)$ a frame for $R_{m,n}$ itself. Let $G^{m,n}$ be the base change of $G^{\univ}$ to $R_{m,n}$.  Since $\iota_{m,n}$ is compatible with Frobenius, the Dieudonn\'e $R_{m,n}$-window $(\cM^{m,n},\Fil^1\cM^{m,n},\varphi^{m,n})$ associated to $G^{m,n}$ is just the base change of $(\cM^{\univ},\Fil^1\cM^{\univ},\varphi)$ via $\iota_{m,n}$. For $i\in \Z/g\Z$, let  $\ee_i,\ff_i$
denote the image of $\ee_i,\ff_i$ in $\cM^{m,n}$ by an obvious abuse of notation. We have $\Fil^1\cM^{m,n}=\oplus_{i\in \Z/g\Z}R_{m,n}\ee_i$ and
\[\varphi^{m,n}(\ee_{2i-1},\ff_{2i-1})=(\ee_{2i},\ff_{2i})\begin{bmatrix}0&1\\
p&p^nt_{2i}\end{bmatrix},\quad\quad
\varphi^{m,n}(\ee_{2i},\ff_{2i})=(\ee_{2i+1},\ff_{2i+1})\begin{bmatrix}0&1\\p&p^mt_{2i+1}\end{bmatrix}.\]

  The following proposition can be considered as a relative version of \ref{prop-split-kisin}(a) in a more general setting.

 \begin{prop}\label{prop-window}

 \emph{(a)} There exists two finite and flat closed group schemes $H^{m,n}_{+}\subset G^{m,n}[p^m]$ and $H^{m,n}_{-}\subset G^{m,n}[p^n]$ stable under the action of $\Z_{p^g}$, and such that we have
 \beq\label{equ-omega-n}\omega_{H^{m,n}_{+,i}}\simeq \begin{cases}0&\text{for $i$ odd}\\
R_{m,n}/p^mR_{m,n}&\text{for $i$ even},
\end{cases}\quad
\omega_{H^{m,n}_{-,i}}\simeq \begin{cases}R_{m,n}/p^n R_{m,n}&\text{for $i$ odd}\\
0&\text{for $i$ even}.\end{cases}\eeq

  \emph{(b)} Let $(\cL^{m,n}_{+},\Fil^1\cL^{m,n}_{+},\varphi_{+})$ $($resp. $(\cL^{m,n}_{-},\Fil^1\cL^{m,n}_{-},\varphi_{-})$ $)$ be the Dieudonn\'e $R_{m,n}$-window over $R^{m,n}$ associated with the quotient $G^{m,n}/H^{m,n}_{+}$ $($resp. $G^{m,n}/H^{m,n}_{-}$ $)$. Then they are actually $R_{m,n}$-windows, and $\cL^{m,n}_{+}$ $($resp. $\cL^{m,n}_{-}$ $)$ is naturally identified with the free $R_{m,n}$-submodule of $\cM^{m,n}$ genearated by $(\ee_{2i-1}, p^m\ff_{2i-1}, p^m\ee_{2i},\ff_{2i})$ $($resp. by $(p^n\ee_{2i-1},\ff_{2i-1},\ee_{2i},p^n\ff_{2i}))$ for $1\leq i\leq g/2$   with the induced $\Fil^1$ and $\varphi$-structures. In particular, we have
  \begin{align*}
&\varphi_{+}(\ee_{2i-1},p^m\ff_{2i-1})=(p^m\ee_{2i},\ff_{2i})\begin{bmatrix}0 &1\\ p& p^{m+n}t_{2i}\end{bmatrix}, \\
  &\varphi_{+}(p^m\ee_{2i},\ff_{2i})=(\ee_{2i+1},p^m\ff_{2i+1})\begin{bmatrix}0 &1\\ p& t_{2i+1}\end{bmatrix};\\
   &\varphi_{-}(p^n\ee_{2i-1},\ff_{2i-1})=(\ee_{2i},p^n\ff_{2i})\begin{bmatrix}0 &1\\ p& t_{2i}\end{bmatrix},\\
  &\varphi_{-}(\ee_{2i},p^{n}\ff_{2i})=(p^n\ee_{2i+1},\ff_{2i+1})\begin{bmatrix}0 &1\\ p& p^{m+n}t_{2i+1}\end{bmatrix}.\\
    \end{align*}

  \end{prop}

\begin{proof}
Let $\cL^{m,n}_{+}$ denote the submodule of $\cM^{m,n}$ described in (b). We  have
$$\Fil^1\cL^{m,n}_{+}=\Fil^1\cM^{m,n}\cap\cL^{m,n}_{+}=\bigoplus_{1\leq i\leq g/2}\bigl(R_{m,n}\cdot\ee_{2i-1}\oplus R_{m,n}\cdot p^m\ee_{2i}\bigr).$$ From the formulas of $\varphi_{+}$ on $\cL^{m,n}_{+}$ given above, it's easy to see that $(\cL^{m,n}_{+},\Fil^1\cL^{m,n}_{+},\varphi_{+})$ is indeed a $R^{m,n}$-window. We have to prove that it  is indeed a $R_{m,n}$-window. From the formulas of $\varphi_{+}$, it's easy to see that the morphism $\psi_{+}: \cL^{m,n}_{+}\ra \varphi^*\cL^{m,n}_{+}$ is given by
\begin{gather*}\psi_{+}(p^m\ee_{2i},\ff_{2i})=(\varphi^*\ee_{2i-1},\varphi^*p^m\ff_{2i-1})\begin{bmatrix}-p^{m+n}t_{2i}&1\\p&0\end{bmatrix},\\
\psi_{+}(\ee_{2i+1},p^m\ff_{2i+1})=(\varphi^*\ee_{2i},\varphi^*p^m\ff_{2i})\begin{bmatrix}-t_{2i+1}& 1\\p&0\end{bmatrix}.
\end{gather*}
Now it's direct to check that the image of the morphism
\[\varphi^{(2g-1)*}(\psi_{+})\circ\cdots\circ\varphi^*(\psi_{+})\circ \psi_{+}: \cL^{m,n}_{+}\lra\varphi^{2g*}\cL^{m,n}_{+} \]
is indeed contained in $p\varphi^{2g*}\cL^{m,n}_{+}$. This proves $(\cL^{m,n}_{+},\Fil^1\cL^{m,n}_{+},\varphi_{+})$ is indeed a $R^{m,n}$-window, and it thus corresponds to a quotient of $G^{m,n}$ by a certain kernel $H^{m,n}_{+}$ by Zink's theorem \ref{thm-zink}. Similar arguments apply to $(\cL^{m,n}_{-},\Fil^1\cL^{m,n}_{-},\varphi_{-})$ and $H^{m,n}_{-}$. This proves statement (b). By our construction, the $p$-divisible group $G^{m,n}/H^{m,n}_{+}$ is clearly equipped with RM by $\Z_{p^g}$.  Therefore, the finite flat closed subgroup scheme $H_{+}^{m,n}$ of $G^{m,n}$ is stable under $\Z_{p^g}$. From the exact sequence of groups over $R^{m,n}$
\[0\ra H^{m,n}_{+}\ra G^{m,n}\ra G^{m,n}/H_{+}^{m,n}\ra 0,\]
we get an exact sequence of $R^{m,n}\otimes_{\Z_p}\Z_{p^g}$-modules
\[0\ra \omega_{G^{m,n}/H^{m,n}_{+}}\ra \omega_{G^{m,n}}\ra \omega_{H^{m,n}_{+}}\ra0.\]
In view of the relations \eqref{equ-omega-win}, we have a canonical isomorphism of $R^{m,n}\otimes_{Z_p}\Z_{p^g}$-modules
\[\omega_{H^{m,n}_{+}}\simeq \Fil^1\cM^{m,n}/\Fil^1\cL^{m,n}_{+}=\bigoplus_{1\leq i\leq g/2}R_{m,n}/p^m\cdot \ee_{2i},\]
from which \eqref{equ-omega-n} for $\omega_{H^{m,n}_{+}}$ results immediately. Similarly arguments work for $\omega_{H^{m,n}_{-}}$.
\end{proof}

\begin{rem}\label{rem-window}
(a) Let $\Spf(R^{\univ})_{\rig}$ and $\Spf(R_{m,n})_{\rig}$ be the associated rigid generic fibers of the corresponding formal schemes \cite[0.2]{Be}. Then $\Spf(R^{\univ})_{\rig}$ is isomorphic to the open unit polydisc $\DD$ of dimension $g$ with parameters $T_1,\cdots, T_g$. Via the morphism of rigid spaces induced by $\iota_n$, $\Spf(R_{m,n})_{rig}$ is identified with the closed subdisc
\[\DD(m,n)=\{x\in \DD\;|\; v_p(T_{2i-1}(x))\geq m, v_p(T_{2i}(x))\geq n\;\text{for $1\leq i\leq g/2$} \}.\]
 The associated rigid $p$-divisible group $G^{m,n}_{\rig}$ is just the restriction of $G^{\univ}_{\rig}$ to $\DD(m,n)$.

(b) Let us relate the results of the Proposition above to the results in Section 3 proven via Breuil-Kisin modules.  Let $K$ be a finite extension of $W[1/p]$ with ring of integers $\cO_K$. Let $x$ be a $K$-valued rigid point of $\DD(1,1)$, i.e. $x$ comes from a morphism of formal schemes $\Spf(\cO_K)\ra \Spf(R^{\univ})$ factoring through $\Spf(R_{1,1})$. We denote by $G_{x}, H_{+,x}, H_{-,x}$ respectively  the pullbacks of $G^{1,1}$, $H^{1,1}_{+}$ and $H^{1,1}_{-}$ over $\cO_K$ via $x$. We have
\[\deg_i(H_{+,x})=\begin{cases}0& \text{if $i$ is odd,}\\ 1&\text{if $i$ is even};\end{cases}\quad\deg_i(H_{-,x})=\begin{cases}1&\text{if $i$ is odd},\\
0&\text{if $i$ is even}.\end{cases}\]
On the other hand, $G_x[p]$ is clearly a truncated Barsotti-Tate group of level $1$ with RM by $\Z_{p^g}$ over $\cO_K$  defined in \ref{defn-gp-RM} with partial Hodge heights $w_i(G)=1$ for all $i\in \Z/g\Z$.  Then the  closed subgroup schemes $H_{+,x}, H_{-,x}$  are just  the group schemes $H_{+}$ and $H_{-}$ obtained by applying Cor. \ref{cor-split-CM} to $G_{x}[p]$. Now Prop. \ref{prop-window}(b) allows us to compute  the partial Hasse invariants  of the quotients $(G_x/H_{+,x})\otimes_{\cO_K}\cO_{K}/p$ and $(G_x/H_{-,x})\otimes_{\cO_K}\cO_K/p$.

\end{rem}


\end{document}